\documentclass[11pt,a4paper]{article}
\usepackage{amsfonts,latexsym,amsmath, amssymb, amscd, amsthm}
\usepackage{mathrsfs} %,dsfont}
\usepackage{subfigure} 
\usepackage{bm}
\usepackage{appendix}
\usepackage{color}
\usepackage[top=1.2in, bottom=1.2in, left=1.25in, right=1.25in]{geometry}
\usepackage{cite}
\usepackage{cases}
\usepackage[utf8]{inputenc}
\usepackage{graphicx,float,enumerate}
\usepackage[T1]{fontenc}
\usepackage{authblk}
\numberwithin{equation}{section}%公式编号1.1开始
\usepackage[colorlinks=true,citecolor=red,linkcolor=blue,urlcolor=RubineRed,pdfpagetransition=Blinds,pdftoolbar=false,pdfmenubar=false]{hyperref}

\def\varep{\varepsilon}

\def\N{\mathbb{N}}

\def\R{\mathbb{R}}

\newcommand{\be}{\begin{equation}}
\newcommand{\ee}{\end{equation}}
\newcommand{\baa}{\begin{array}}
\newcommand{\eaa}{\end{array}}

\newtheorem{theorem}{Theorem}[section]
\newtheorem{lemma}[theorem]{Lemma}
\newtheorem{proposition}[theorem]{Proposition}

\newtheorem{definition}[theorem]{Definition}
\newtheorem{remark}[theorem]{Remark}

\allowdisplaybreaks

\title{\bf{Propagation and blocking in a two-patch reaction-diffusion model}}
\author{Fran\c cois Hamel\thanks{Aix-Marseille Univ, CNRS, I2M, Marseille, France (\texttt{francois.hamel@univ-amu.fr}, corresponding author). This work has received funding from Excellence Initiative of Aix-Marseille Universit\'e~-~A*MIDEX, a French ``Investissements d'Avenir'' programme, and from the ANR RESISTE (ANR-18-CE45-0019) project.}, \  Frithjof Lutscher\thanks{University of Ottawa, Department of Mathematics and Statistics \& Department of Biology, Ottawa, Canada (\texttt{flutsche@uottawa.ca}).} \ and \
Mingmin Zhang\thanks{Aix-Marseille Univ, CNRS, I2M, Marseille, France, and School of Mathematical Sciences, University of Science and Technology of China, Hefei, Anhui 230026, China (\texttt{mingmin.zhang.math@gmail.com}). M. Zhang acknowledges the China Scholarship Council for the two-year financial support during her study at Aix-Marseille Universit\'{e}.} \\
}

\date{}
\begin{document}
	
\maketitle
\begin{abstract} 
This paper is concerned with propagation phenomena for the solutions of the Cauchy problem associated with a two-patch one-dimensional reaction-diffusion model. It is assumed that each patch has a relatively well-defined structure which is considered as homogeneous. A coupling interface condition between the two patches is involved. We first study the spreading properties of solutions in the case when the per capita growth rate in each patch is maximal at low densities, a configuration which we call the KPP-KPP case, and which  turns out to have some analogies with the homogeneous KPP equation in the whole line. Then, in the KPP-bistable case, we provide various conditions under which  the solutions show different dynamics in the bistable patch, that is, blocking, virtual blocking (propagation with speed zero), or spreading with positive speed.  Moreover, when propagation occurs with positive speed, a global stability result is proved. Finally, the analysis in the KPP-bistable frame is extended to the bistable-bistable case.
\vskip 2mm
\noindent{\small{\it{The manuscript has no associated data.}}}
\end{abstract}

\tableofcontents

%%%%%%%%%%%%%%%%%%%%%%%%%%%%%%%%%%%%%%%%%%%%%%%%%%%%%
%%%%%%%%%%%%%%%%%%%%%%%%%%%%%%%%%%%%%%%%%%%%%%%%%%%%%

\section{Introduction}

Propagation and propagation failure are two fundamental phenomena of great importance to many fields of science. For example, signal propagation in nerve cells occurs when the medium is homogeneous but can fail when inhomogeneities are present, such as a change in cross-sectional area, junctions to several other cells, or localized regions of reduced excitability \cite{P1981,LK2000}. The mathematical framework of choice for modeling such phenomena are reaction-diffusion equations. In the simplest case, space is one-dimensional and inhomogeneities are represented as spatial changes in diffusivity or reaction terms at a single location, within a bounded region, or at periodically repeating locations. Our work here is inspired by the ecological dynamics of invasive species. When such species spread across a landscape, they encounter different habitat types, and their movement behavior as well as population dynamics may change according to landscape type. Our work is based on recent progress in modeling individual movement behaviors around interfaces where the landscape type changes \cite{ML2013} and continues the rigorous analysis of propagation phenomena in such models \cite{HLZ, SKW2015}. 

Specifically, we consider a one-dimensional infinite landscape comprised of two semi-infinite patches. We denote $(-\infty,0)$ as patch 1 and $(0,+\infty)$ as patch 2. The interface that separates the two patches occurs at $x=0$. Our model consists of a reaction-diffusion equation for the species' density on each patch and conditions that match the density and flux across the interface. We assume that each patch is homogeneous but the two patches may differ, so that the diffusion coefficients and the reaction terms (i.e.~net population growth rates) may differ.  Whereas most existing models for propagation and propagation failure assume that the population dynamics outside of a bounded region are identical, we are explicitly interested in the case where the dynamics differ, qualitatively and quantitatively, between the two patches. Hence, on each patch, the population density $\widetilde{u}=\widetilde{u}(t,x)$ satisfies an equation of the form 
$$\widetilde u_t = d_i \widetilde u_{xx} + \widetilde f_i(\widetilde u),$$
where $i=1,2$, depending on patch type. Since we want the interface to be neutral with respect to reaction dynamics (i.e.~no individuals are born or die from crossing the interface), the density flux is continuous at the interface, i.e., $d_1 \widetilde u_x(t,0^-)=d_2 \widetilde u_x(t,0^+)$. Continuity of the flux implies mass conservation in the absence 
of reaction terms. Individuals at the interface may show a preference for one or the other patch type. We denote this preference by $\alpha\in(0,1)$, where $\alpha>0.5$ indicates a preference for patch 1 and $\alpha<0.5$ for patch 2. Then the population density may be discontinuous at the interface with
$$(1-\alpha) d_1 \widetilde u(t,0^-) = \alpha d_2 \widetilde u(t,0^+).$$
Please see \cite{ML2013} for a detailed derivation of this condition from a random walk and a thorough discussion of the biological implications. (A second case exists where both diffusion constants appear under square roots \cite{ML2013}; the theory developed below applies to that case as well.)

The discontinuity of the density at $x=0$ creates some difficulties in the analysis of propagation phenomena in our equations. It turns out to be much easier to scale the equations (by setting $u(t,x)=\widetilde u(t,x)$ in patch~1, $u(t,x)=k\widetilde u(t,x)$ in patch~2 with $k=\frac{\alpha}{1-\alpha}\,\frac{d_2}{d_1}$, $f_1=\widetilde{f}_1$ and $f_2(s)=k\widetilde f_2(s/k)$) so that the density is continuous; see \cite{HLZ} for details. Hence, in the present paper, we study the following equivalent two-patch problem:
\begin{align}
\label{1.1}
\begin{cases}
u_t=d_1 u_{xx}+f_1( u),~~&t>0,\ x< 0,\\
u_t=d_2  u_{xx}+f_2( u),~~&t>0,\ x> 0,\cr
u(t,0^-)= u(t,0^+),~~&t>0,\cr
u_x(t,0^-)=\sigma u_x(t,0^+),~~&t> 0.
\end{cases}
\end{align}
Here, the density is continuous across the interface but its derivative is not. The diffusion constants are assumed positive. Parameter $\sigma=(1-\alpha)/\alpha>0$ is related to $\alpha$, the probability that an individual at the interface chooses to move to patch 1. Please see Section~\ref{sec:biology} for more biological background and some interpretation of our results. Throughout this work, we assume that the functions $f_i$ $(i=1,2)$ are of class $C^1(\R)$ and that
\be\label{hypf}
\exists\,K_i>0,\ \ f_i(0)=f_i(K_i)=0\ \hbox{ and }\ f_i\le0\hbox{ in }[K_i,+\infty).
\ee
Our analysis and results will depend on a few characteristic properties of the functions $f_i$. We distinguish between the Fisher-KPP type and the bistable type. We give precise definitions of these properties below in (\ref{hypkpp}) and (\ref{hypbistable}), respectively. 

In~\cite{HLZ}, we analyzed in full detail the well-posedness problem for a related patch model in a one-dimensional spatially periodic habitat  and also the spatial dynamics of the solution for the Cauchy problem under certain hypotheses on the reaction terms. Our goal of the present paper is to study spreading properties and propagation vs.~blocking phenomena for the solutions of this two-patch model for various combinations of the reaction terms. Specifically, we investigate:
\begin{enumerate}
\item the asymptotic spreading properties of the solutions to the Cauchy problem~\eqref{1.1} with compactly supported initial data when both reaction terms are of KPP type;
\item conditions for the solutions to the Cauchy problem~\eqref{1.1} with compactly supported initial data to be blocked or to propagate with positive or zero speed when one reaction term is of KPP type and the other of bistable type; we also study the stability of a traveling wave in the bistable patch;
\item the asymptotic dynamics when both reaction terms are of bistable type.
\end{enumerate}

Previous work on action potentials in nerve cells obtained some propagation and stability results when the reaction terms in both patches are identical and of bistable type and when the derivative is continuous at the interface, i.e., $\sigma=1$ \cite{P1981}. We also mention recent work  on a bistable equation in multiple (three or more) disjoint half-lines with a junction \cite{JM2019}: the existence of entire (defined for all times $t\in\R$) solutions is proved and blocking phenomena of entire solutions caused by the emergence of certain stationary solutions are investigated.
  
Before we state our main results, we summarize some relevant results on the classical homogeneous reaction-diffusion equation 
\begin{equation}
\label{1.2}
u_t=u_{xx}+f(u),~~t>0,~x\in\mathbb{R},
\end{equation}
where $f$ is a $C^1(\R)$ function satisfying $f(0)=f(1)=0$. This equation has been extensively studied in the mathematical, physical and biological literature since the pioneering works of Fisher \cite{F1937} and Kolmogorov, Petrovskii and Piskunov \cite{KPP1937} on population genetics. We say that $f$ is of Fisher-KPP type (or simply KPP type) if
\be\label{hypkpp}
f(0)=f(1)=0\ \hbox{ and }\ 0<f(s)\le f'(0)s\ \hbox{ for all }\ s\in (0,1).
\ee
If $f$ in (\ref{1.2}) is of KPP type,~\eqref{1.2} admits traveling front solutions $u(t,x)=\varphi_c(x\cdot e-ct)$ with $\varphi_c:\R\to(0,1)$ and $\varphi_c(-\infty)=1$, $\varphi_c(+\infty)=0$, if and only if $c\ge c^*=2\sqrt{f'(0)}$, where $e=\pm 1$ denotes the direction of propagation and $c$ is the speed. For each $c\ge c^*$, $\varphi_c$ satisfies 
\be\label{varphic1}
\varphi_c''+c\varphi_c'+f(\varphi_c)=0~\text{in}~\mathbb{R},~~\varphi_c'<0~\text{in}~\mathbb{R},~~\varphi_c(-\infty)=1,~~\varphi_c(+\infty)=0,
\ee
and it is unique up to shifts. Moreover, there holds
\begin{align}\label{phiexp}
\varphi_c(s)\mathop{\sim}_{s\to+\infty}\begin{cases}
Ae^{-\lambda_c s} & \hbox{if }c>c^*,\\
A^*s e^{-\lambda_c s} & \hbox{if }c=c^*,\\
\end{cases}
\end{align} 
where $A,\,A^*$ are positive constants and the decay rate $\lambda_c>0$ is obtained from the linearized equation $u_t=u_{xx}+f'(0)u$ and is given by $\lambda_c=\big(c-\sqrt{c^2-4f'(0)}\big)/2$. It was proved in~\cite{B1983,HNRR2013,L1985,U1978} that the front with minimal speed $c^*$ attracts, in some sense, the solutions of the Cauchy problem~\eqref{1.2} associated with nonnegative bounded nontrivial compactly supported initial data $u_0=u(0,\cdot)$. Furthermore, Aronson and Weinberger~\cite{AW2} proved that if $0\le u\le 1$ is the solution to the Cauchy problem~\eqref{1.2} with a  nontrivial compactly supported initial datum $0\le u_0 \le 1$, then $\sup_{\mathbb{R}\backslash(-ct,ct)}u(t,\cdot)\to 0$ as $t\to+\infty$  for every $c>c^*$, and $\inf_{[-ct,ct]} u(t,\cdot)\to 1$ as $t\to+\infty$ for every $c\in[0,c^*)$. We refer to these results as spreading properties. The minimal speed of traveling fronts, $c^*$, can therefore also be thought of as the asymptotic spreading speed.

In contrast, in the bistable case, defined as
\be\label{hypbistable}
f(0)\!=\!f(\theta)\!=\!f(1)\!=\!0\hbox{ for some $\theta\in(0,1)$, $f'(0)<0$, $f'(1)<0$, $f<0$ in $(0,\theta)$, $f>0$ in $(\theta,1)$},
\ee
equation~\eqref{1.2} has traveling front solutions $u(t,x)=\phi(x\cdot e-ct)$, where $\phi:\R\to(0,1)$, $\phi(-\infty)=1$, $\phi(+\infty)=0$, and $e=\pm 1$ is the direction of propagation, for a unique propagation speed $c\in\mathbb{R}$, depending only on $f$. Furthermore, the sign of $c$ equals the sign of $\int_0^1 f(s) \mathrm{d}s$ \cite{AW2,FM1977}. The profile $\phi$ satisfies~\eqref{varphic1} (with $\phi$ instead of $\varphi_c$) and is unique up to shifts. It is known that
$$\left\{\begin{aligned}
a_0 e^{-\alpha s}\le \phi(s)\le a_1 e^{-\alpha s},~~s\ge 0,\\
b_0 e^{\beta s}\le 1-\phi(s)\le b_1 e^{\beta s},~~s\le0,
\end{aligned}\right.$$
where $a_0$, $a_1$, $b_0$ and $b_1$ are some positive constants, $\alpha$ and $\beta$ are given by $\alpha=(c+\sqrt{c^2-4f'(0)})/2>0$ and $\beta=(-c+\sqrt{c^2-4f'(1)})/2>0$~\cite{FM1977}. Fronts in the bistable case are globally stable in the sense that any solution of the Cauchy pro\-blem~\eqref{1.2} with an initial datum $0\le u_0\le1$ satisfying $\liminf_{x\to-\infty}u_0(x)\!>\!\theta\!>\!\limsup_{x\to+\infty} u_0(x)$ converges to the unique bistable traveling front $\phi(x\!-\!ct\!+\!\xi)$ uniformly in~$x\in\mathbb{R}$ as $t\to+\infty$, where~$\xi$ is a real number depending only on $u_0$ and $f$ \cite{FM1977}. Stationary solutions~$u:\R\to[0,1]$ of equation~\eqref{1.2} in the bistable case~\eqref{hypbistable} are either: (a)~constant solutions (zeros of~$f$, that is, $0$, $\theta$ or~$1$); or (b)~periodic non-constant solutions; or (c)~symmetrically decreasing solutions, namely, for some $x_0\in\mathbb{R}$, $u(x)=u(2x_0-x)$ in $\R$, $u'<0$ in $(x_0,+\infty)$ and~$u(\pm\infty)=0$; or (d)~symmetrically increasing solutions, namely, for some $x_0\in\mathbb{R}$, $u(x)=u(2x_0-x)$ in~$\R$, $u'>0$ in $(x_0,+\infty)$, and $u(\pm\infty)=1$; or (e)~strictly decreasing or increasing solutions converging to~$0$ and~$1$ at~$\pm\infty$ \cite{FM1977}. Case (c) (respectively case~(d), respectively case~(e)) occurs if and only if $\int_0^1f(s)\mathrm{d}s>0$ (respectively $\int_0^1f(s)\mathrm{d}s<0$, respectively~$\int_0^1f(s)\mathrm{d}s=0$). Notice that, in the KPP case~\eqref{hypkpp}, the only stationary solutions $u:\R\to[0,1]$ of~\eqref{1.2} are the constants~$0$ and~$1$.

Much work has been devoted to extinction, blocking, and propagation results for the one-dimensional homogeneous equation~\eqref{1.2},  where extinction, blocking and propagation are understood as follows:
\begin{itemize}
\item \textit{extinction}: $u(t,x)\to 0$ as $t\to+\infty$ uniformly in $x\in\mathbb{R}$;
\item \textit{blocking $($say, in the right direction$)$}: $u(t,x)\to0$ as $x\to+\infty$ uniformly in $t\ge0$;
\item \textit{propagation}: $u(t,x)\to 1$ as $t\to+\infty$ locally uniformly in $x\in\mathbb{R}$.
\end{itemize}	
 Kanel' \cite{K1964} considered the combustion nonlinearity (i.e., $f=0$ in $[0,\theta]\cup\{1\}$ and $f>0$ in $(\theta,1)$ for some $0<\theta<1$) and showed that, for the particular family of initial data being characteristic functions of intervals (namely, $u_0=\chi_{[-L,L]}$, with $L>0$), there exist $0<L_0\le L_1$ such that extinction occurs for $L<L_0$, while propagation occurs for $L>L_1$. This result was then extended by Aronson and Weinberger~\cite{AW1} to the bistable case~\eqref{hypbistable} with $\int_0^1 f(s) \mathrm{d}s>0$ (so-called bistable unbalanced case). Zlato\v{s}~\cite{Z2006} improved these results in both cases by showing that $L_0=L_1$. Du and Matano \cite{DM2010} generalized this sharp transition result for  a wider class of one-parameter families of initial data. Moreover, they showed that the solutions to the Cauchy problem~\eqref{1.2} with nonnegative bounded and compactly supported initial data always converge to a stationary solution of~\eqref{1.2} as $t\to+\infty$ locally uniformly in $x\in\mathbb{R}$, and this limit turns out to be either a constant  or a symmetrically decreasing stationary solution of~\eqref{1.2}. Whether such a {\it sharp} criterion for extinction vs.~propagation holds in our patch model~\eqref{1.1} is a delicate issue, since there is no translation invariance due to the interface conditions at $x=0$ and since the reaction terms and diffusion coefficients may differ in general. This question will be left for future work. We however provide in the present paper for the patch problem~\eqref{1.1} a list of sufficient conditions for extinction, blocking and/or propagation with KPP and/or bistable dynamic in the two patches.

To see the difficulties in our patchy setting, let us briefly recall the standard methods used for the one-dimensional reaction-diffusion equation~\eqref{1.2}. For the investigation of the Cauchy problem~\eqref{1.2} with compactly supported initial data, reflection techniques can be effectively used to prove, among other things, the monotonicity of the solution $u(t,\cdot)$ outside any interval containing the initial support \cite{DM2010, DP2015, Z2006}. Properties of the solutions to the parabolic equation~\eqref{1.2} can also be connected with certain structures in the phase plane portrait of the ODE $u''+f(u)=0$. However, this is no longer the case for the patch model~\eqref{1.1}. Our proofs rest on comparison and PDE arguments. For instance, by estimating the behavior, for large $|x|$ and/or $t$, of the solution $u(t,x)$ of the Cauchy problem~\eqref{1.1} with compactly supported initial data and then by comparing it with the standard traveling fronts, we can retrieve the classical spreading results \cite{AW2,FM1977} in a sense (see Theorems~\ref{thm2.3},~\ref{thm_propagation-1} and~\ref{thm_propagation-1'} below). Besides, in the KPP-bistable case (i.e.,~the case where $f_1$ is KPP and $f_2$ is bistable), we provide some  sufficient conditions  under which either blocking or propagation occurs in the bistable patch. At first glance, one may anticipate  similar dynamics or features at large times for the solutions of the Cauchy problem~\eqref{1.1} as for the solutions of the scalar homogeneous equation~\eqref{1.2} in each patch, possibly with some nuances. However, that turns out to be not exactly true. We prove that the propagation phenomena in the KPP-bistable case can be remarkably different from what happens for the homogeneous bistable equation. We especially show a ``virtual blocking'' phenomenon, i.e., the solution indeed does propagate, but with speed zero. This unusual phenomenon reveals that the effect of the KPP patch on the bistable patch cannot be neglected and that~\eqref{1.1} is truly a coupled system of the reaction-diffusion equations. 

%%%%%%%%%%%%%%%%%%%%%%%%%%%%%%%%%%%%%%%%%%%%%%%%%%%%%
%%%%%%%%%%%%%%%%%%%%%%%%%%%%%%%%%%%%%%%%%%%%%%%%%%%%%

\section{Definitions and main results}\label{Sec 2}

Throughout the paper, we set
$$I_1=(-\infty,0)\ \hbox{ and }\ I_2=(0,+\infty).$$

By a solution to the Cauchy problem~\eqref{1.1} associated with a continuous bounded initial datum $u_0$, we mean a classical solution in the following sense \cite{HLZ}.

\begin{definition}\label{def1}
For $T\in(0,+\infty]$, we say that a continuous function $u:[0,T)\times\R\to\R$ is a classical solution of the Cauchy problem~\eqref{1.1} in $[0,T)\times\R$ with an initial datum~$u_0$, if $u(0,\cdot)=u_0$ in~$\R$, if~$u|_{(0,T)\times\overline{I_i}}\in C^{1;2}_{t;x}\big((0,T)\times \overline{I_i}\big)$ $($$i=1,2$$)$, and if all identities in~\eqref{1.1} are satisfied pointwise for $0<t<T$.
\end{definition}

Similarly, by a classical stationary solution of~\eqref{1.1}, we mean a continuous function $U: \mathbb{R}\to\mathbb{R}$ such that $U|_{\overline{I_i}}\in C^2(\overline{I_i})$ ($i=1,2$) and all identities in~\eqref{1.1} are satisfied pointwise, but without any dependence on $t$.

We also define super- and  subsolutions  as follows.
	
\begin{definition}\label{def2}
For $T\in(0,+\infty]$, we say that a continuous function $\overline{u}:[0,T)\times\R\to\mathbb{R}$, which is assumed to be bounded in $[0,T_0]\times\R$ for every $T_0\in(0,T)$, is a supersolution of~\eqref{1.1} in~$[0,T)\times\R$, if~$\overline{u}|_{(0,T)\times\overline{I_i}}\in  C^{1;2}_{t;x}((0,T)\times\overline{I_i})$ $($$i=1,2$$)$, if $\overline{u}_t(t,x)\ge d_i\overline{u}_{xx}(t,x)+f_i(\overline{u}(t,x))$ for all $i=1,2$, $0<t<T$ and $x\in I_i$, and if
$$\overline{u}_x(t,0^-)\ge \sigma \overline{u}_x(t,0^+)~~\text{for all}~t\in(0,T).$$
A subsolution is defined in a similar way with all the inequality signs above reversed.
\end{definition}

%%%%%%%%%%%%%%%%%%%%%%%%%%%%%%%%%%%%%%%%%%%%%%%%%%%%%

\subsection{Existence and comparison results for the Cauchy problem associated with~\eqref{1.1}}

\begin{proposition}
\label{pro1}
For any nonnegative bounded continuous function $u_0:\R\to\R$, there is a unique nonnegative bounded classical solution~$u$ of~\eqref{1.1} in~$[0,+\infty)\times\R$ with initial datum $u_0$ such that, for any $\tau>0$ and $A>0$, 
\begin{align*}
\Vert u|_{[\tau,+\infty)\times[-A,0]}\Vert_{C^{1,\gamma;2,\gamma}_{t;x}([\tau,+\infty)\times[-A,0])}+\Vert u|_{[\tau,+\infty)\times[0,A]}\Vert_{C^{1,\gamma;2,\gamma}_{t;x}([\tau,+\infty)\times[0,A])}\le C,
\end{align*}
with a positive constant $C$ depending on $\tau$, $A$, $d_{1,2}$, $f_{1,2}$, $\sigma$ and $\Vert u_0 \Vert_{L^\infty(\mathbb{R})}$, and with a universal positive constant $\gamma\in(0,1)$. Moreover, $u(t,x)>0$ for all $(t,x)\in(0,+\infty)\times\R$ if $u_0\not\equiv0$ in $\R$. Lastly, the solutions depend monotonically and continuously on the initial data, in the sense that if $u_0\le v_0$ then the corresponding solutions satisfy $u\le v$ in~$[0,+\infty)\times\R$, and for any $T\in(0,+\infty)$ the map~$u_0\mapsto u$ is continuous from $C^+(\R)\cap L^\infty(\R)$ to~$C([0,T]\times\R)\cap L^\infty([0,T]\times\R)$ equipped with the sup norms, where~$C^+(\R)$ denotes the set of nonnegative continuous functions in $\R$.
\end{proposition}

The existence in Proposition~\ref{pro1} can be proved by following the proof of~\cite[Theorem~2.2]{HLZ}. Namely, we can introduce a sequence of continuous cut-off functions $(\delta_n)_{n\ge1}$ such that $0\le\delta_n\le1$ in $\R$, $\delta_n=1$ in $[-n+1,n-1]$ and $\delta_n=0$ in $\R\setminus(-n,n)$. As in~\cite[Section~3.1, Theorem~3.2]{HLZ}, for each integer~$n\ge1$, there is a unique continuous function $u_n:[0,+\infty)\times[-n,n]\to\R$ such that $u_n|_{(0,+\infty)\times[-n,0]}\in C^{1;2}_{t;x}((0,+\infty)\times[-n,0])$, $u_n|_{(0,+\infty)\times[0,n]}\in C^{1;2}_{t;x}((0,+\infty)\times[0,n])$, and
$$\left\{\baa{ll}
(u_n)_t=d_1(u_n)_{xx}+f_1(u_n), & t>0,\ x\in[-n,0),\vspace{3pt}\\
(u_n)_t=d_2(u_n)_{xx}+f_2(u_n), & t>0,\ x\in(0,n],\vspace{3pt}\\
(u_n)_x(t,0^-)=\sigma(u_n)_x(t,0^+), & t>0,\vspace{3pt}\\
u_n(t,\pm n)=0, & t\ge0,\vspace{3pt}\\
u_n(0,x)=\delta_n(x)u_0(x), & x\in[-n,n].\eaa\right.$$
Furthermore, $0\le u_n(t,x)\le\max(K_1,K_2,\|u_0\|_{L^\infty(\R)})$ for all $(t,x)\in[0,+\infty)\times[-n,n]$, with $K_{1,2}$ as in~\eqref{hypf}. A comparison principle holds for the above truncated problem and, for each $(t,x)\in[0,+\infty)\times\R$, the sequence $(u_n(t,x))_{n\ge\max(1,|x|)}$ is nondecreasing. Next, as in~\cite[Section~3.2]{HLZ}, the following properties hold: 1)~there is $\gamma>0$ such that, for every $A>0$ and $\tau>0$, the sequences $(u_n|_{[\tau,+\infty)\times[-A,0]})_{n\ge\max(A,1)}$ and $(u_n|_{[\tau,+\infty)\times[0,A]})_{n\ge\max(A,1)}$ are bounded in $C^{1,\gamma;2,\gamma}_{t;x}([\tau,+\infty)\times[-A,0])$ and $C^{1,\gamma;2,\gamma}_{t;x}([\tau,+\infty)\times[0,A])$ respectively, by a constant depending only on $\tau$, $A$, $d_{1,2}$, $f_{1,2}$, $\sigma$ and $\|u_0\|_{L^\infty(\R)}$; 2)~the sequence $(u_n)_{n\ge1}$ converges pointwise in $[0,+\infty)\times\R$ to a nonnegative bounded classical solution $u$ of~\eqref{1.1} with initial datum~$u_0$, in the sense of Definition~\ref{def1}, and $u$ satisfies
$$0\le u(t,x)\le\max(K_1,K_2,\|u_0\|_{L^\infty(\R)})\ \hbox{ for all $(t,x)\in[0,+\infty)\times\R$};$$
3)~the solutions $u$ depend continuously on the initial data in the sense of Proposition~\ref{pro1}. Lastly, the monotonicity with respect to the initial data and the uniqueness in Proposition~\ref{pro1} are consequences of the following comparison principle stated in~\cite[Proposition~A.3]{HLZ}.

\begin{proposition}
\label{prop 1.3}
$\!\!\!${\rm{\cite{HLZ}}} For $T\in(0,+\infty]$, let $\overline{u}$ and $\underline{u}$ be, respectively, a super- and a subsolution of~\eqref{1.1} in~$[0,T)\times\R$ in the sense of Definition~$\ref{def2}$, and assume that $\overline{u}(0,\cdot)\ge\underline{u}(0,\cdot)$ in~$\R$. Then,~$\overline u\ge \underline u$ in~$[0,T)\times\R$ and, if $\overline{u}(0,\cdot)\not\equiv\underline{u}(0,\cdot)$ in $\R$, then $\overline{u}>\underline{u}$ in $(0,T)\times\R$.
\end{proposition}

In the sequel, when we speak of the solution $u$ to~\eqref{1.1} with a nonnegative bounded continuous initial datum $u_0$, we always mean the unique nonnegative bounded classical solution $u$ given in Proposition~\ref{pro1}.

%%%%%%%%%%%%%%%%%%%%%%%%%%%%%%%%%%%%%%%%%%%%%%%%%%%%%

\subsection{Propagation in the KPP-KPP case}

We here investigate the spreading properties of the solutions to the Cauchy problem~\eqref{1.1} associated with nonnegative, continuous and compactly supported initial data $u_0$ when $f_i$ ($i=1,2$) in both patches  $I_i$ satisfy, in addition to~\eqref{hypf}, the KPP assumptions, that is,
\be\label{2.1}
f_i(0)=f_i(K_i)=0,\ \ 0<f_i(s)\le f'_i(0)s~\hbox{for all }s\in(0, K_i),~f'_i(K_i)<0,\ \ f_i<0\hbox{ in }(K_i,+\infty).
\ee
We call this configuration the KPP-KPP case. Without loss of generality, we assume that $K_1\le K_2$. In particular, if each function $f_i$ satisfies~\eqref{hypf} and is positive in $(0,K_i)$ and concave in $[0,+\infty)$, then~\eqref{2.1} holds. An archetype is the logistic function $f_i(s)=s(1-s/K_i)$.

We start with a Liouville-type result, which is proved essentially with ODE tools, for the stationary problem associated with~\eqref{1.1}.

\begin{proposition}
\label{prop2.1}
Under the assumption~\eqref{2.1} with $0<K_1\le K_2$, problem~\eqref{1.1} admits a unique positive, bounded and classical stationary solution $V$. Furthermore, $V(-\infty)=K_1$, $V(+\infty)=K_2$, and~$V'>0$ in $(-\infty, 0^-]\cup[0^+,+\infty)$ if $K_1<K_2$,\footnote{The notation $V'>0$ in $(-\infty, 0^-]\cup[0^+,+\infty)$ means that the functions $V|_{(-\infty,0]}$ and $V|_{[0,+\infty)}$ have positive first-order derivatives in $(-\infty,0]$ and $[0,+\infty)$, respectively.} while $V\equiv K_1$ in $\R$ if $K_1=K_2$.
\end{proposition}

\begin{figure}[H]
\centering
\includegraphics[scale=0.4]{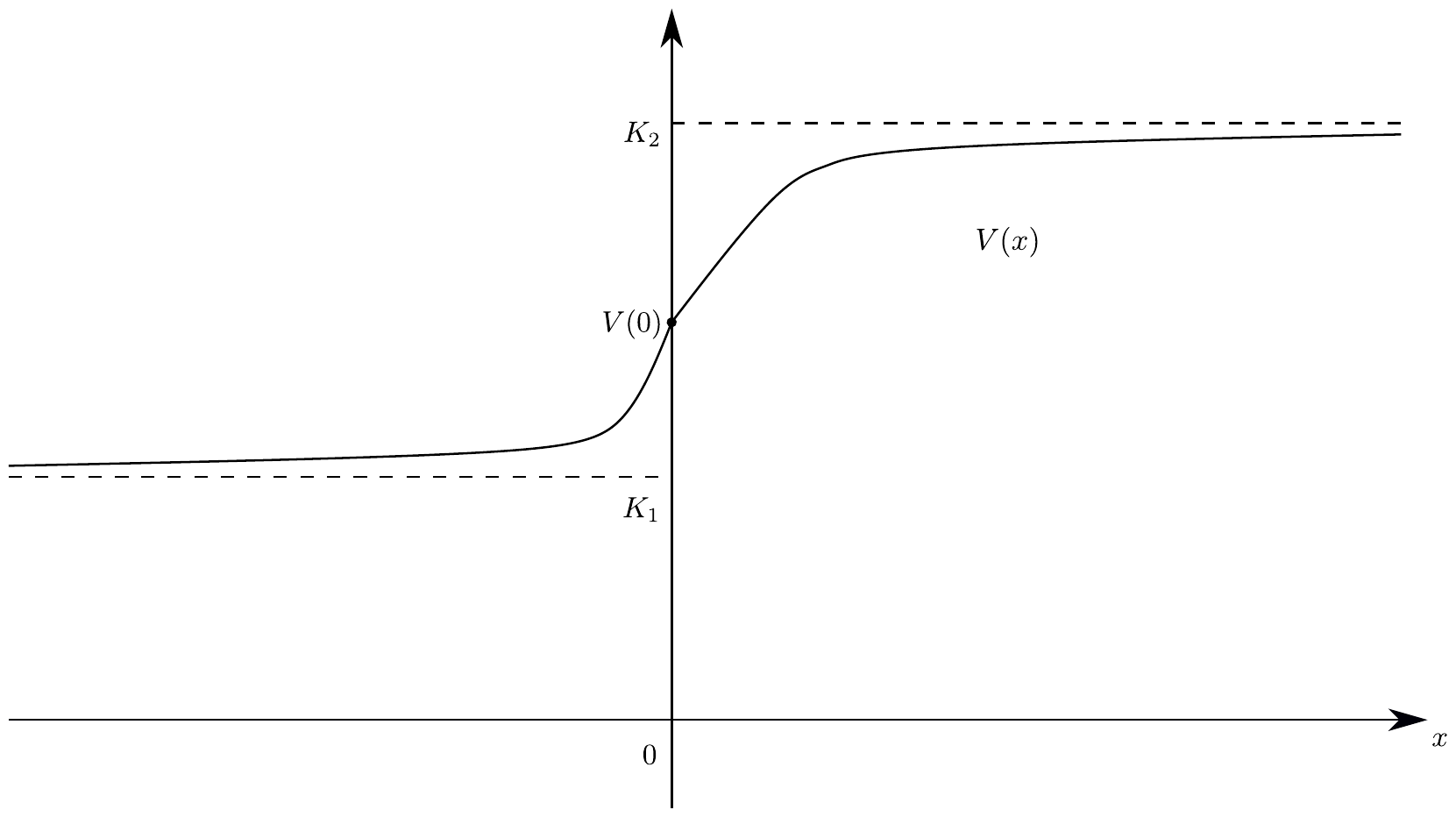}
\caption{The profile of the unique positive bounded stationary solution $V$ in the KPP-KPP case.}
\end{figure}

The assumption~\eqref{2.1} guarantees that the zero state is unstable with respect to any nontrivial perturbation, a phenomenon known from~\cite{AW2} as the hair-trigger effect for the homogeneous equation~\eqref{1.2}. It turns out that the hair-trigger effect holds good for the patch model~\eqref{1.1} in the KPP-KPP case~\eqref{2.1}, and that the solutions to~\eqref{1.1} spread with well defined spreading speeds in both directions, as the following first main result of the paper shows.

\begin{theorem}
\label{thm2.3}
%\label{prop2.2}
Assume that~\eqref{2.1} holds with $K_1\le K_2$. Then, the solution $u$ of~\eqref{1.1} with a non\-negative bounded and continuous initial datum $u_0\not\equiv 0$ satisfies:
\be\label{convV}
u(t,x)\to V(x) \ \ \text{as}~t\to+\infty,~\text{locally uniformly in}~x\in\mathbb{R},
\ee
where $V$ is the unique positive bounded classical stationary solution given in Proposition~$\ref{prop2.1}$. Furthermore, if $u_0$ is compactly supported, there exist leftward and rightward asymptotic spreading speeds, $c^*_1=2\sqrt{d_1f'_1(0)}>0$ and $c^*_2=2\sqrt{d_2f'_2(0)}>0$, respectively, such that
\be\label{spreadKPP}\left\{\baa{ll}
\displaystyle\lim_{t\to+\infty}\Big(\sup_{x\le -(c^*_1+\varep)t}u(t,x)\Big)= \lim_{t\to+\infty}\Big(\sup_{x\ge (c^*_2+\varep)t}u(t,x)\Big)=0 & \hbox{for all }\varep>0,\vspace{3pt}\\
\displaystyle\lim_{t\to+\infty}\Big(\sup_{(-c^*_1+\varep)t\le x\le (c^*_2-\varep)t} |u(t,x)- V(x)|\Big)=0 & \displaystyle\hbox{for all }0<\varep\le\min(c^*_1,c^*_2).\eaa\right.
\ee	
\end{theorem}

This theorem says that the positions of the level sets of $u(t,\cdot)$ asymptotically behave as $2\sqrt{d_1f'_1(0)}t$ in patch~$1$ and as $2\sqrt{d_2f'_2(0)}t$ in patch~$2$ at large times. It is an analogue of the standard spreading result for the solutions to homogeneous KPP equations~\eqref{1.2} (see, e.g. \cite{AW2}). This demonstrates that, in the KPP-KPP case, the spreading speeds are essentially determined by the problems obtained at the limits as $x\to\pm\infty$. The proofs actually rely on comparisons with sub- or supersolutions, which solve some approximated problems, in semi-infinite intervals away from the interface, and at large times.

It is easy to see from the proofs given in Section~\ref{Sec 3} that Proposition~$\ref{prop2.1}$ and the convergence result~\eqref{convV} in Theorem~\ref{thm2.3} still hold, while the spreading property~\eqref{spreadKPP} in Theorem~\ref{thm2.3} can be extended (though with non-explicit values of the positive spreading speeds $c^*_i$), when the KPP assumption $f_i(s)\le f'_i(0)s$ is deleted in~\eqref{2.1} (with still keeping the positivity of $f'_i(0)$). Nevertheless, for the clarity of the presentation and in order to reduce the number of hypotheses, we chose to include the KPP assumption in~\eqref{2.1}.

%%%%%%%%%%%%%%%%%%%%%%%%%%%%%%%%%%%%%%%%%%%%%%%%%%%%%%%%%

\subsection{Persistence, blocking or propagation in the KPP-bistable case}

In this section, in addition to~\eqref{hypf}, we assume that $f_1$ is of KPP type, whereas $f_2$ is of bistable type, namely:
\begin{align}
\label{f1-kpp}
f_1(0)\!=\!f_1(K_1)\!=\!0,\,0\!<\!f_1(s)\!\le\! f'_1(0)s~\hbox{for }s\!\in\!(0, K_1),\,f'_1(K_1)\!<\!0,\,f_1\!<\!0\hbox{ in }(-\infty,0)\!\cup\!(K_1,+\infty)
\end{align}
and 
\begin{equation}
\label{f2-bistable}
\left\{\baa{l}
f_2(0)=f_2(\theta)=f_2(K_2)=0~\text{for some}~\theta\in(0,K_2),\vspace{3pt}\\
f_2'(0)\!<\!0,~f_2'(\theta)\!>\!0,~f_2'(K_2)\!<\!0,~f_2\!<\!0~\text{in}~(0,\theta)\!\cup\!(K_2,+\infty),~f_2\!>\!0~\text{in}~(-\infty,0)\!\cup\!(\theta,K_2).\eaa\right.
\end{equation}
Let $\phi(x-c_2t)$ be the unique traveling wave solution connecting $K_2$ to $0$ for the equation $u_t=d_2 u_{xx}+f_2(u)$ viewed in the whole line $\R$, that is, $\phi:\R\to(0,K_2)$ obeys:
\begin{equation}
\label{2.5}
\begin{cases}
d_2\phi''+c_2\phi'+f_2(\phi)=0\hbox{ in }\R,\ \ \phi'<0~\text{in}~\mathbb{R},\\
\phi(-\infty)=K_2,~\phi(+\infty)=0,~\phi(0)=\theta,
\end{cases}
\end{equation}
where the speed $c_2$ has the same sign as $\int_{0}^{K_2}\!f_2(s) \mathrm{d}s$~\cite{FM1977}. The normalization condition $\phi(0)=\theta$ uniquely determines $\phi$.  Moreover,
\begin{equation}
\label{2.6}
\left\{\begin{aligned}
a_0 e^{-\alpha s}\le \phi(s)\le a_1 e^{-\alpha s},~~s\ge 0,\\
b_0 e^{\beta s}\le K_2-\phi(s)\le b_1 e^{\beta s},~~s\le0,
\end{aligned}\right.
\end{equation}
where $a_0$, $a_1$, $b_0$ and $b_1$ are positive constants, and $\alpha$ and $\beta$ are given by
$$\alpha=\frac{c_2+\sqrt{(c_2)^2-4d_2f_2'(0)}}{2d_2}>0,~~~\beta=\frac{-c_2+\sqrt{(c_2)^2-4d_2f_2'(K_2)}}{2d_2}>0.$$

For scalar equations of the type $u_t=u_{xx}+f(x,u)$ with bistable reaction terms $f$, solutions may be blocked (especially by the existence of certain steady states) or may propagate (see e.g.~\cite{AMO2005,BBC2016,CS2011,CG2005,DHZ2015,DLL2018,DR2018,DGM2014,E2019,E2020,HFR2010,HZ2021,LK2000,N2015,P1981,XZ1995} for various inhomogeneities and geometric configurations), whereas, for KPP reactions $f$, solutions mostly propagate (see e.g.~\cite{BHN2008,BHN2005,BHR2005,BN2012,GGN2012,HFR2010,HN2021,LZ2007,LZ2010,W2002,Z2012}). For the patch problem~\eqref{1.1} in the mixed KPP-bistable framework, we will give sufficient conditions so that blocking phenomena occur in patch 2, see Theorem~\ref{thm2.8}. We point out that the ordering between~$K_1$ and~$K_2$ is considered here in complete generality. Besides, we also prove propagation and stability results inspired by Fife and McLeod \cite{FM1977}, see Theorems~\ref{thm_propagation-1}--\ref{thm_propagation-2}. A specific ``virtual blocking'' phenomenon is also investigated, see Theorem~\ref{thm_propagation-2}. Before that, we start with the following persistence and propagation result in the KPP patch~$1$, which is the second main result of the paper.

\subsubsection*{Persistence in the KPP patch 1}

\begin{theorem}
\label{thm2.4} 	
Assume that~\eqref{f1-kpp}--\eqref{f2-bistable} hold. Let $u$ be the solution of~\eqref{1.1} with a nonnegative continuous and compactly supported initial datum $u_0\not\equiv 0$. Then, for every $\overline x\in\mathbb{R}$,
$$\inf_{x\le\bar x}\Big(\liminf_{t\to+\infty} u(t,x)\Big)>0.$$
Moreover, $u$ propagates to the left with speed $c^*_1=2\sqrt{d_1f'_1(0)}>0$ in the sense that 
$$\left\{\baa{l}
\displaystyle\forall\,\varep>0,~~\lim_{t\to+\infty}\Big(\sup_{x\le -(c^*_1+\varep)t}u(t,x)\Big)=0,\vspace{3pt}\\
\displaystyle\forall\,\varep\in(0,c^*_1),\ \forall\,\delta>0,\ \exists\,x_1\in\R,~~\limsup_{t\to+\infty}\Big(\sup_{-(c^*_1-\varep)t\le x\le x_1}|u(t,x)-K_1|\Big)<\delta.\eaa\right.$$
In particular, $\sup_{-ct\le x\le -c't}|u(t,x)-K_1|\to0$ as $t\to+\infty$ for every $0<c'\le c<c^*_1$.
\end{theorem}

An immediate consequence of Theorem~\ref{thm2.4} is that, for each~$\varep\in(0,c^*_1)$ and each map $t\mapsto\zeta(t)$ such that $\zeta(t)\to-\infty$ and $|\zeta(t)|=o(t)$ as $t\to+\infty$, it holds 
$$\lim_{t\to+\infty}\sup_{-(c^*_1-\varep)t\le x \le \zeta(t)}|u(t,x)-K_1|=0.$$
Furthermore, Theorem~\ref{thm2.4}, together with Proposition~\ref{pro1}, provides some informations on the $\omega$-limit set $\omega(u)$ of $u$ in the topology of $C^2_{loc}((-\infty,0])$ and $C^2_{loc}([0,+\infty))$ (more precisely, a function~$w$ belongs to~$\omega(u)$ if and only if  there exists a sequence $(t_k)_{k\in\mathbb{N}}$ diverging to~$+\infty$ such that $\lim_{k\to+\infty}u(t_k,\cdot)|_{[-A,0]}=w|_{[-A,0]}$ in $C^2([-A,0])$ and $\lim_{k\to+\infty}u(t_k,\cdot)|_{[0,A]}=w|_{[0,A]}$ in $C^2([0,A])$, for every $A>0$). Proposition~\ref{pro1} implies that $\omega(u)$ is not empty and Theorem~\ref{thm2.4} yields $w(-\infty)=K_1$ for any $w\in\omega(u)$. 

\subsubsection*{Stationary solutions connecting $K_1$ and $0$, or $K_1$ and $K_2$}

In the KPP-bistable case~\eqref{f1-kpp}--\eqref{f2-bistable}, because of the existence of several possible limit profiles as $x\to+\infty$, the description of the set of positive bounded and classical stationary solutions of~\eqref{1.1} is not as simple as in Proposition~\ref{prop2.1} concerned with the KPP-KPP case~\eqref{2.1}. We start with the following Proposition~\ref{prop 2.5}, which provides some necessary conditions for a stationary solution connecting $K_1$ and~$0$ to exist, whereas Proposition~\ref{prop 2.6} gives some sufficient conditions for such a solution to exist. These solutions will act as blocking barriers in the bistable patch~$2$ for the solutions of~\eqref{1.1} with initial data which are in some sense small (see part~(iv) of Theorem~\ref{thm2.8}). 

\begin{proposition}
\label{prop 2.5}
Assume that~\eqref{f1-kpp}--\eqref{f2-bistable} hold and~\eqref{1.1} admits a nonnegative classical statio\-nary solution $U$ with $U(-\infty)=K_1$ and $U(+\infty)=0$. Then $U>0$ in $\R$ and:
\begin{enumerate}[(i)]
\item if $\int_0^{K_2}f_2(s) \mathrm{d}s< 0$, then $U'<0$ in $(-\infty,0^-]\cup[0^+,+\infty)$, $0<U(0)<K_1$, and
\be\label{eqU0}
\int_{U(0)}^{K_1}f_1(s)\mathrm{d}s=-\frac{d_1\sigma^2}{d_2}\int_0^{U(0)}f_2(s)\mathrm{d}s>0;
\ee
\item if $\int_0^{K_2}f_2(s) \mathrm{d}s= 0$, then $U'<0$ in $(-\infty,0^-]\cup[0^+,+\infty)$, $0<U(0)<\min(K_1,K_2)$, and~\eqref{eqU0}~holds;
\item if $\int_0^{K_2}f_2(s) \mathrm{d}s>0$, with $\theta^*\in (\theta,K_2)$ such that $\int_0^{\theta^*} f_2(s) \mathrm{d}s=0$, then:
\begin{enumerate}
\item either $U'<0$ in $(-\infty,0^-]\cup[0^+,+\infty)$, $0<U(0)<\min(K_1,\theta^*)$, and~\eqref{eqU0} holds;
\item or $U'\ge0$ in $(-\infty,0^-]\cup[0^+,x_0)$ and $U'<0$ in $(x_0,+\infty)$ for some $x_0\ge 0$, with $U(x_0)=\max_\mathbb{R} U=\theta^*$ and $U'(x_0)=0$. Furthermore, either $x_0>0$, $U'>0$ in $(-\infty,0^-]\cup[0^+,x_0)$, $K_1<U(0)<\theta^*$ and~\eqref{eqU0} holds $($$U$ is then bump-like$)$; or $x_0=0$, $K_1=\theta^*$, $U\equiv K_1$ in~$(-\infty,0]$, and both integrals in~\eqref{eqU0} vanish.
\end{enumerate}
\end{enumerate}
\end{proposition}  

\begin{proposition}
\label{prop 2.6}
Assume that~\eqref{f1-kpp}--\eqref{f2-bistable} hold. Then~\eqref{1.1} admits a positive classical  stationary solution $U$ with $U(-\infty)=K_1$ and $U(+\infty)=0$, provided one of the following conditions holds:
\begin{enumerate}[(i)]
\item $\int_0^{K_2}f_2(s) \mathrm{d}s< 0$;
\item $\int_0^{K_2}f_2(s) \mathrm{d}s=0$ and $K_1<K_2$;
\item $\int_0^{K_2}f_2(s) \mathrm{d}s> 0$ and $K_1\le \theta^*$, with $\theta^*\in (\theta,K_2)$ such that $\int_0^{\theta^*} f_2(s) \mathrm{d}s=0$.
\end{enumerate}
\end{proposition}

In the sufficient conditions~(i)-(iii) of Proposition~\ref{prop 2.6} for the existence of a stationary solution~$U$ of~\eqref{1.1} such that $U(-\infty)=K_1$ and $U(+\infty)=0$, the parameters $d_{1,2}$ and $\sigma$ do not play any role (only the functions $f_{1,2}$ are involved). On the other hand, when $\int_0^{K_2}f_2(s)\mathrm{d}s= 0$ and $K_1\ge K_2$, or when~$\int_0^{K_2}f_2(s)\mathrm{d}s> 0$ and $K_1> \theta^*$, it turns out that stationary solutions $U$ of~\eqref{1.1} such that~$U(-\infty)=K_1$ and $U(+\infty)=0$ may not exist and the parameters~$d_{1,2}$ and $\sigma$ play crucial roles in the non-existence of $U$ (see Remark~\ref{rem41} below for further details).

The third proposition, which will be a key step in the large-time dynamics of the spreading solutions in patch~$2$, is the analogue of Proposition~\ref{prop2.1} in the present KPP-bistable framework, namely it is concerned with the stationary solutions of~\eqref{1.1} connecting $K_1$ and $K_2$. 

\begin{proposition}
\label{prop 2.7}
Assume that~\eqref{f1-kpp}--\eqref{f2-bistable} hold and that $\int_0^{K_2}f_2(s)\mathrm{d}s\ge 0$. Then problem~\eqref{1.1} has a positive monotone and classical  stationary solution $V$ such that $V(-\infty)=K_1$ and~$V(+\infty)=K_2$. Moreover, $V$ is unique if $K_1\ge\theta$.
\end{proposition}

Notice from the statements that the functions $U$ and $V$ given in Propositions~\ref{prop 2.6}--\ref{prop 2.7} can coexist.

\subsubsection*{Blocking phenomena if patch 2 has bistable dynamics}

We now turn to the investigation of blocking phenomena. If $U$ is a stationary solution of~\eqref{1.1} with~$U(-\infty)=K_1$ and $U(+\infty)=0$ and if a nonnegative bounded continuous function $u_0$ satisfies~$0\le u_0\le U$ in $\mathbb{R}$, then the comparison principle (Proposition~\ref{prop 1.3}) implies that the solution $u$ of the Cauchy problem~\eqref{1.1} with initial datum $u_0$ satisfies $0\le u(t,x)\le U(x)$ for all $(t,x)\in[0,+\infty)\times\R$, hence it is blocked in patch~$2$, that is,
\begin{equation}\label{blocking}
u(t,x)\to 0~~\text{as}~x\to+\infty, ~\text{uniformly in}~t\ge 0.
\end{equation}

In the following and much less immediate result, which is one of the main results of the paper, we provide various sufficient conditions for the solutions $u$ of~\eqref{1.1} to be blocked in the bistable patch~$2$.

\begin{theorem}
\label{thm2.8}
Assume that~\eqref{f1-kpp}--\eqref{f2-bistable} hold. Let $u$ be the solution to~\eqref{1.1} with a nonnegative continuous and compactly supported initial datum $u_0$. Then,~$u$ is blocked in patch~$2$, that is, it satisfies~\eqref{blocking}, if one of the following conditions is satisfied:
\begin{enumerate}[(i)]
\item either $\int_{0}^{K_2}f_2(s) \mathrm{d}s<0$;
\item or $\int_{0}^{K_2}f_2(s) \mathrm{d}s=0$ and $K_1<K_2$;
\item or $K_1<\theta$ and $u_0<\theta$ in $\R$;
\item or~\eqref{1.1} admits a nonnegative classical  stationary solution $U$ with $U(-\infty)=K_1$ and $U(+\infty)=0$, and $\Vert u_0 \Vert_{L^1(\mathbb{R})}\le\varep$, for some $\varep>0$ depending on $f_{1,2}$, $d_{1,2}$, $U$ and $L$, with {\rm{spt}}$(u_0)\subset[-L,L]$.\footnote{Throughout the paper, for any continuous function $\psi:\R\to\R$, we denote spt$(\psi)$ the support of $\psi$.}
\end{enumerate}
\end{theorem}

Notice that, in contrast with parts~(i) and~(ii) of Theorem~\ref{thm2.8}, which are concerned with the case~$\int_{0}^{K_2}f_2(s) \mathrm{d}s\le0$ and for which the traveling front solution $\phi(x-c_2t)$ of~\eqref{2.5} serves as a blocking barrier in patch~$2$ independently of the initial datum $u_0$, parts~(iii) and~(iv) show that blocking can also occur when $\int_{0}^{K_2}f_2(s) \mathrm{d}s>0$ provided the initial datum $u_0$ is not too large in $L^\infty$ or $L^1$ (notice also that  the existence of $U$ in part~(iv) is fulfilled when $\int_{0}^{K_2}f_2(s) \mathrm{d}s>0$ and $K_1\le\theta^*$, as follows from Proposition~\ref{prop 2.6}) . These results show some similarities with the standard results of Fife and McLeod~\cite{FM1977} concerned with the homogeneous bistable equation~\eqref{1.2}. However, for our patch problem~\eqref{1.1}, the presence of patch~1 with KPP dynamics introduces new difficulties and, in particular, the solutions~$u$ never converge to $0$ as $t\to+\infty$ even only pointwise in $\R$, thanks to Theorem~\ref{thm2.4}.

\subsubsection*{Propagation with positive or zero speed when patch $2$ has bistable dynamics} 

Finally, we turn to propagation results in patch~$2$. Our first result is motivated by the one-dimensional propagation result of Fife and McLeod \cite{FM1977}, saying that a solution of the homogeneous equation \eqref{1.2} with $f$ of bistable type \eqref{hypbistable} spreads with positive speed in both directions if its initial datum exceeds $\theta+\eta$ (with $\eta>0$) on a large enough set and if $\int_{0}^{K_2}f_2(s) \mathrm{d}s>0$. 

\begin{theorem}
\label{thm_propagation-1}
Assume that~\eqref{f1-kpp}--\eqref{f2-bistable} hold and that $\int_{0}^{K_2}f_2(s) \mathrm{d}s>0$. Let $u$ be the solution of~\eqref{1.1} with a nonnegative continuous and compactly supported initial datum $u_0\not\equiv 0$. Then, for any $\eta>0$, there is $L>0$ such that, if $u_0\ge\theta+\eta$ in an interval of size $L$ included in patch~$2$, then $u$ propagates to the right with speed $c_2$ and, more precisely, there is $\xi\in\mathbb{R}$ such that
\begin{equation}\label{convphi}
\sup_{t\ge A,\,x\ge A}|u(t,x)-\phi(x-c_2t+\xi)|\to0\ \hbox{ as }A\to+\infty,
\end{equation}
where $\phi$ is the traveling front profile given by~\eqref{2.5}.
\end{theorem}

Theorem~\ref{thm_propagation-1} assumes some conditions on~$f_2$ and~$u_0$. The following result shows that propagation can also occur independently of $u_0$, provided no  stationary solution connecting $K_1$ and $0$ exists.

\begin{theorem}
\label{thm_propagation-2}
Assume that~\eqref{f1-kpp}--\eqref{f2-bistable} hold, that $\int_0^{K_2}f_2(s)\mathrm{d}s\ge 0$, and that~\eqref{1.1} has no non\-negative classical stationary solution $U$ such that $U(-\infty)=K_1$ and $U(+\infty)=0$ $($then, necessarily, $K_1>\theta$ by Proposition~$\ref{prop 2.6}$$)$. Then the solution $u$ of~\eqref{1.1} with a nonnegative continuous and compactly supported initial datum $u_0\not\equiv 0$ propagates completely, namely,
\begin{equation}
\label{2.7}
u(t,x)\to V(x)~~\text{as}~t\to+\infty,~\text{locally uniformly in}~x\in\mathbb{R},
\end{equation}
where $V$ is the unique positive classical stationary solution of~\eqref{1.1} such that $V(-\infty)=K_1$ and $V(+\infty)=K_2$, given in Proposition~$\ref{prop 2.7}$. Furthermore, 
\begin{enumerate}[(i)]
\item if $\int_0^{K_2}f_2(s)\mathrm{d}s>0$, then $u$ spreads with speed $c_2>0$ in patch~$2$, and~\eqref{convphi} holds for some $\xi\in\R$;
\item if $\int_0^{K_2}f_2(s)\mathrm{d}s=0$, then $u$ propagates to the right with speed zero in patch $2$, in the sense that~\eqref{2.7} holds and $\sup_{x\ge ct}u(t,x)\to0$ as $t\to+\infty$ for every $c>0$.
\end{enumerate}	
\end{theorem}

Theorem~\ref{thm_propagation-2} leads to several comments. Firstly, we provide in Remark~\ref{rem41} below explicit examples of functions satisfying~\eqref{f1-kpp}--\eqref{f2-bistable} for which $\int_0^{K_2}f_2(s)\mathrm{d}s>0$ and~\eqref{1.1} has no nonnegative classical stationary solution $U$ such that $U(-\infty)=K_1$ and $U(+\infty)=K_2$, whence Theorem~\ref{thm_propagation-2} yields~\eqref{2.7} and implies that all nontrivial solutions $u$ of~\eqref{1.1} spread in patch~$2$ with speed $c_2>0$.

Secondly, in the balanced case
\be\label{hypbalanced}
\int_{0}^{K_2}f_2(s)\mathrm{d}s=0,
\ee
blocking in patch~$2$ can occur, as follows from part~(ii)--(iv) of Theorem~\ref{thm2.8}. However, in contrast to the case~$\int_{0}^{K_2}f_2(s)\mathrm{d}s<0$ (see part~(i) of Theorem~\ref{thm2.8}), blocking is not guaranteed. Indeed, if~\eqref{hypbalanced} holds, Proposition~\ref{prop 2.5}~(ii) and Theorem~\ref{thm_propagation-2}~(ii) provide some sufficient conditions for the solution $u$ of~\eqref{1.1} to propagate to the right with speed zero.\footnote{It is straightforward to see that these conditions are fulfilled, for instance, when $f_2$ is of the type $f_2(s)=\tilde{f}_2(s/\varep)$ for a fixed function $\tilde f_2$ satisfying~\eqref{f2-bistable} (with a parameter $\tilde K_2>0$) and when $\varep>0$ is small enough, while all other parameters are fixed. More precisely, under these assumptions, a nonnegative classical stationary solution $U$ of~\eqref{1.1} satisfying $U(-\infty)=K_1$ and $U(+\infty)=0$ can not exist when $\varep>0$ is small enough: the equality in~\eqref{eqU0} can not be fulfilled when $\varep>0$ is small enough, since otherwise the second integral in~\eqref{eqU0} would converge to $0$ as $\varep\to0$ because $0<U(0)<K_2=\varep\tilde K_2\to0$ as $\varep\to0$ whereas the first integral would converge to the positive constant $\int_0^{K_1}f_1(s)\mathrm{d}s>0$ as $\varep\to0$.} We give a heuristic explanation for this phenomenon. First, it follows from Proposition~\ref{prop 2.6} that $K_1\ge K_2$ under the assumptions of Theorem~\ref{thm_propagation-2}~(ii). Then, since~$u(t,x)$ converges as $t\to+\infty$ locally uniformly in $x\in\R$ to the stationary solution $V$ connecting~$K_1$ and~$K_2$, the KPP patch provides exterior energy through the interface and forces the solution~$u$ to persist in patch~$2$ and then propagate with zero speed. A similar phenomenon, called  ``virtual blocking'' or ``virtual pinning'', was previously investigated in a one-dimensional heterogeneous bistable equation~\cite{M2003} and in the mean curvature equation in two-dimensional sawtooth cylinders~\cite{LMN2013}. It is also well known that for the homogeneous bistable equation~\eqref{1.2} with $f$ satisfying~\eqref{hypbalanced}, the solution $u$ to the Cauchy problem with any nonnegative bounded compactly supported initial datum is blocked at large times and extinction occurs. In contrast, Theorem~\ref{thm_propagation-2} states that, when~\eqref{hypbalanced} is fulfilled, the solution to the patch problem~\eqref{1.1} with a compactly supported initial datum can still propagate into the bistable patch~$2$, but its level sets then move to the right with speed zero.

Thirdly, when the initial datum of the scalar homogeneous bistable equation~\eqref{1.2} is small in the $L^1(\R)$ norm, then $\Vert u(1,\cdot)\Vert_{L^\infty(\mathbb{R})}$ can be bounded from above by a constant less than $\theta$. Hence, extinction occurs and the blocking property~\eqref{blocking} holds if the initial datum is compactly supported. In our work, due to the presence of the KPP patch~$1$ in~\eqref{1.1}, the smallness of the $L^1(\R)$ norm of the initial datum is not sufficient to cause blocking in general, as follows from Theorem~\ref{thm_propagation-2}, since the conclusion of Theorem~\ref{thm_propagation-2} is independent of $u_0$. 

%%%%%%%%%%%%%%%%%%%%%%%%%%%%%%%%%%%%%%%%%%%%%%%%%%%%%%%%%

\subsection{Blocking or propagation in the bistable-bistable case}

In this section, we deal with the bistable-bistable case, namely we assume that the functions $f_i$ ($i=1,2$) are of bistable type:
\begin{equation}
\label{fi-bistable}
\left\{\baa{l}
f_i(0)=f_i(\theta_i)=f_i(K_i)=0~\text{for some}~\theta_i\in(0,K_i),\vspace{3pt}\\
f_i'(0)\!<\!0,~f'_i(\theta_i)\!>\!0,~f_i'(K_i)\!<\!0,~f_i<0~\text{in}~(0,\theta_i)\!\cup\!(K_i,+\infty),~f_i>0~\text{in}~(-\infty,0)\!\cup\!(\theta_i,K_i).\eaa\right.
\end{equation}
For each $i\in\{1,2\}$, let $\phi_i(x-c_it)$ $(i=1,2)$ be the unique traveling wave connecting $K_i$ to $0$  for the equation $u_t=d_i u_{xx}+f_i(u)$ viewed in the whole line $\R$, that is, $\phi_i:\R\to(0,K_i)$ satisfies
\begin{equation}
\label{2.5'}
\begin{cases}
d_i\phi''_i+c_i\phi'_i+f_i(\phi_i)=0\hbox{ in }\R,\ \ \phi'_i<0~\text{in}~\mathbb{R},\\
\phi_i(-\infty)=K_i,~\phi_i(+\infty)=0,~\phi_i(0)=\theta_i,
\end{cases}
\end{equation}
where the speed $c_i$ has the sign of $\int_{0}^{K_i}f_i(s) \mathrm{d}s$~\cite{FM1977} (the normalization condition $\phi_i(0)=\theta_i$ uniquely determines $\phi_i$). Moreover, each function $\phi_i$ satisfies similar exponential estimates to~\eqref{2.6}.

The first main result in the bistable-bistable case states that, when the traveling fronts $\phi_i(x-c_it)$ have negative speeds $c_i$, all solutions to~\eqref{1.1} with compactly supported initial data go to extinction:

\begin{theorem}\label{thextinction}
Assume that~\eqref{fi-bistable} holds with $\int_0^{K_i}f_i(s)\mathrm{d}s<0$, that is, $c_i<0$, for $i=1,2$. Then, for any nonnegative continuous and compactly supported initial datum $u_0$, the solution $u$ to~\eqref{1.1} goes to extinction, that is, $\|u(t,\cdot)\|_{L^\infty(\R)}\to0$ as $t\to+\infty$.
\end{theorem}

In other words, for propagation to occur, at least one of the reaction terms $f_i$ must have a nonnegative mass. By analogy with the KPP-bistable case~\eqref{f1-kpp}--\eqref{f2-bistable} and without loss of generality, we then assume in some statements that $\int_0^{K_1}\!f_1(s)\mathrm{d}s\ge0$.

In the spirit of Propositions~\ref{prop 2.5}--\ref{prop 2.7}, we then provide some necessary and/or sufficient conditions for a stationary solution connecting~$K_1$ and~$0$ (or~$K_2$) exists. Namely, the following result holds.

\begin{proposition}
\label{prop 2.5'}
Assume that~\eqref{fi-bistable} holds with $\int_0^{K_1}\!f_1(s)\mathrm{d}s\ge 0$.
\begin{enumerate}[(i)]
\item If~\eqref{1.1} admits a nonnegative classical stationary solution $U$  such that $U(-\infty)=K_1$ and $U(+\infty)=0$, then the conclusion of Proposition~$\ref{prop 2.5}$ holds true, in which $\theta^*$ is replaced by $\theta^*_2\in(\theta_2,K_2)$ such that $\int_0^{\theta^*_2}f_2(s)\mathrm{d}s=0$ when $\int_0^{K_2}f_2(s)\mathrm{d}s>0$;
\item if $\int_0^{K_1}f_1(s)\mathrm{d}s>0$, then the conclusion of Proposition~$\ref{prop 2.6}$ holds true, in which $\theta^*$ is replaced by $\theta^*_2\in(\theta_2,K_2)$ such that $\int_0^{\theta^*_2}f_2(s)\mathrm{d}s=0$ when $\int_0^{K_2}f_2(s)\mathrm{d}s>0$;
\item if $\int_0^{K_2}f_2(s)\mathrm{d}s\ge 0$, then problem~\eqref{1.1} has a positive classical stationary solution $V$ such that $V(-\infty)=K_1$ and $V(+\infty)=K_2$. Moreover, all solutions $V$ are monotone, and $V$ is unique if $K_2\ge K_1\ge\theta_2$ or $K_1\ge K_2\ge\theta_1$.
\end{enumerate}
\end{proposition}

Finally, as in Theorems~\ref{thm2.8}--\ref{thm_propagation-2} in the KPP-bistable case, the last two main theorems are concerned with blocking or propagation with positive or zero speed.\footnote{We state the blocking phenomena and propagation with zero speed only in patch~$2$, but similar statements hold true in patch~$1$ with suitable assumptions.}

\begin{theorem}
\label{thm2.8'}
Under the assumption~\eqref{fi-bistable}, the conclusion of Theorem~$\ref{thm2.8}$ holds, with $\theta$ replaced by~$\theta_2$ in~(ii).
\end{theorem}

\begin{theorem}
\label{thm_propagation-1'}
Assume that~\eqref{fi-bistable} holds.
\begin{enumerate}[(i)]
\item If $\int_0^{K_2}f_2(s)\mathrm{d}s>0$, then the conclusion of Theorem~$\ref{thm_propagation-1}$ holds with $\theta$ and $\phi$ replaced by $\theta_2$ and~$\phi_2$;
\item if $\int_0^{K_1}f_1(s)\mathrm{d}s> 0$ and $\int_0^{K_2}f_2(s)\mathrm{d}s\ge 0$, and if~\eqref{1.1} has no nonnegative classical stationary solution $U$ such that $U(-\infty)=K_1$ and $U(+\infty)=0$, then, for any $\eta>0$, there is $L>0$ such that the following holds: for any nonnegative continuous and compactly supported initial datum satisfying $u_0\ge\theta_1+\eta$ on an interval of size $L$ included in patch~$1$, the solution $u$ of~\eqref{1.1} with initial datum $u_0$ propagates completely, more precisely,
\begin{equation}
\label{2.7'}
\liminf_{t\to+\infty}u(t,x)\ge p(x),~\text{locally uniformly in}~x\in\mathbb{R},
\end{equation}
where $p$ is a positive classical stationary solution of~\eqref{1.1} such that $p(-\infty)=K_1$ and $p(+\infty)=K_2$. Moreover, if $K_2\ge K_1\ge\theta_2$ or $K_1\ge K_2\ge\theta_1$, then $u(t,x)\to V(x)$ as $t\to+\infty$ locally uniformly in $x\in\R$, where $V$ is the unique positive classical stationary solution of~\eqref{1.1} such that $V(-\infty)=K_1$ and $V(+\infty)=K_2$, given in Proposition~$\ref{prop 2.5'}$~(iii). Lastly, $u$ also propagates in patch~$1$ with speed $c_1$ and
\be\label{convphi1}
\sup_{t\ge A,\,x\le-A}|u(t,x)-\phi_1(-x-c_1t+\xi_1)|\to 0~~\text{as}~A\to+\infty,
\ee
for some $\xi_1\in\R$, while it propagates with positive or zero speed in patch~$2$ as in the conclusions~(i) and~(ii) of Theorem~$\ref{thm_propagation-2}$, with $\phi$ replaced by $\phi_2$ in~(i). 
\end{enumerate}	
\end{theorem}

%%%%%%%%%%%%%%%%%%%%%%%%%%%%%%%%%%%%%%%%%%%%%%%%%%%%%%%%%

\subsection{Biological interpretation and explanation}
\label{sec:biology}

We briefly discuss our results from an ecological point of view here. We envision a landscape of two different characteristics, say a large wooded area and an adjacent open grassland area. We assume that the movement rates of individuals are small relative to landscape scale so that we can essentially consider each landscape type as infinitely large. In the first scenario (KPP--KPP), the population has its highest growth rate at low density in both patches. While the low-density growth rates and carrying capacities may differ between the two landscape types, the population will grow in each type from low densities to the carrying capacity. When introduced locally, the population will spread in both directions, and the speed of spread will approach the famous Fisher speed $2\sqrt{d_i f_i'(0)}$ in each patch. The interface will not stop the population advance unless it is completely impermeable. This would be the special case (that we excluded from our analysis) where an individual at the interface will choose one of the two habitat types with probability one, i.e., $\alpha=0$ or $\alpha=1$. 

The second scenario (KPP--bistable) is more interesting. This time, the population dynamics change qualitatively from the highest growth rate being at low density to being at intermediate density. In ecological terms, this corresponds to a strong Allee effect and the threshold value $\theta$ is known as the Allee threshold. In this case, the interface can prevent a population that is spreading in the one habitat type (without Allee dynamic) from continuing to spread in the other type (with Allee dynamics). At first glance, it seems surprising that the conditions for propagation failure do not include parameter $\sigma$ that reflects the movement behavior at the interface. To understand the reasons, we need to understand the scaling that led to system (\ref{1.1}). The scaled reaction function $f_2$ and its unscaled counterpart, say~$\tilde f_2$, are related via
$$f_2(s) = k \tilde f_2\Big(\frac{s}{k}\Big), \qquad k = \frac{\alpha}{1-\alpha}\,\frac{d_2}{d_1},$$
see \cite{HLZ}. In particular, if $\tilde K_2$ and $\tilde{\theta}$ are the unscaled carrying capacity and Allee threshold, then $K_2 = k\tilde K_2$ and $\theta = k\tilde\theta$ are the corresponding scaled quantities. The sign of the integral that determines the sign of the speed of propagation in the homogeneous bistable equation does not change under this scaling. Hence, by choosing $k$ large enough, one can satisfy the condition $K_1<\theta$ in part~(iii) of Theorem~\ref{thm2.8}. A population that starts on a bounded set inside the KPP patch will be bounded by $K_1$ and therefore unable to spread in the Allee patch. Large values of $k$ arise when the preference for patch 1 is high ($\alpha\approx 1$) or when the diffusion rate in the Allee patch is much larger than in the KPP patch. The mechanisms in this last scenario is similar to that when a population spreads from a narrow into a wide region in two or higher dimensions \cite{CG2005,HZ2021}. As individuals diffuse broadly, their density drops below the Allee threshold and the population cannot reproduce and spread. 

A change in population dynamics from KPP to Allee effect need not be triggered by landscape properties, it can also be induced by management measures. For example, when male sterile insects are released in large enough densities, the probability of a female insect to meet a non-sterile male decreases substantially so that a mate-finding Allee effect may arise. The use of this technique to create barrier zones for insect pest spread has recently been explored by related but different means \cite{AEV2020}.

\subsubsection*{Outline of this paper} 

The rest of this paper is organized as follows. In Section~\ref{Sec 3}, we consider~\eqref{1.1} with  KPP-KPP reactions and prove Proposition~\ref{prop2.1} and Theorem~\ref{thm2.3}. Section~\ref{Sec 4} is devoted to the KPP-bistable case. We begin by proving the semi-persistence result Theorem~\ref{thm2.4} in Section~\ref{Sec 4.1}. Then, in Section~\ref{Sec 4.2}, we present the proofs of  Propositions~\ref{prop 2.5}--\ref{prop 2.7}. In Sections~\ref{Sec 4.3} and~\ref{Sec 4.4}, we collect the proofs of the main results on blocking, virtual blocking and propagation in patch~$2$, namely,  Theorems~\ref{thm2.8}--\ref{thm_propagation-2}. In Section~\ref{Sec 5}, we sketch the essential parts of the proofs in the bistable-bistable case which are different from those in the KPP-bistable case.

%%%%%%%%%%%%%%%%%%%%%%%%%%%%%%%%%%%%%%%%%%%%%%%%%%%%
%%%%%%%%%%%%%%%%%%%%%%%%%%%%%%%%%%%%%%%%%%%%%%%%%%%%

\section{The KPP-KPP case}\label{Sec 3}

This section is devoted to the analysis of~\eqref{1.1} with KPP-KPP reactions satisfying~\eqref{2.1}. We start with proving Proposition~\ref{prop2.1} for the stationary problem associated with~\eqref{1.1}.

\begin{proof}[Proof of Proposition~$\ref{prop2.1}$]
The existence of the stationary solution follows immediately from the existence of a pair of ordered sub- and supersolutions. Indeed, from~\eqref{2.1} and the condition $K_1\le K_2$, one sees that the functions equal to the constants $K_1$ and $K_2$ are, respectively, a sub- and a supersolution for~\eqref{1.1}, in the sense of Definition~\ref{def2}. Thus, from Proposition~\ref{prop 1.3}, the solution $\underline{u}$ of~\eqref{1.1} with initial datum $K_1$ satisfies $K_1\le\underline{u}(t,x)\le K_2$ for all $(t,x)\in[0,+\infty)\times\R$, hence $\underline{u}(t',x)\le\underline{u}(t'+t,x)$ for all~$t\ge0$ and for all $(t',x)\in[0,+\infty)\times\R$, that is, $\underline{u}(t,x)$ is nondecreasing with respect to~$t$ in $[0,+\infty)\times\R$. Together with Proposition~\ref{pro1}, it follows that the function $V$ defined by $V(x):=\lim_{t\to+\infty}\underline{u}(t,x)$ is a positive bounded classical stationary solution to~\eqref{1.1} such that $K_1\le V(x) \le K_2$ for all $x\in\mathbb{R}$. 
	
Next, let us turn to the uniqueness, which actually holds in the class of nonnegative nontrivial bounded classical solutions. So, consider any nonnegative bounded classical stationary solution $V$ of~\eqref{1.1}. If there is $x_0\in(-\infty,0)$ such that $V(x_0)=0$, then $V\equiv 0$ in $(-\infty,0)$ from the elliptic strong maximum principle (or the Cauchy-Lipschitz theorem), and then in $(-\infty,0]$ by continuity of~$V$. If $V>0$ in $(-\infty,0)$ and $V(0)=0$, then it follows from the Hopf lemma (or the Cauchy-Lipschitz theorem) that $V'(0^-)<0$. Similarly, if there is $x_0\in(0,+\infty)$ such that $V(x_0)=0$, then $V\equiv 0$ in~$[0,+\infty)$. If $V>0$ in $(0,+\infty)$ and $V(0)=0$, then $V'(0^+)>0$. From these observations and the fact that $V(0^-)=\sigma V'(0^+)$ with $\sigma>0$,  it follows that either $V\equiv0$ in $\R$, or $V>0$ in $\R$. 

In the sequel, we assume that $V>0$ in $\mathbb{R}$. We then claim that $\inf_{\mathbb{R}}V>0$ and
\begin{equation}
\label{3.1}
\begin{aligned}
V(-\infty)= K_1,\ \ V(+\infty)= K_2.
\end{aligned}
\end{equation}
As a matter of fact, since $f'_i(0)>0$ $(i=1,2)$, one can choose $R>0$ so large that 
\begin{align}
\label{R}
0<\frac{\pi}{2R}\le \sqrt{\frac{\min(f'_1(0),f'_2(0))}{2\max(d_1,d_2)}}.
\end{align}
Set
\begin{equation}
\label{w}
\begin{aligned}
\Psi(x)=\begin{cases}
\ \displaystyle\cos\Big(\frac{\pi}{2R}x\Big)&\text{for }x\in[-R,R],\cr
\ 0&\text{for }x\in\R\setminus[-R,R].
\end{cases}
\end{aligned}
\end{equation}
Then there exists $\widetilde\varep>0$ such that $-d_i(\varep \Psi)''<f_i(\varep \Psi)$ in $(-R,R)$ for $i=1,2$ and for all $\varep\in(0,\widetilde\varep]$, since~$f_i(0)=0$ and $f'_i(0)>0$ $(i=1,2)$. Fixing $x_0=-R-1$, one can choose $\varep_0\in(0,\widetilde \varep]$ such that
$$V>\varep_0\,\Psi(\cdot-x_0)~~\text{in}~\mathbb{R}.$$
Then, by continuity of $V$ and $\varep_0 \Psi$, there is $s_0>1$ such that
$$V>\varep_0 \Psi(\cdot-s x_0)~\text{in}~[s x_0-R,s x_0+R]~\text{for all}~s\in[1,s_0].$$
Define 
$$s^*=\sup\big\{\widetilde s>1: V>\varep_0 \Psi(\cdot-s x_0)~\text{in}~[s x_0-R,s x_0+R] ~\text{for all}~ s\in[1,\widetilde s]\big\}\ \in[s_0,+\infty].$$
We wish to prove that $s^*=+\infty$. Assume not. By the definition of $s^*$, one has $V\ge \varep_0\Psi(\cdot-s^* x_0)$ in~$[s^*x_0-R,s^*x_0+R]$ and there is $\hat x\in[s^*x_0-R,s^*x_0+R]$ such that $V(\hat x)=\varep_0 \Psi(\hat x-s^*x_0)$. Since~$V>0$ in $\mathbb{R}$ and $\Psi(\cdot-s^* x_0)=0$ at $x=s^*x_0\pm R$, one derives that $\hat x\in (s^*x_0-R,s^*x_0+R)$. The elliptic strong maximum principle then yields $V\equiv \varep_0\Psi(\cdot-s^*x_0)$ in~$(s^*x_0-R,s^*x_0+R)$ and then in~$[s^*x_0-R,s^*x_0+R]$ by continuity. This is impossible at $s^*x_0\pm R$. Consequently, $s^*=+\infty$, hence 
\begin{align*}
V> \varep_0\Psi(\cdot-s x_0)~\text{in}~[s x_0-R,s x_0+R]~\text{for all}~s\ge 1
\end{align*}
and $\inf_{x\le -R-1} V(x)\ge\varep_0$.  Similarly, one can also show that $\inf_{x\ge R+1}V(x)\ge\varep_1$ for some $\varep_1\in(0,\tilde \varep]$. Together with the continuity and positivity of $V$ in $\R$, we get $\inf_{\mathbb{R}}V>0$.
	
In order to show~\eqref{3.1}, consider now an arbitrary sequence $(x_n)_{n\in\mathbb{N}}$ in $\mathbb{R}$ diverging to  $-\infty$ as $n\to+\infty$ and define $V_n:=V(\cdot+x_n)$ in $\R$ for each $n\in\mathbb{N}$. Then, by standard elliptic estimates, the sequence $(V_n)_{n\in\mathbb{N}}$ converges as $n\to+\infty$, up to extraction of some subsequence, in $C^2_{loc}(\mathbb{R})$ to a bounded function $V_\infty$ which solves $d_1V_\infty''+f_1(V_\infty)=0$ in $\mathbb{R}$. Moreover, $\inf_{\R}V_\infty>0$. It follows that~$V_\infty\equiv K_1$ in $\R$, thanks to the hypothesis that $f_1>0$ in $(0,K_1)$ and $f_1<0$ in $(K_1,+\infty)$. That is,~$V_n\to K_1$ as~$n\to+\infty$ in $C^2_{loc}(\R)$. Since the limit does not depend on the particular sequence~$(x_n)_{n\in\mathbb{N}}$, it follows that  $V(x)\to K_1$ as $x\to-\infty$ and $V'(x)\to0$ as $x\to-\infty$. By the same argument as above and by the assumption that $f_2>0$ in $(0,K_2)$ and $f_2<0$ in $(K_2,+\infty)$, one can also derive $V(x)\to K_2$ and~$V'(x)\to0$ as $x\to+\infty$. Thus,~\eqref{3.1} is achieved.

We prove now that $V$ is monotone in $\R$. Assume first that $V$ is not monotone in $(-\infty,0)$. Then there is $x_0\in(-\infty,0)$ such that $V(x_0)$ reaches a local minimum or maximum with $V\not\equiv V(x_0)$ in~$(-\infty,0)$. On the one hand, $V'(x_0)=0$. On the other hand, by multiplying $d_1V''+f_1(V)=0$ by $V'$ and integrating over $(-\infty,x]$ for any $x\le0$, one gets that
\begin{equation}
\label{3.5}
\frac{d_1}{2}(V'(x^-))^2=\int_{V(x)}^{K_1}f_1(s) \mathrm{d}s.\footnote{The notation $x^-$ in $V'(x^-)$ is used in order to cover the case $x=0$, where $V$ is not differentiable in general. The same type of notation is used in~\eqref{3.6}, as well as in further subsequent proofs.}
\end{equation}
Remember that $f_1>0$ in $(0,K_1)$ and $f_1<0$ in $(K_1,+\infty)$, while $V>0$ in $\R$. Hence,~\eqref{3.5} yields~$V(x_0)=K_1$. By the Cauchy-Lipschitz theorem, one then has $V\equiv K_1$ in $(-\infty,0]$, a contradiction. Similarly, integrating $d_2V''+f_2(V)=0$ against $V'$ over $[x,+\infty)$ for any $x\ge0$ implies
\begin{equation}
\label{3.6}
\frac{d_2}{2}(V'(x^+))^2=\int^{K_2}_{V(x)}f_2(s)\mathrm{d}s.
\end{equation} 
One can use the same procedure to show that $V$ is monotone in $(0,+\infty)$. Consequently, $V$ is monotone in $(-\infty,0)$ and in $(0,+\infty)$. Together with the continuity of $V$ in $\R$ and the interface condition $V'(0^-)=\sigma V'(0^+)$ with $\sigma>0$, one then deduces that $V$ is nondecreasing in $\R$ if $V'(0^\pm)>0$ (and then $K_1<K_2$ in this case). Furthermore, if $V'(0^\pm)=0$, then necessarily $V(0)=K_1=K_2$ by~\eqref{3.5}--\eqref{3.6}, hence $V\equiv K_1=K_2$ in $(-\infty,0]$ and $V\equiv K_2=K_1$ in $[0,+\infty)$ by the Cauchy-Lipschitz theorem. Notice that the case $V'(0^\pm)<0$ is impossible since it would imply that $V$ is nonincreasing and not constant in $\R$, and then $K_1>K_2$, which is ruled out by assumption. Therefore, in all cases, $V$ is monotone in $\R$, and $V\equiv K_1=K_2$ in $\R$ if $K_1=K_2$.
	
Consider now the case $K_1<K_2$. Then, $V'\ge0$ in $(-\infty,0^-]\cup[0^+,+\infty)$ and $V'(0^\pm)>0$, from the previous paragraph. If there is $x_0\in\R\setminus\{0\}$ such that $V'(x_0)=0$, then~\eqref{3.5}--\eqref{3.6} and the Cauchy-Lipschitz theorem imply that $V(x_0)=K_1$ and~$V\equiv K_1$ in $(-\infty,0]$ (if $x_0<0$), or $V(x_0)=K_2$ and~$V\equiv K_2$ in $[0,+\infty)$ (if $x_0>0$), hence $V'(0^-)=0$ or $V'(0^+)=0$, a contradiction. Therefore,~$V'>0$ in $\R\setminus\{0\}$ and then in $(-\infty,0^-]\cup[0^+,+\infty)$, yielding in particular $K_1<V<K_2$ in $\R$. Moreover, by~\eqref{3.5}--\eqref{3.6} and by the interface condition $V'(0^-)=\sigma V'(0^+)$, one has 
$$\int_{V(0)}^{K_1}f_1(s) \mathrm{d}s=\frac{d_1 \sigma^2}{d_2}\int^{K_2}_{V(0)}f_2(s)\mathrm{d}s.$$
Notice that the function $\nu\mapsto \int_{\nu}^{K_1}f_1(s) \mathrm{d}s$ is continuous increasing in $[K_1,K_2]$ and vanishes at $K_1$, while the function $\nu\mapsto \frac{d_1 \sigma^2}{d_2}\int^{K_2}_{\nu}f_2(s)\mathrm{d}s$  is continuous  decreasing in $[K_1,K_2]$ and vanishes at $K_2$. Therefore, there exists a unique $\nu_0\in (K_1,K_2)$ such that 
$$\int_{\nu_0}^{K_1}f_1(s) \mathrm{d}s=\frac{d_1 \sigma^2}{d_2}\int^{K_2}_{\nu_0}f_2(s)\mathrm{d}s,$$
and necessarily $V(0)=\nu_0$. Hence, $V(0)$ is unique, and $V'(0^-)$ and $V'(0^+)$ are uniquely determined by 
$$V'(0^-)=\sqrt{\frac{2}{d_1}\int_{V(0)}^{K_1}f_1(s) \mathrm{d}s},~~~~V'(0^+)=\sqrt{\frac{2}{d_2}\int_{V(0)}^{K_2}f_2(s) \mathrm{d}s},$$
whence the uniqueness of $V$ follows from the Cauchy-Lipschitz theorem. This completes the proof of Proposition~\ref{prop2.1}.
\end{proof}

\begin{proof}[Proof of Theorem~$\ref{thm2.3}$]
Let $u$ be the solution  to~\eqref{1.1} with a nonnegative bounded and continuous initial datum $u_0\not\equiv 0$. The comparison principle (Proposition~\ref{prop 1.3}) yields $0<u(t,x)\le M:=\max(K_2,\Vert u_0\Vert_{L^\infty(\mathbb{R})})$ for all $(t,x)\in(0,+\infty)\times\R$.
	 
Choosing $R>0$ and $\Psi$ as in~\eqref{R}--\eqref{w}, there is $\varep>0$ small enough such that $-d_2\varep \Psi''<f_2(\varep \Psi)$ in~$(-R,R)$ and $\varep \Psi(\cdot-R-1)<u(1,\cdot)\le M$ in $\mathbb{R}$. Let $\underline u$ and $\overline u$ be, respectively, the solutions to~\eqref{1.1} with initial data $\varep \Psi(\cdot-R-1)$ and $M$. It follows in particular from Proposition~\ref{prop 1.3} that $\underline u$ is nonnegative in $[0,+\infty)\times\R$ (and even positive in $(0,+\infty)\times\R$). The standard parabolic maximum principle applied in $[0,+\infty)\times[1,2R+1]$ then implies that $\underline{u}(t,x)>\Psi(x-R-1)$ for all $(t,x)\in(0,+\infty)\times[1,2R+1]$. Therefore, $\underline{u}(h,\cdot)>\Psi(\cdot-R-1)=\underline{u}(0,\cdot)$ in $\R$, for every $h>0$. Proposition~\ref{prop 1.3} again then implies that $\underline{u}(t+h,\cdot)>\underline{u}(t,\cdot)$ in $\R$ for every $h>0$ and $t\ge0$, that is, $\underline{u}$ is increasing with respect to~$t$ in $[0,+\infty)\times\R$. Similarly,~$\overline u$ is nonincreasing with respect to~$t$ in $[0,+\infty)\times\R$. Since $0<\underline u(t,x)<\overline u(t,x)\le M$ for all $(t,x)\in(0,+\infty)\times\R$ (the strict inequalities come from Proposition~\ref{prop 1.3}), the Schauder estimates of Proposition~\ref{pro1} imply that~$\underline u(t,\cdot)$ and $\overline u(t,\cdot)$ converge as $t\to+\infty$, locally uniformly in $\mathbb{R}$, to positive bounded classical stationary solutions $p$ and $q$  of~\eqref{1.1}, respectively. Moreover,
\begin{align*}
0<p=\lim_{t\to+\infty} \underline u(t,\cdot)\le \liminf_{t\to+\infty} u(t,\cdot)\le \limsup_{t\to+\infty} u(t,\cdot)\le\lim_{t\to+\infty}\overline u(t,\cdot)=q\le M,
\end{align*}
locally uniformly in $\mathbb{R}$. From Proposition~\ref{prop2.1} and the uniqueness of the positive bounded classical stationary solution $V$ to problem~\eqref{1.1}, one gets $p=q=V$ in $\mathbb{R}$, and the desired property~\eqref{convV} of Theorem~\ref{thm2.3} is thereby proved.

Assume now that $u_0$ is compactly supported. Since $V(-\infty)=K_1$, $V(+\infty)=K_2$ and $K_1\le V(x)\le K_2$ for all $x\in\R$, it follows that, for any $\delta\in(0,K_1)$, there exist $x_1<0$ negative enough and $x_2>0$ positive enough such that 
\begin{align}
\label{3.8}
K_1\le V(x)\le K_1+\frac\delta2\ \ \text{for all}\ x\le x_1, \ \hbox{ and }\ K_2-\frac\delta2 \le V(x)\le K_2 \ \ \text{for all}\   x\ge x_2.
\end{align}
By~\eqref{convV}, one can pick $t_0>0$ sufficiently large so that
\begin{align}
\label{3.9}
|u(t,x)-V(x)|\le\frac{\delta}{2}~~\text{for all }t\ge t_0\text{ and }x\in[x_1,x_2].
\end{align}
Thanks to~\eqref{3.8}--\eqref{3.9}, it is easily seen that, for all $t\ge t_0$,
\begin{equation}
\label{3.10}
K_1-\frac{\delta}{2}\le u(t,x_1)\le K_1+\delta,
\end{equation}
and
$$K_2-\delta\le  u(t,x_2)\le K_2+\frac{\delta}{2}.$$
	 
We first look at the spreading of $u$ in patch~$1$. Let $z_0\not\equiv 0$ be a nonnegative bounded continuous and  compactly supported function in $\mathbb{R}$ such that spt$(z_0)\subset[x_1-2,x_1-1]$ and $0\le z_0(x)<\min\big(\Vert u_0\Vert_{L^\infty(\mathbb{R})},K_1-\delta,u(t_0,x)\big)$ for all $x\in\mathbb{R}$. Consider the Cauchy problem
\begin{align}
\label{3.12}
\begin{cases}
z_t=d_1z_{xx}+g_1(z),\ \ &t>0,~x\in\mathbb{R},\\
z(0,x)=z_0,\ \ &x\in\R,
\end{cases}
\end{align}
where $g_1$ is of class $C^1([0,+\infty))$ and satisfies $g_1(0)=g_1(K_1-\delta)=0$, $0<g_1(s)\le g'_1(0)s$ for all~$s\in(0,K_1-\delta)$, $g'_1(K_1-\delta)<0$, and $g_1<0$ in $(K_1-\delta,+\infty)$. Moreover, $g_1$ can be chosen so that~$g'_1(0)=f'_1(0)$ and  $g_1\le f_1$ in $[0,K_1-\delta]$. From the maximum principle, it immediately follows that $0\le z(t,x)<K_1-\delta$ for all $t\ge0$ and $x\in\mathbb{R}$. This implies that $z(t-t_0,x_1)<K_1-\delta<u(t,x_1)$ for all $t\ge t_0$, thanks to~\eqref{3.10}. Notice also that $z_0(x)<u(t_0,x)$ for $x\in(-\infty,x_1]$. By the comparison principle, it turns out that $z(t-t_0,x)<u(t,x)$ for all $t>t_0$ and $x\le x_1$. Furthermore, it is known that the solution $z$ of~\eqref{3.12} spreads in both directions with the spreading speed $c^*_1=2\sqrt{d_1g'_1(0)}=2\sqrt{d_1f'_1(0)}$ (see \cite{AW2}), hence
\begin{align*}
\inf_{|x|\le (c^*_1-\varep)t}z(t,x)\to K_1-\delta \ \ \text{as}~t\to+\infty,\ \ \text{for all}~0< \varep<c^*_1.
\end{align*} 
By virtue of~\eqref{3.8}, we then obtain that, for any $0<\varep<c^*_1$, there is $t'_0>t_0$ such that, for all $t>t'_0$ and $x\le x_1$, 
\begin{align}
\label{3.13}
V(x)-3\delta<K_1-2\delta\le\inf_{(-c^*_1+\varep/2)(t-t_0)\le y\le x_1} z(t-t_0,y)\!\le\! \inf_{(-c^*_1+\varep)t\le y\le x_1} u(t,y).
\end{align}
	 
Next, set $M_1:=\max\big(\Vert u_0\Vert_{L^\infty(\mathbb{R})}, K_1+\delta,K_2\big)$. Let $\widetilde g_1$ be a $C^1([0,+\infty))$ function such that $\widetilde g_1(0)=\widetilde g_1(K_1+\delta)=0$, $\widetilde g_1>0$ in $(0,K_1+\delta)$, $\widetilde g'_1(0)>0$, $\widetilde g'_1(K_1+\delta)<0$, and $\widetilde g<0$ in $(K_1+\delta,+\infty)$. We can also choose $\widetilde g_1$ so that $f_1\le \widetilde g_1$ in $[0,+\infty)$. Then, the solution to the ODE $\xi'(t)=\widetilde g_1(\xi(t))$ for~$t\ge t_0$ with $\xi(t_0)=M_1$ is nonincreasing for $t\ge t_0$ and satisfies $\xi(t)\to K_1+\delta$ as $t\to+\infty$. One has~$0<u(t,x)\le M_1$ for all $(t,x)\in(0,+\infty)\times \mathbb{R}$ thanks to Proposition~\ref{prop 1.3}, hence $u(t_0,x)\le M_1=\xi(t_0)$ for all $x\le x_1$. Moreover,  $u(t,x_1)\le K_1+\delta\le \xi(t)$  for all $t\ge t_0$ by~\eqref{3.10}. Applying a comparison argument yields $u(t,x)\le \xi(t)$ for all $t\ge t_0$ and $x\le x_1$. Therefore, we can choose $t_1> t_0$ such that 
\begin{align}
\label{3.14}
\sup_{x\le x_1} u(t_1, x)\le K_1+\frac{3\delta}{2}.
\end{align}

Let now $\overline g_1$ be of class $C^1([0,+\infty))$ satisfying $\overline g_1(0)=\overline g_1(K_1+2\delta)=0$,  $0<\overline g_1(s)\le \overline g'_1(0)s$ for~$s\in(0,K_1+2\delta)$, $\overline g'_1(0)=f'_1(0)$ and $f_1\le \overline g_1$ in $[0,+\infty)$. Then, it is well-known that the KPP equation $v_t=d_1 v_{xx}+\overline g_1(v)$ admits standard traveling wave solutions of the type $v(t,x)=\overline\varphi_c(\pm x-ct)$ with (decreasing) $\overline\varphi_c:\R\to(0,K_1+2\delta)$ if and only if $c\ge c^*_1=2\sqrt{d_1\overline{g}'_1(0)}=2\sqrt{d_1f'_1(0)}$. For each~$c\ge c^*_1$, the function $\overline\varphi_c$ satisfies
\be\label{overphic}
d_1\overline\varphi_c''+c\overline\varphi_c'+\overline g_1(\overline\varphi_c)=0 \ \text{in}~\mathbb{R}, \ \ \overline\varphi_c'<0 \ \text{in}~\R, \ \ \overline\varphi_c(-\infty)=K_1+2\delta,\ \ \overline\varphi_c(+\infty)=0,
\ee
and $\overline\varphi_c$ is unique up to translations. In particular, for $c=c^*_1$, by choosing $A>0$ sufficiently large, there holds
\begin{align}
\label{3.17}
K_1+\frac{3\delta}{2}\le\overline\varphi_{c^*_1}(-x_1-c^*_1 t-A)< K_1+2\delta \ \ \text{for all} \ t\ge t_1.
\end{align}

Due to the exponential decay of $\overline\varphi_{c^*_1}(s)$ as $s\to+\infty$ (as in the second case of~\eqref{phiexp}) and the Gaussian upper bound of $u(t_1,x)$ for all $x\le x_1$ by Lemma~\ref{lemma1.3}, together with~\eqref{3.14}, it can be derived that (up to increasing $A$ if needed)
$$u(t_1,x)\le \overline\varphi_{c^*_1}(-x-c^*_1 t_1 -A)~~\text{for}~x\le x_1.$$
We also notice from~\eqref{3.10} and~\eqref{3.17} that $u(t,x_1)\le K_1+\delta<\overline\varphi_{c^*_1}(-x_1-c^*_1 t -A)$ for all $t\ge t_1$. The comparison principle  gives 
\begin{equation}
\label{3.18}
u(t,x)\le\overline\varphi_{c^*_1}(-x-c^*_1 t -A)~~\text{for all}~t\ge t_1~\text{and}~x\le x_1.
\end{equation}
Therefore, for all $0<\varep<c^*_1$ and for all $t\ge t_1$ and $x\le x_1$, there holds 
\begin{align}
\label{3.19}
\sup_{-(c^*_1-\varep)t\le y \le x_1} u(t,y)\!\le\! \sup_{-(c^*_1-\varep)t\le y \le x_1}\overline\varphi_{c^*_1}(-y-c^*_1 t -A)\!\le \! K_1+2\delta \!\le \! V(x)+2\delta.
\end{align}
Combining~\eqref{3.13} with~\eqref{3.19}, we obtain
\begin{align*}
\limsup_{t\to+\infty}\Big(\sup_{-(c^*_1-\varep)t\le x \le x_1}|u(t,x)-V(x)|\Big)\le3\delta\ \hbox{ for all }0< \varep<c^*_1.
\end{align*}
Together with~\eqref{convV} and the arbitrariness of $\delta>0$ small enough, one gets that $u$ spreads to the left at least with speed $c^*_1$, that is, for every $0< \varep\le c^*_1$,
$$\sup_{-(c^*_1-\varep)t\le x \le0}|u(t,x)-V(x)|\to0\hbox{ as }t\to+\infty.$$
	  
On the other hand,~\eqref{3.18} also implies that, for all $\varep>0$,
\begin{align*}
\limsup_{t\to+\infty}\Big(\sup_{x\le -(c^*_1+\varep)t} u(t,x)\Big)\le \lim_{t\to+\infty}\Big(\sup_{x\le -(c^*_1+\varep)t}\overline\varphi_{c^*_1}(-x-c^*_1 t -A)\Big)=0,
\end{align*}
hence the limsup is a limit and $u$ spreads to the left at most with speed $c^*_1$.

Therefore, the leftward spreading result of $u$ is proved. Similarly, one can also show that $u$ spreads in patch~$2$ with speed $c^*_2=2\sqrt{d_2f'_2(0)}$. Hence,~\eqref{spreadKPP} follows, and the proof of Theorem~\ref{thm2.3} is complete.
\end{proof}

%%%%%%%%%%%%%%%%%%%%%%%%%%%%%%%%%%%%%%%%%%%%%%%%%%%%
%%%%%%%%%%%%%%%%%%%%%%%%%%%%%%%%%%%%%%%%%%%%%%%%%%%%

\section{The KPP-bistable case}\label{Sec 4}

In this section, we investigate~\eqref{1.1} with KPP-bistable reactions. We assume that patch 1 is of KPP type, whereas patch 2 is of bistable type, that is, we assume~\eqref{f1-kpp}--\eqref{f2-bistable}. We consider in complete generality the sign of the mass $\int_{0}^{K_2}f_2(s)\mathrm{d}s$ and the relation between $K_1$ and $\theta$ or $K_2$ (or possibly $\theta^*$ where $\theta^*\in(\theta,K_2)$ is such that $\int_0^{\theta^*}f_2(s)\mathrm{d}s=0$ when $\int_0^{K_2}f_2(s)\mathrm{d}s>0$).

%%%%%%%%%%%%%%%%%%%%%%%%%%%%%%%%%%%%%%%%%%%%%%%%%%%%

\subsection{Semi-persistence result: proof of Theorem~\ref{thm2.4}}
\label{Sec 4.1}

To begin with, we prove the semi-persistence result and the spreading result in patch 1, thanks to the KPP assumption on $f_1$. The technique here is similar to that of Theorem~\ref{thm2.3}.

\begin{proof}[Proof of Theorem~$\ref{thm2.4}$]
Let $u$ be the solution to~\eqref{1.1} with a nonnegative continuous and compactly supported initial datum $u_0\not\equiv 0$. By Proposition~\ref{prop 1.3}, we have $0\!<\!u(t,x)\!<\!M\!:=\!\max\big(K_1,K_2,\Vert u_0\Vert_{L^\infty(\mathbb{R})}\big)$ for all $t>0$ and $x\in\mathbb{R}$.
	
Take $R>0$ large enough  such that
\begin{equation}
\label{4.1}
\frac{\pi}{2R}<\sqrt{\frac{f'_1(0)}{2d_1}},
\end{equation}
and then define $\Psi:\R\to\R$ as in~\eqref{w}, that is,
\begin{equation}
\label{4.2}
\begin{aligned}
\Psi(x)=\begin{cases}
\ \displaystyle\cos\Big(\frac{\pi}{2R}x\Big)&\text{for }x\in[-R,R],\cr
\ 0&\text{for }x\in\R\setminus[-R,R].
\end{cases}
\end{aligned}
\end{equation}
Then there exists $\eta_0>0$ such that $-d_1(\eta\Psi)''\le f_1(\eta\Psi)$ in $(-R,R)$ for all $0<\eta\le \eta_0$. Choose now any~$x_0\le -R$ and pick $\eta\in(0,\eta_0]$ such that $\eta\Psi(\cdot-x_0)<u(1,\cdot)$ in $\mathbb{R}$. Let $v$ and $w$ be  solutions to~\eqref{1.1} with initial data $v(0,\cdot)=\eta\Psi(\cdot-x_0)$ in $\mathbb{R}$ and $w(0,\cdot)\equiv \max(K_1,K_2,\Vert u_0\Vert_{L^\infty(\mathbb{R})})$ in $\mathbb{R}$. Then, as in the proof of the first part of Theorem~\ref{thm2.3}, $v$ is increasing with respect to~$t$ and $w$ is nonincreasing with respect to~$t$. Moreover,~$0<v(t,x)< u(t+1,x)<w(t+1,x)\le M$ for all $t>0$ and $x\in\mathbb{R}$. By the Schauder estimates of Proposition~\ref{pro1}, it follows that $v(t,\cdot)$ and $w(t,\cdot)$ converge as $t\to+\infty$, locally uniformly in $\mathbb{R}$,  to  positive bounded stationary solutions $p$ and $q$ of~\eqref{1.1}, respectively. Furthermore,
\begin{equation}\label{4.3}
0<p\le \liminf_{t\to+\infty} u(t,\cdot)\le \limsup_{t\to+\infty} u(t,\cdot)\le q\le M,~~\text{locally uniformly in}~\mathbb{R}.
\end{equation} 

Notice also that $p> v_0$ in $\mathbb{R}$. We observe  from the continuity of $p$ and $v_0$ that there is $\hat\kappa>1$ such that $p>\eta\Psi(\cdot-\kappa x_0)$ in $[\kappa x_0-R,\kappa x_0+R]$ for all $\kappa\in[1,\hat \kappa]$. Define
$$\kappa^*:=\sup\big\{\kappa\ge 1: p>\eta\Psi(\cdot-\widetilde\kappa x_0)~\text{in}~[\widetilde\kappa x_0-R,\widetilde\kappa x_0+R]~\text{for all}~\widetilde\kappa\in [1,\kappa]\big\}.$$
It follows that $\kappa^*\ge\hat\kappa>1$. We are going to prove that $\kappa^*=+\infty$. Assuming by contradiction that~$\kappa^*<+\infty$, we see from the  definition of $\kappa^*$ that $p\ge \eta\Psi(\cdot-\kappa^*x_0)$ in $[\kappa^*x_0-R,\kappa^*x_0+R]$ and there is $x^*\in[\kappa^*x_0-R,\kappa^*x_0+R]$ such that $p(x^*)= \eta\Psi(x^*-\kappa^*x_0)$. Since $p>0$ in $\mathbb{R}$ and $\Psi(\cdot-\kappa^*x_0)=0$ at $\kappa^*x_0\pm R$, one has $x^*\in(\kappa^*x_0-R,\kappa^*x_0+R)$. Then the strong elliptic maximum principle implies that $p\equiv \eta\Psi(\cdot-\kappa^*x_0)$ in $(\kappa^*x_0-R,\kappa^*x_0+R)$ and then in $[\kappa^*x_0-R,\kappa^*x_0+R]$ by continuity, which is impossible at $\kappa^*x_0\pm R$. Thus, $\kappa^*=+\infty$ and $p>\eta\Psi(\cdot-\kappa x_0)$ in $[\kappa x_0-R,\kappa x_0+R]$ for all $\kappa\ge 1$. This implies, in particular, that $p(x)>\eta\Psi(0)=\eta$ for all $x\le x_0$. Thus, 
\begin{equation}
\label{4.4}
\liminf_{t\to+\infty}u(t,x)\ge p(x)>\eta~~\text{for all}~x\le x_0.
\end{equation}
On the other hand, since $p$ is continuous and positive in $\mathbb{R}$, one gets from~\eqref{4.3} that, for any given~$\overline x>x_0$,
\begin{equation}
\label{4.5}
\liminf_{t\to+\infty}u(t,x)\ge \min_{[x_0,\overline x]}p>0~~\text{for all}~ x\in[x_0, \overline x].
\end{equation}
Combining~\eqref{4.4} with~\eqref{4.5}, one reaches the semi-persistence result, that is, for any $\overline x\in\mathbb{R}$,
$$\inf_{x\le\overline{x}}\Big(\liminf_{t\to+\infty} u(t,x)\Big)>0.$$

In what follows, we turn to the proof of the spreading result in patch~$1$. First of all, as for $V$ in the proof of Proposition~\ref{prop2.1}, one sees that the functions $p$ and $q$ given in~\eqref{4.3} satisfy $p(x)\to K_1$ and $q(x)\to K_1$ as $x\to-\infty$. Fix now any $\delta\in(0,K_1/2)$. From the previous observations together with~\eqref{4.3}, there exist $t_1>0$ and $x_1<0$ such that
\begin{equation}\label{4.6}
K_1-\frac{\delta}{2}\le u(t,x_1)\le K_1+\delta~~\text{for all}~t\ge t_1.
\end{equation} 
The rest of the proof is similar to that of Theorem~\ref{thm2.3}. We just sketch main steps. With $z_0$ and $g_1$ as in~\eqref{3.12} in the proof of Theorem~\ref{thm2.3}, and using especially the left inequality in~\eqref{4.6}, it follows as in~\eqref{3.13} that, for any $\varep\in(0,c^*_1)$, there is $t'_1>t_1$ such that, for all $t>t'_1$,
\begin{align}
\label{4.8}
K_1-2\delta\le\inf_{-(c^*_1-\varep/2)(t-t_1)\le y\le x_1} z(t-t_1,y)\le \inf_{-(c^*_1-\varep)t\le y\le x_1} u(t,y).
\end{align}
Similarly, as in~\eqref{3.14}, using especially the right inequality in~\eqref{4.6}, there is $t_2>t'_1$ such that
$$\sup_{x\le x_1} u(t_2, x)\le K_1+\frac{3\delta}{2}.$$

Next, let $\overline{g}_1$ and $\overline{\varphi}_{c^*_1}$ be as in~\eqref{overphic} with $c^*=2\sqrt{d_1f'_1(0)}$. Then there is $A>0$ such that, for each $\varep\in(0,c^*_1)$,~\eqref{3.19} holds without any reference to $V$, that is, there is $t_3>0$ such that
\begin{align}
\label{4.13}
\sup_{-(c^*_1-\varep)t\le x \le x_1} u(t,x)\le \sup_{-(c^*_1-\varep)t\le x \le x_1}\overline\varphi_{c^*_1}(-x-c^*_1t-A)\le K_1+2\delta\ \hbox{ for all }t\ge t_3.
\end{align}
Owing to~\eqref{4.8} and~\eqref{4.13}, it follows that 
$$\forall\,\varep\in(0,c^*_1), \ \forall\,\delta\in\Big(0,\frac{K_1}{2}\Big),~\exists\,x_1\in\R,~~\limsup_{t\to+\infty}\Big(  \sup_{-(c^*_1-\varep)t\le x \le x_1} |u(t,x)-K_1|\Big)\le 2\delta,$$
namely, $u$ spreads to the left at least with speed $c^*_1$. Moreover, we can also deduce as in the proof of Theorem~\ref{thm2.3} that, for every $\varep>0$,
$$\limsup_{t\to+\infty}\Big(\sup_{x\le -(c^*_1+\varep)t} u(t,x)\Big)\le\lim_{t\to+\infty}\Big(\sup_{x\le -(c^*_1+\varep)t}\overline\varphi_{c^*_1}(-x-c^*_1t-A)\Big)=0,$$
hence $\sup_{x\le -(c^*_1+\varep)t} u(t,x)\to0$ as $t\to+\infty$, for all $\varep>0$. That is, $u$ spreads at most with speed $c^*_1$ in the negative direction. This finishes the proof of Theorem~\ref{thm2.4}.
\end{proof}

%%%%%%%%%%%%%%%%%%%%%%%%%%%%%%%%%%%%%%%%%%%%%%%%%%%%

\subsection{Preliminaries on the stationary problem: proofs of Propositions~\ref{prop 2.5}--\ref{prop 2.7}}
\label{Sec 4.2}

This section is devoted to the study of the stationary problem associated with~\eqref{1.1} in the KPP-bistable case~\eqref{f1-kpp}--\eqref{f2-bistable}, and we give the proofs of Propositions~\ref{prop 2.5}--\ref{prop 2.7}.

\begin{proof}[Proof of Proposition~$\ref{prop 2.5}$]
Suppose that $U$ is a nonnegative classical stationary solution of~\eqref{1.1} such that $U(-\infty)=K_1$ and $U(+\infty)=0$ (hence, $U'(\pm\infty)=0$ from standard elliptic estimates). As in the proof of Proposition~\ref{prop2.1}, it follows that $U>0$ in $\R$, that $U$ is monotone in $(-\infty,0]$, and that $U'<0$ (resp. $U'>0$, resp. $U'\equiv0$) in $(-\infty,0^-]$ if $U(0)<K_1$ (resp. if $U(0)>K_1$, resp. if $U(0)=K_1$). Furthermore, multiplying $d_2 U''+f_2(U)=0$ by $U'$ and  integrating the resulting equation over $[x,+\infty)$ for any $x\ge0$ yields 
\begin{equation}
\label{4.14}
\frac{d_2}{2}(U'(x^+))^2=-\int_{0}^{U(x)}f_2(s)\mathrm{d}s\ \hbox{ for all }x\ge0.
\end{equation}
To discuss the behavior of $U$ in $[0,+\infty)$, we distinguish three cases, according to the sign of $\int_0^{K_2}\!f_2(s)\mathrm{d}s$.

\vskip 1mm
\textit{Case 1: $\int_0^{K_2}\!f_2(s) \mathrm{d}s<0$}. Then, $\int_0^\nu f_2(s)\mathrm{d}s<0$ for all $\nu>0$ and one infers from~\eqref{4.14} that $U'$ has a strict constant sign in $[0^+,+\infty)$, whence $U'<0$ in $[0^+,+\infty)$ since $U(0)>0$ and  $U(+\infty)=0$. This implies that $U'(0^-)<0$ by using the interface condition in~\eqref{1.1}, hence $U(0)<K_1$ and $U'<0$ in~$(-\infty,0^-]$ from the previous paragraph. Lastly, formulas~\eqref{4.14} and~\eqref{3.5} (at $x=0$ and with $U$ instead of $V$), together with the interface condition $U'(0^-)=\sigma U'(0^+)$, lead to~\eqref{eqU0}.
	
\vskip 1mm
\textit{Case 2: $\int_0^{K_2}\!f_2(s) \mathrm{d}s=0$}. Suppose that there is a point $x_0\in[0,+\infty)$ such that $U(x_0)= K_2$. By~\eqref{4.14}, one deduces that $U'(x_0^+)=0$, and then $U\equiv K_2$ in $[0,+\infty)$ by the Cauchy-Lipschitz theorem. This contradicts the limit $U(+\infty)=0$. Thus, $0<U<K_2$ in $[0,+\infty)$ and therefore $U'$ has a strict constant sign in $[0^+,+\infty)$ by~\eqref{4.14}, hence $U'<0$ in $[0^+,+\infty)$. Consequently, as in Case~1, $U'(0^-)<0$, $U'<0$ in $(-\infty,0^-]$ and $U(0)<K_1$ (see Fig.~\ref{balanced}). Notice also that $\int_0^{U(0)}f_2(s)\mathrm{d}s<0$ and that~\eqref{eqU0} holds as in Case~1.
\begin{figure}[htbp]
\centering
\subfigure[$K_1>K_2$]{\includegraphics[width=1.9in]{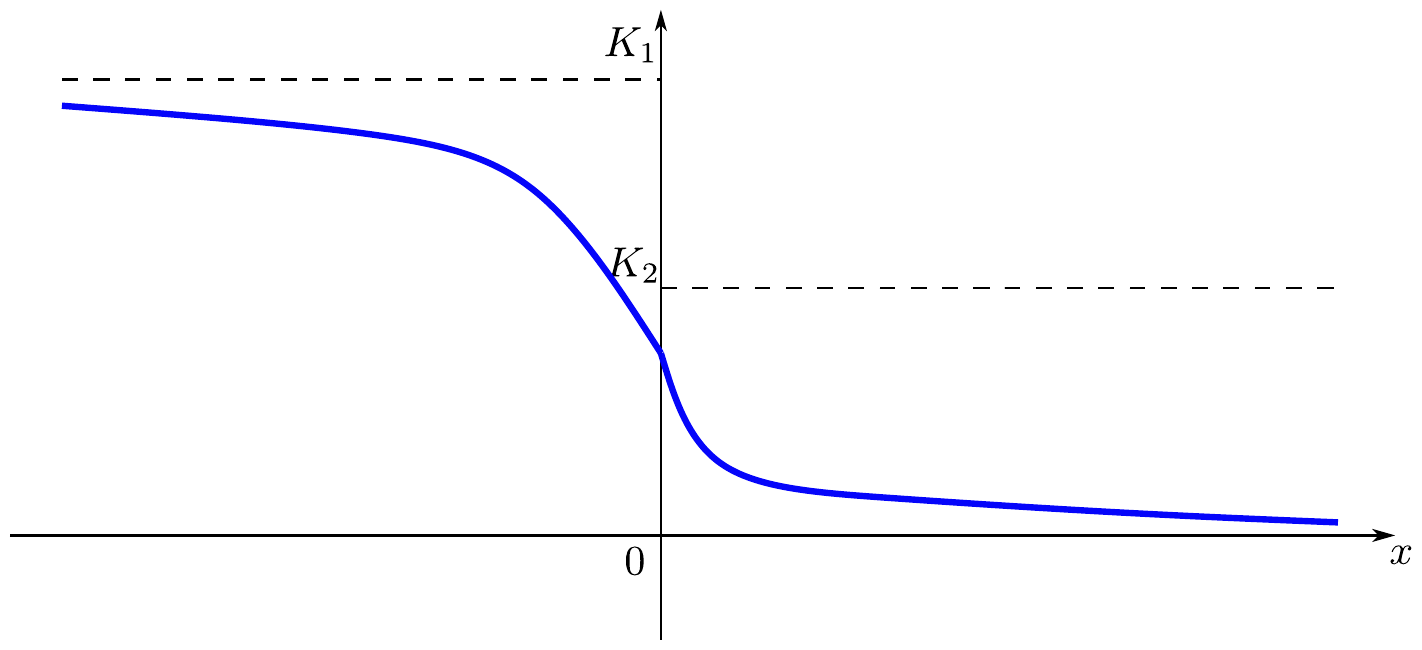}} ~ \subfigure[$K_1=K_2$]{\includegraphics[width=1.9in]{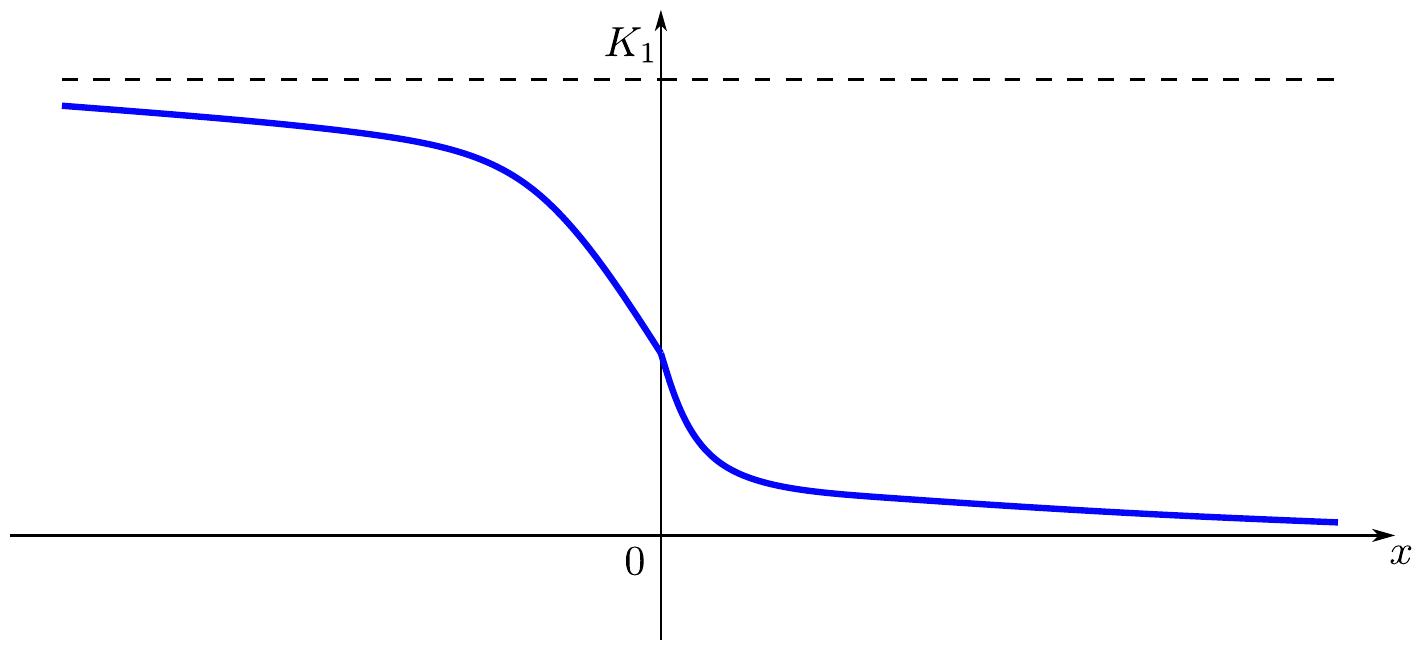}} ~ \subfigure[$K_1<K_2$]{\includegraphics[width=1.9in]{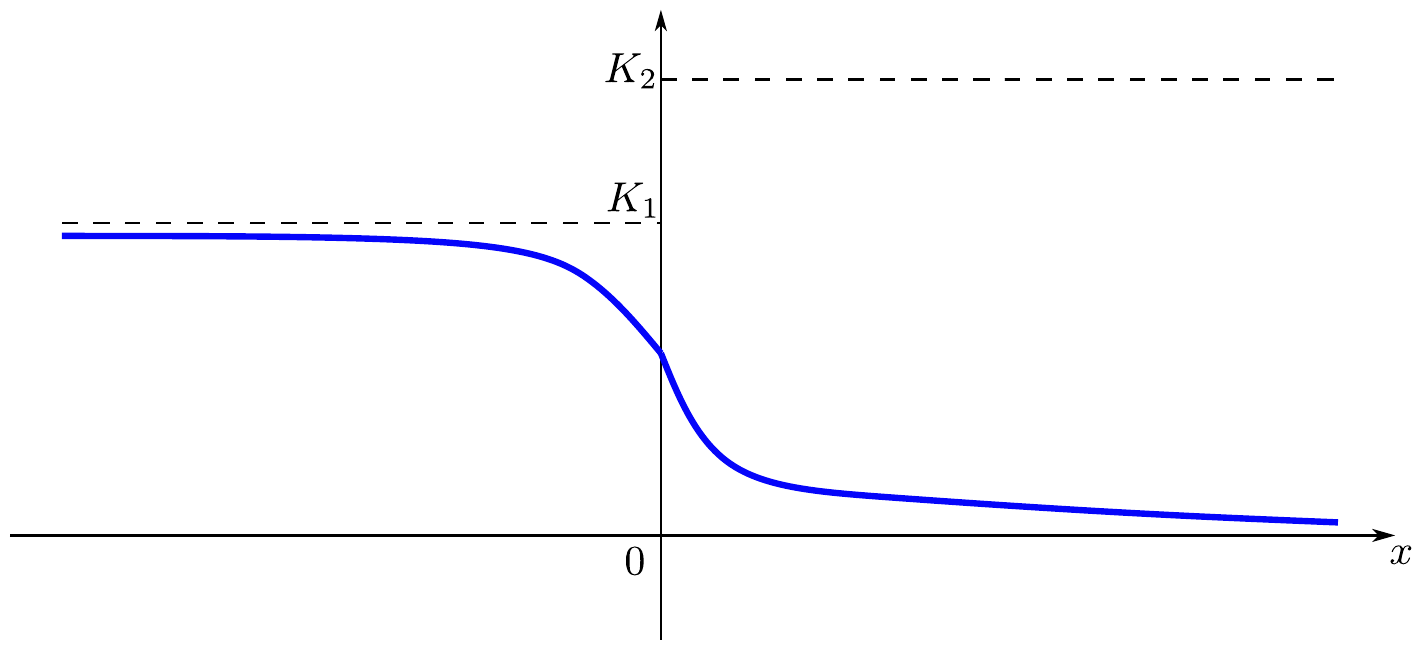}}
\caption{Profile of a steady solution $U$ with $U(-\infty)\!=\!K_1$ and $U(+\infty)\!=\!0$, if $\int_0^{K_2}f_2(s) \mathrm{d}s\!=\!0$.}
\label{balanced}
\end{figure}

\vskip 1mm
\textit{Case 3: $\int_0^{K_2}\!f_2(s) \mathrm{d}s>0$}. Let $\theta^*\in(\theta,K_2)$  be such that $\int_0^{\theta^*}\!f_2(s)\mathrm{d}s=0$, and denote
\be\label{defQ}
Q=\sup\Big\{\nu>\theta^*:\int_0^{\nu'}\!\!f_2(s)\mathrm{d}s>0\hbox{ for all }\nu'\in(\theta^*,\nu)\Big\}\ \in(\theta^*,+\infty].
\ee
We first observe from~\eqref{4.14} that $U(x)\notin(\theta^*,Q)$ for all $x\ge0$. By  continuity of $U$ and $U(+\infty)=0$, one then derives that  $0<U\le \theta^*$ in $[0,+\infty)$. Suppose in this paragraph that the set $\{x\ge0:U'(x)=0\}$ is not empty.\footnote{Notice that, if $U'(0^+)=0$, then $U'(0^-)=0$ as well, hence $U$ is differentiable at $0$ with $U'(0)=0$.} From~\eqref{4.14} and the inequality $U\le\theta^*$ in $[0,+\infty)$, this set is included in $\{x\ge0:U(x)=\theta^*\}$ and, since $U(+\infty)=0$, one can then define $x_0:=\max\{x\ge 0:U'(x)=0\}\,\in[0,+\infty)$. One then has~$U(x_0)=\theta^*$ and $U'<0$ in $(x_0,+\infty)$ by definition of $x_0$. The Cauchy-Lipschitz theorem then implies that $U(x)=U(2x_0-x)$ for all $x\in[0,x_0]$, hence $U'>0$ in $[0^+,x_0)$ if $x_0>0$. From the general observations at the beginning of the proof of the present proposition, one then gets that, if $x_0>0$, then $U'>0$ in $(-\infty,0^-]\cup[0^+,x_0)$, $K_1<U(0)<\theta^*$, and $U'<0$ in $(x_0,+\infty)$ (see the black curve in Fig.~\ref{bump-sol} (a)), whereas $U\equiv K_1=\theta^*$ in $(-\infty,0]$ and $U'<0$ in $(0,+\infty)$ if $x_0=0$ (see the black curve in Fig.~\ref{bump-sol} (b)). To sum up, under the assumption $\{x\ge0:U'(x)=0\}\neq\emptyset$, one has $K_1\le\theta^*$ and~\eqref{eqU0} holds good if $x_0>0$, while the two integrals in~\eqref{eqU0} vanish if $x_0=0$.
\begin{figure}[htbp]
\centering
\subfigure[$K_1<\theta^*$]{\includegraphics[width=1.9in]{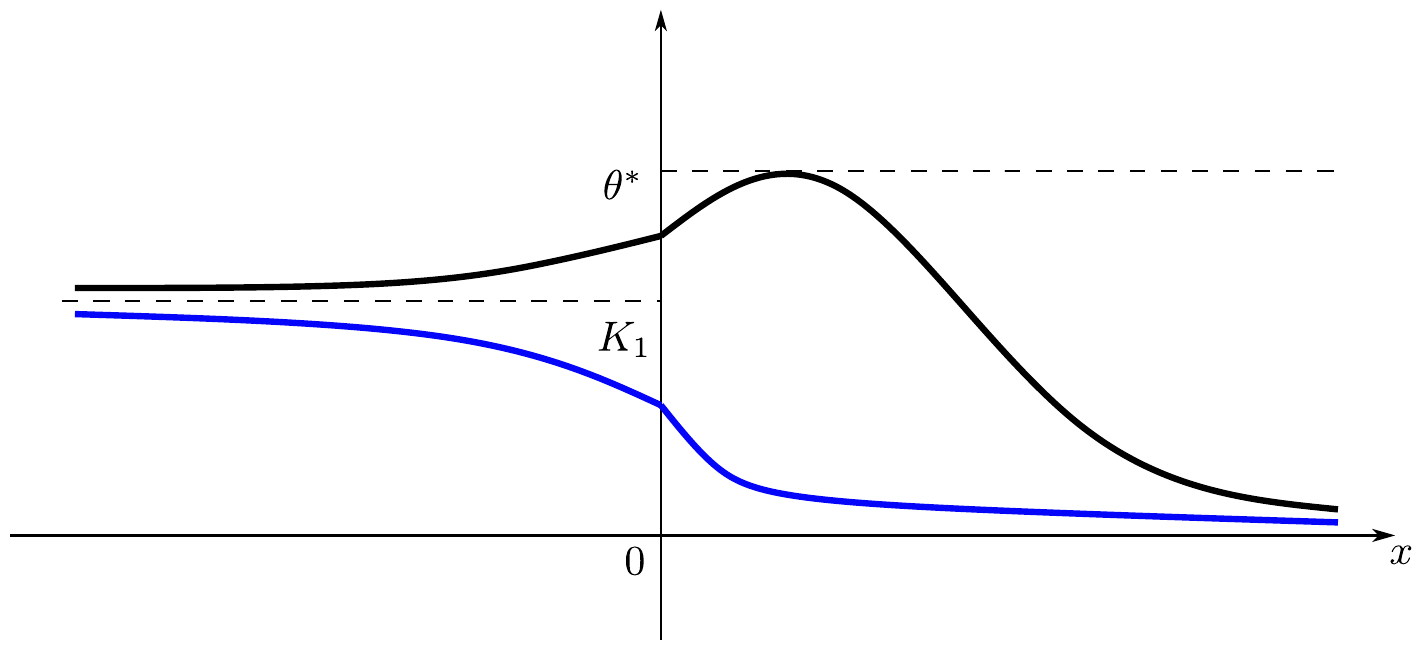}} ~ \subfigure[$K_1=\theta^*$]{\includegraphics[width=1.9in]{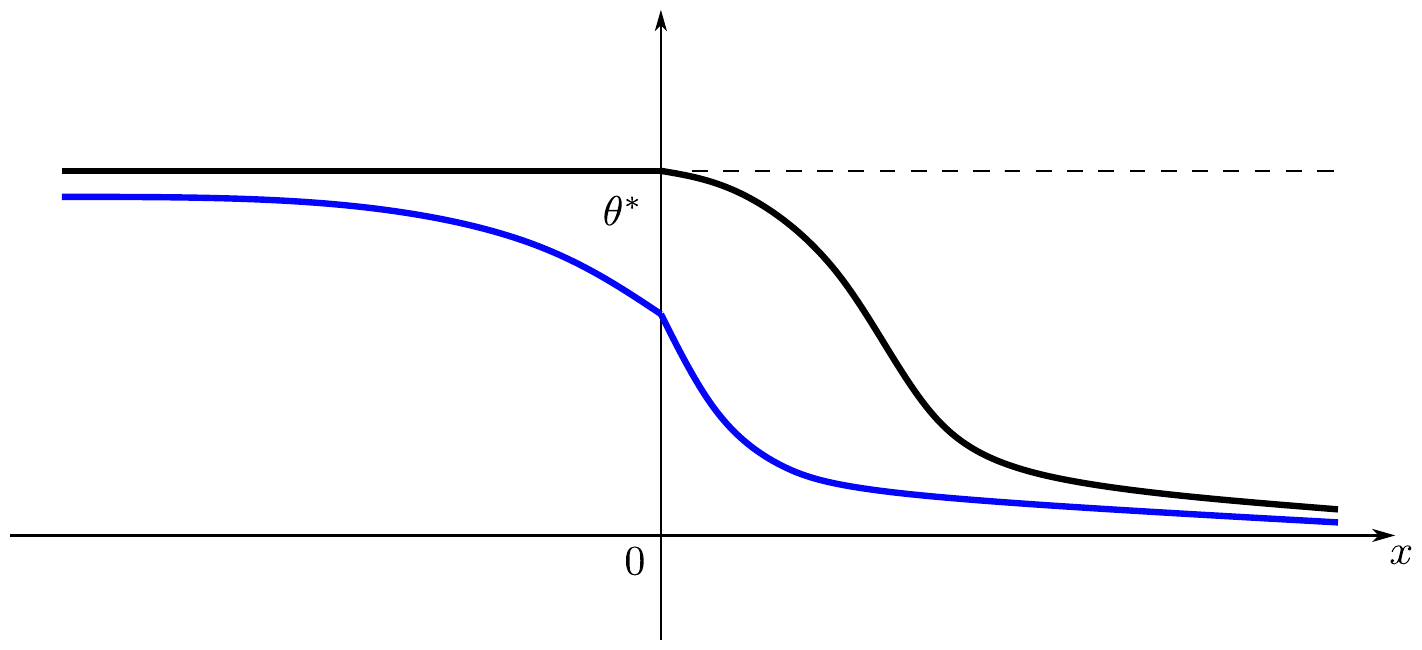}} ~ \subfigure[$K_1>\theta^*$]{\includegraphics[width=1.9in]{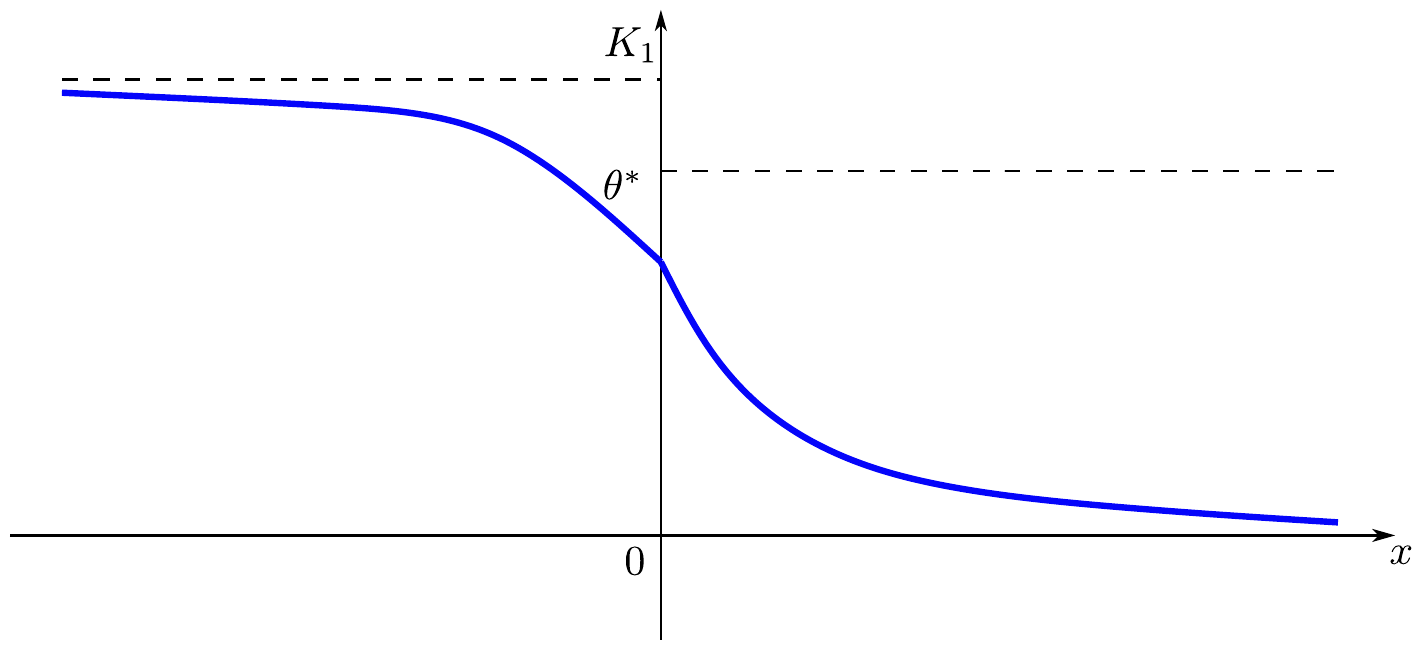}}
\caption{Profile of a steady solution $U$ with $U(-\infty)\!=\!K_1$ and $U(+\infty)\!=\!0$, if $\int_0^{K_2}f_2(s) \mathrm{d}s\!>\!0$.}
\label{bump-sol}
\end{figure}
	
Now suppose that $U'$ has a strict constant sign in $[0^+,+\infty)$, which implies necessarily $U'<0$ in~$[0^+,+\infty)$ since $U(0)>0$ and $U(+\infty)=0$. Then, $U'<0$ in $(-\infty,0^-]$, $U(0)<K_1$ and~\eqref{eqU0} holds as before, while the inequality $U(0)<\theta^*$ holds too from~\eqref{4.14} since $
U(0)\le\theta^*$ and $U'(0^+)<0$ (see the blue curves in Fig.~\ref{bump-sol}). The proof of Proposition~\ref{prop 2.5} is complete.
\end{proof}

\begin{proof}[Proof of Proposition~$\ref{prop 2.6}$]
We first claim that the existence of a positive classical stationary solution~$U$ of~\eqref{1.1} such that $U(-\infty)=K_1$ and $U(+\infty)=0$ is equivalent to the existence of $\xi>0$ such that
\begin{equation}
\label{4.15}
\begin{aligned}
\int^{\xi}_0f_2(s)\mathrm{d}s\le 0,~~\int^{K_1}_{\xi}\!f_1(s)\mathrm{d}s=- \frac{d_1\sigma^2}{d_2}\int^{\xi}_0f_2(s)\mathrm{d}s
\end{aligned}
\end{equation}
and 
\begin{equation}
\label{4.16}
\begin{aligned}
\begin{cases}
0<\xi<K_1 ~~& \displaystyle\text{if}~\int_{0}^{K_2}f_2(s)\mathrm{d}s<0,\cr
0<\xi<\min(K_1,K_2) ~~& \displaystyle\text{if}~\int_{0}^{K_2}f_2(s)\mathrm{d}s=0,\cr
0<\xi\le \theta^* ~~& \displaystyle\text{if}~\int_{0}^{K_2}f_2(s)\mathrm{d}s>0,
\end{cases}
\end{aligned}
\end{equation}
where $\theta^*\in(\theta,K_2)$  is such that $\int_0^{\theta^*}\!f_2(s)\mathrm{d}s=0$ when $\int_{0}^{K_2}\!f_2(s)\mathrm{d}s>0$. Assume this claim for the moment. Under the assumptions of Proposition~\ref{prop 2.6}, it is straightforward to see that such a $\xi>0$ satisfying~\eqref{4.15}--\eqref{4.16} exists by qualitative comparisons of the graphs of the integrals in~\eqref{4.15}, namely:
\begin{itemize}
\item[(i)] in the case $\int_0^{K_2}f_2(s)\mathrm{d}s<0$, since the function $\nu\mapsto\int_\nu^{K_1} f_1(s)\mathrm{d}s$ is continuous decreasing in~$[0,K_1]$ and vanishes at $K_1$, whereas the function $\nu\mapsto-(d_1\sigma^2/d_2)\int_0^\nu f_2(s)\mathrm{d}s$ is continuous in~$[0,K_1]$, positive in $(0,K_1]$ and vanishes at $0$, it follows that there is $\xi\in(0,K_1)$ satisfying~\eqref{4.15};
\item[(ii)] in the case $\int_0^{K_2}f_2(s)\mathrm{d}s=0$ with $K_1<K_2$, since the function $\nu\mapsto\int_\nu^{K_1} f_1(s)\mathrm{d}s$ is continuous decreasing in $[0,K_1]$ and vanishes at $K_1$, whereas the function $\nu\mapsto-(d_1\sigma^2/d_2)\int_0^\nu f_2(s)\mathrm{d}s$ is continuous and positive in $(0,K_2)\supseteq(0,K_1]$ and vanishes at $0$, then there is $\xi\in(0,K_1)$ such that~\eqref{4.15} holds true;
\item[(iii)] in the case $\int_0^{K_2}f_2(s)\mathrm{d}s>0$ with $K_1\le \theta^*$, we consider two subcases. Assume first that $K_1<\theta^*$. Since the function $\nu\mapsto\int_\nu^{K_1} f_1(s)\mathrm{d}s$ is continuous decreasing in $[0,K_1]$ and vanishes at $K_1$, whereas the function $\nu\mapsto-(d_1\sigma^2/d_2)\int_0^\nu f_2(s)\mathrm{d}s$ is continuous and positive in $(0,\theta^*)\supseteq(0,K_1]$ and vanishes at $0$, there exists $\xi\in(0,K_1)$ such that~\eqref{4.15} holds. Lastly, if $K_1=\theta^*$, then $\xi=K_1=\theta^*$ satisfies~\eqref{4.15}.
\end{itemize}  

The conclusion of Proposition~\ref{prop 2.6} will therefore be achieved once the claim is proved. For the proof of the claim, observe first that, if $U$ is a positive classical stationary solution of~\eqref{1.1} such that $U(-\infty)=K_1$ and $U(+\infty)=0$, then the quantity $\xi:=U(0)>0$ necessarily satisfies~\eqref{4.15}--\eqref{4.16} by Proposition~\ref{prop 2.5}. Therefore, we only have to show that the conditions~\eqref{4.15}--\eqref{4.16} yield the existence of such a solution $U$. So let $\xi>0$ satisfy~\eqref{4.15}--\eqref{4.16}. We wish to show that~\eqref{1.1} admits a positive classical stationary solution $U$  such that $U(-\infty)=K_1$ and $U(+\infty)=0$. Set
\begin{equation}\label{defU0pm}\left\{\baa{rcl}
U(0) & = & \xi,\vspace{3pt}\\
U'(0^+) & = & \displaystyle\text{sgn}(U(0)-K_1)\sqrt{-\frac{2}{d_2}\int^{U(0)}_0f_2(s)\mathrm{d}s},\vspace{3pt}\\
U'(0^-) & = & \displaystyle\text{sgn}(U(0)-K_1)\sqrt{\frac{2}{d_1}\int^{K_1}_{U(0)} f_1(s)\mathrm{d}s},\eaa\right.
\end{equation}
where sgn$(t)=t/|t|$ if $t\in\R^*$ and sgn$(0)=0$. Observe that $U'(0^-)=\sigma U'(0^+)$, thanks to~\eqref{4.15} and~\eqref{defU0pm}. Given these values at $0^\pm$, we will now solve the two Cauchy problems in $(-\infty,0]$ and $[0,+\infty)$ and show that these two solutions, glued together, give rise to a solution $U$ of~\eqref{1.1} such that~$U(-\infty)=K_1$ and $U(+\infty)=0$.
	
\vskip 2mm
\noindent{\it Step 1}. Consider first the Cauchy problem in $(-\infty,0]$:
\begin{equation}
\label{4.19}
\begin{cases}
\ d_1 U''+f_1(U)=0,~~x\le0,\cr
\ U(0)=\xi>0,~ \displaystyle U'(0^-)=\text{sgn}(U(0)-K_1)\sqrt{\frac{2}{d_1}\int^{K_1}_{U(0)} f_1(s)\mathrm{d}s}.
\end{cases}
\end{equation}
By the Cauchy-Lipschitz theorem,~\eqref{4.19} has a unique solution $U$ of class $C^2$ and defined in a maximal interval $(\overline x, 0]$ for some $\overline x\in[-\infty,0)$. Multiplying the equation in~\eqref{4.19} by $U'$ and then integrating over $[x,0]$ for any $x\in (\overline x,0]$, and using the definition of $U'(0^-)$, yields
\begin{equation}
\label{4.20}
\frac{d_1}{2}(U'(x^-))^2=\int^{K_1}_{U(x)} f_1(s)\mathrm{d}s~~\text{for all }x\in (\overline x,0].
\end{equation}
We claim that 
\begin{equation}\label{eitherU}
\begin{aligned}
\begin{cases}
\text{ either} & \!\!U>K_1\hbox{ in $(\overline{x},0]$ and }U'>0 ~\text{in}~(\overline x,0^-],\cr
\text{ or} & \!\!U<K_1\hbox{ in $(\overline{x},0]$ and }U'<0 ~\text{in}~(\overline x,0^-],\cr
\text{ or} & \!\!U\equiv K_1\text{ in}~(\overline x,0].\cr
\end{cases}
\end{aligned}
\end{equation}
For this purpose, we first prove that either $U-K_1$ has a strict constant sign in $(\overline x,0]$ or $U\equiv K_1$ in~$(\overline x,0]$. Indeed, assume that there is $x_0\in(\overline x,0]$ such that $U(x_0)=K_1$, then~\eqref{4.20} implies $U'(x_0^-)=0$, hence~$U\equiv K_1$ in $(\overline x,0]$ by the Cauchy-Lipschitz theorem. Assume now that $U-K_1$ has a strict constant sign in $(\overline x,0]$. Then~\eqref{4.20} implies that $U'$ has a strict constant sign in $(\overline x,0^-]$. Together with the definition of $U'(0^-)$ in~\eqref{4.19}, one concludes that, if $U(0)>K_1$ (respectively $U(0)<K_1$), then~$U>K_1$ in $(\overline{x},0]$ and $U'>0$ in $(\overline x,0^-]$ (respectively $U<K_1$ in $(\overline{x},0]$ and $U'<0$ in $(\overline x,0^-]$). Our claim~\eqref{eitherU} is achieved.

From the above observation, we derive that $U$ is monotone and bounded in $(\overline x,0]$ and, from the Cauchy-Lipschitz theorem, that the solution $U$ of~\eqref{4.19} is defined on $(-\infty,0]$, i.e. $\overline x=-\infty$. Let us finally show that $U(-\infty)=K_1$. Let $a:=U(-\infty)$. Then, by~\eqref{eitherU}, $a=K_1$ if $U(0)=K_1$, $K_1\le a<U(0)$ if $U(0)>K_1$, or $U(0)<a\le K_1$ if $0<U(0)<K_1$. Using~\eqref{4.20}, one has
\begin{equation*}
\frac{d_1}{2}(U'(x))^2\to \int^{K_1}_a\!\!f_1(s) \mathrm{d}s~~~\text{as}~x\to -\infty,
\end{equation*}
whence $U'(-\infty)=0$ and $U(-\infty)=a=K_1$, from the assumption~\eqref{f1-kpp} on $f_1$.

\vskip 2mm
\noindent{\it Step 2}. Consider now the Cauchy problem in $[0,+\infty)$:
\begin{equation}
\label{4.21}
\begin{cases}
d_2 U''+f_2(U)=0,~~x\ge0,\cr
\displaystyle U(0)=\xi>0,~U'(0^+)=\text{sgn}(U(0)-K_1)\sqrt{-\frac{2}{d_2}\int^{U(0)}_0\!\!f_2(s)\mathrm{d}s}.
\end{cases}
\end{equation}
The solution of~\eqref{4.21} exists, is of class $C^2$ and is unique in a maximal interval $[0,x^*)$, for some ${x^*}\in(0,+\infty]$. Integrating the equation in~\eqref{4.21} against $U'$ over $[0,x]$ for any $x\in[0,{x^*})$, and using the expression of $U'(0^+)$, yields
\begin{equation}
\label{4.22}
\frac{d_2}{2}(U'(x^+))^2=\left\{\baa{lll}
\displaystyle-\int_0^{U(x)}f_2(s)\mathrm{d}s & \hbox{for all $x\in [0,{x^*})$} & \hbox{if $U(0)\neq K_1$},\vspace{3pt}\\
\displaystyle-\int_{U(0)}^{U(x)}f_2(s)\mathrm{d}s & \hbox{for all $x\in [0,{x^*})$} & \hbox{if $U(0)=K_1$}.\eaa\right.
\end{equation}
Notice from~\eqref{4.16}--\eqref{defU0pm} that the case $U(0)=K_1$ can only occur when $\int_0^{K_2}f_2(s)\mathrm{d}s>0$, and then $\xi=U(0)=K_1=\theta^*$ by~\eqref{4.15}, while $U'(0^+)=0$ by~\eqref{defU0pm}. In that case, by uniqueness, $U$ is equal in $[0,+\infty)$ to the half-bump associated to the reaction $f_2$, that is, $x^*=+\infty$, $U'<0$ in $(0,+\infty)$, $U(+\infty)=0$, $U(0)=\theta^*$ and $U'(0^+)=0$.

Therefore, one can assume in the sequel that $U(0)\neq K_1$. We observe that $U>0$ in $[0,{x^*})$. Indeed, otherwise, there is $x_0\in(0,{x^*})$ such that $U(x_0)=0$, hence $U'(x_0)=0$ by~\eqref{4.22} and $U\equiv0$ in $[0,x^*)$ by the Cauchy-Lipschitz theorem. This would contradict $U(0)>0$. Thus, $U>0$ in $[0,{x^*})$. Next, we solve~\eqref{4.21} by dividing into three cases according to the sign of the mass $\int_0^{K_2}f_2(s)\mathrm{d}s$.
		
\vskip 2mm
\noindent{\it Case 1: $\int_0^{K_2}f_2(s)\mathrm{d}s< 0$}. One infers from~\eqref{4.16} that $U(0)=\xi<K_1$ and thus $U'(0^+)<0$ by~\eqref{4.21}. Moreover, one deduces from~\eqref{4.22} that  $U'$ does not change sign in $[0^+,{x^*})$. Therefore, $U'<0$ in~$[0^+,{x^*})$.  Since $U>0$ in $[0,{x^*})$, one has $0<U<U(0)<K_1$ in $(0,{x^*})$, whence ${x^*}=+\infty$.	Define~$b:=U(+\infty)\ge 0$. From~\eqref{4.22}, it follows that
\begin{equation*}
\frac{d_2}{2}(U'(x))^2 \to-\int_0^b f_2(s)\mathrm{d}s~~~\text{as}~x\to+\infty, 
\end{equation*}
hence, $U'(+\infty)=0$ and $U(+\infty)=b=0$.  
			
\vskip 2mm
\noindent{\it Case 2: $\int_0^{K_2}f_2(s)\mathrm{d}s= 0$}. It follows from~\eqref{4.16} that $0<U(0)=\xi<\min(K_1,K_2)$ and thus $U'(0^+)<0$ by~\eqref{4.21}. We now show that $U'<0$ in $(0,{x^*})$. Assume by contradiction that there is $x_0=\min\{x\in (0,{x^*}): U'(x)=0\}\ \in(0,x^*)$. Then $0<U(x_0)<U<U(0)<\min(K_1,K_2)$ in $(0,x_0)$. On the other hand, taking $x=x_0$ in~\eqref{4.22} and using $U(x_0)>0$ yields $U(x_0)=K_2$, a contradiction. Thus,~$U'<0$ in $[0^+,{x^*})$, whence $0<U<U(0)$ in $(0,{x^*})$  and ${x^*}=+\infty$. As in Case 1, one concludes that $U(+\infty)=0$.
		
\vskip 2mm
\noindent{\it Case 3: $\int_0^{K_2}f_2(s)\mathrm{d}s>0$}. We let $\theta^*\in(\theta,K_2)$ be such that $\int_0^{\theta^*}\!f_2(s)\mathrm{d}s=0$ and $Q\in(\theta^*,+\infty]$ be as in~\eqref{defQ}. From~\eqref{4.16}, it is seen that $0<U(0)=\xi\le \theta^*$. Moreover, we observe from~\eqref{4.22} that~$U(x)\notin (\theta^*,Q)$ for every $x\in[0,{x^*})$, hence $0<U\le\theta^*$ in $[0,x^*)$ and $x^*=+\infty$. We recall that the bistable equation $d_2 u''+f_2(u)=0$ in $\mathbb{R}$ admits an even bump-like solution $u$, satisfying
$$u(0)=\theta^*,~ u'(0)=0, ~u'<0 ~\text{in}~(0,+\infty),~u(\pm\infty)=0.$$	
\begin{enumerate}[(i)]
\item Suppose first that $K_1<U(0)=\xi\ (\le \theta^*)$, whence $U'(0^+)>0$ and $K_1<U(0)=\xi<\theta^*$ by~\eqref{4.15} and~\eqref{4.21}. If $U'>0$ in $[0^+,+\infty)$, then $U(+\infty)$ exists and belongs to $(0,\theta^*]$, and $U'(+\infty)=U''(+\infty)=0$ from standard elliptic estimates. Together with~\eqref{4.22}, one gets that~$U(+\infty)=\theta^*$, hence $U''(x)=-f_2(U(x))/d_2\to-f_2(\theta^*)/d_2<0$ as $x\to+\infty$, a contradiction. Therefore, $U$ has a critical point in $(0,+\infty)$, that is, $x_0=\min\{x>0: U'(x)=0\}$ is a well defined positive real number, and one has $U'>0$ in $(0,x_0)$ and $U'(x_0)=0$. Combining~\eqref{4.22} with the fact that $0<U\le \theta^*$ in $[0,+\infty)$, one infers that $U(x_0)=\theta^*$. Therefore, by the uniqueness of the solution to the Cauchy problem, $U$ has to be the bump-like solution $u(\cdot-x_0)$ in $[0,+\infty)$. Namely,~$U(x_0)=\theta^*$, $U'(x_0)=0$, $U'<0$ in $(x_0,+\infty)$ and $U(+\infty)=0$.
\item Finally, let us assume $U(0)=\xi<K_1$. Then, $U'(0^+)<0$ by~\eqref{4.15} and~\eqref{4.21}. Remember also that $0<U(0)=\xi\le \theta^*$. We now show that $U'<0$ in $(0,+\infty)$. If not, then there is $x_0=\min\{x>0: U'(x)=0\}>0$ such that $U'(x_0)=0$ and $0<U(x_0)<U<U(0)\le\theta^*$ in~$(0,x_0)$. It follows from~\eqref{4.22} that  $0=(d_2/2)(U'(x_0))^2=-\int_0^{U(x_0)}f_2(s)\mathrm{d}s>0$, a contradiction. Consequently, $U'<0$ in $(0,+\infty)$ and the argument used in Case~1 yields $U(+\infty)=0$.
\end{enumerate}		
	
Gluing the solutions of~\eqref{4.19} and~\eqref{4.21} proves the existence of the desired stationary solution $U$ of~\eqref{1.1} such that $U(-\infty)=K_1$ and $U(+\infty)=0$. Therefore, our claim at the beginning of the proof is achieved and the proof of Proposition~\ref{prop 2.6} is thereby complete. 
\end{proof}

\begin{remark}\label{rem41}{\it Based on the above proof, it is easy to find examples of functions $f_{1,2}$ satisfying~\eqref{f1-kpp}--\eqref{f2-bistable} and $\int_0^{K_2}f_2(s)\mathrm{d}s>0$ such that~\eqref{1.1} has no stationary solution $U$ connecting $K_1$ and $0$. For instance, let us take $d_1=d_2=\sigma=1$, and set
\begin{equation*}
f_1(u)=u(K_1-u),~~f_2(u)=u(K_2-u)(u-\theta)
\end{equation*}	
with $K_1=K_2=4$ and $\theta=1$. It is straightforward to check that~\eqref{4.15} $($with $\xi>0$$)$ yields $\xi>4$, contradicting the condition $\xi\le \theta^*<4$ implied by~\eqref{4.16}. Therefore,~\eqref{4.15} and~\eqref{4.16} can not be fulfilled simultaneously, and there is no positive classical stationary solution $U$ of~\eqref{1.1} such that~$U(-\infty)=K_1$ and $U(+\infty)=0$.}
\end{remark}

\begin{proof}[Proof of Proposition~$\ref{prop 2.7}$]
The strategy is very similar to that of Proposition~\ref{prop 2.6}. For completeness, we sketch the proof. Here, in addition to~\eqref{f1-kpp}--\eqref{f2-bistable}, we assume that $\int_0^{K_2}f_2(s)\mathrm{d}s\ge0$. We first claim that the existence (respectively the existence and uniqueness) of a positive classical stationary solution~$V$ of~\eqref{1.1} satisfying $V(-\infty)=K_1$ and $V(+\infty)=K_2$ is equivalent to the existence (respectively the existence and uniqueness) of $\xi>0$ such that 
\begin{equation}
\label{4.23}
\begin{aligned}
\begin{cases}
\ \xi=K_1=K_2&\text{if}~K_1=K_2,\cr
\ \min(K_1,K_2)<\xi<\max(K_1,K_2)&\text{if}~K_1\neq K_2,
\end{cases}
\end{aligned}
\end{equation}
and 
\begin{equation}
\label{4.24}
\int_\xi^{K_1}\!f_1(s)\mathrm{d}s=\frac{d_1\sigma^2}{d_2}\int_\xi^{K_2}\!f_2(s)\mathrm{d}s.
\end{equation}

In this paragraph, we observe that such $\xi>0$ satisfying~\eqref{4.23}--\eqref{4.24} always exists, and is unique if $K_1\ge\theta$. To check this, it is sufficient to consider the case of $K_1\neq K_2$. Suppose $K_1<K_2$. Observe that the function $\nu\mapsto\int_\nu^{K_1}f_1(s)\mathrm{d}s$ is continuous increasing in $[K_1,K_2]$ and vanishes at $K_1$, whereas the function $\nu\mapsto(d_1\sigma^2/d_2)\int_\nu^{K_2} f_2(s)\mathrm{d}s$ is continuous positive in $[K_1,K_2)$, vanishes at $K_2$, and is either increasing in $[K_1,\theta]$ and decreasing in $[\theta,K_2]$ (if $K_1<\theta$), or decreasing in $[K_1,K_2]$ (if $K_1\ge \theta$). Therefore, there is $\xi\in(K_1,K_2)$ such that~\eqref{4.24} is satisfied, and $\xi$ is unique if $K_1\ge\theta$. Consider now the case of $K_2<K_1$. Since the function $\nu\mapsto\int_\nu^{K_1}f_1(s)\mathrm{d}s$ is continuous decreasing in $[K_2,K_1]$ and vanishes at $K_1$, whereas the function $\nu\mapsto(d_1\sigma^2/d_2)\int_\nu^{K_2} f_2(s)\mathrm{d}s$ is continuous increasing in $[K_2,K_1]$ and vanishes at $K_2$, it follows that there is a unique $\xi\in (K_2,K_1)$ satisfying~\eqref{4.24}.

Therefore, it remains to prove our claim, whose proof is divided into two steps, each corresponding to one implication of the equivalence. 
	
\vskip 2mm
\noindent{\it Step 1: necessary condition for the existence of $V$}. Suppose $V$ is a positive classical  stationary solution of~\eqref{1.1} satisfying $V(-\infty)=K_1$ and $V(+\infty)=K_2$. Multiplying $d_1 V''+f_1(V)=0$ by $V'$ and  integrating the resulting equation over $(-\infty,x]$ for any $x\in (-\infty,0]$ yields 
\begin{equation}
\label{4.25} 
\frac{d_1}{2}(V'(x^-))^2=\int_{V(x)}^{K_1}f_1(s)\mathrm{d}s\ \hbox{ for all }x\le0.
\end{equation}
Similarly, one also derives that
\begin{equation}
\label{4.26}
\frac{d_2}{2}(V'(x^+))^2=\int^{K_2}_{V(x)}f_2(s)\mathrm{d}s\ge 0\ \hbox{ for all }x\ge0.
\end{equation}
Following the same argument as for~\eqref{4.20}-\eqref{eitherU}, one derives from~\eqref{4.25} that $V$ is monotone in patch~$1$ and, more precisely,
$$\begin{aligned}
\begin{cases}
\text{ either}&\!\!V>K_1\hbox{ in $(-\infty,0]$ and }V'>0 ~\text{in}~(-\infty,0^-],\cr
\text{ or}&\!\!V<K_1\hbox{ in $(-\infty,0]$ and }V'<0 ~\text{in}~(-\infty,0^-],\cr
\text{ or}&\!\!V\equiv K_1 ~\text{in}~(-\infty,0].\cr
\end{cases}
\end{aligned}$$
Similarly, since $\int_\nu^{K_2}f_2(s)\mathrm{d}s>0$ for all $\nu\in(0,K_2)\cup(K_2,+\infty)$, it follows from~\eqref{4.26} that $V$ is also monotone in patch 2 and, more precisely,
$$\begin{aligned}
\begin{cases}
\text{ either}&\!\!V>K_2\hbox{ in $[0,+\infty)$ and }\ V'<0~\text{in}~[0^+,+\infty),\cr
\text{ or}&\!\!V<K_2\hbox{ in $[0,+\infty)$ and }\ V'>0~\text{in}~[0^+,+\infty),\cr
\text{ or}&\!\!V\equiv K_2 ~\text{in}~[0,+\infty).\cr
\end{cases}
\end{aligned}$$
Using $V'(0^-)=\sigma V'(0^+)$, one then infers that $V$ is monotone in $\mathbb{R}$ and, more precisely,
$$\begin{aligned}
\begin{cases}
\ V\equiv K_1=K_2~~&\text{if}~~K_1=K_2,\cr
\ \min(K_1,K_2)<V<\max(K_1,K_2)\text{ and sgn}(V')=\text{sgn}(V(0)-K_1)~~&\text{if}~~K_1\neq K_2.
\end{cases}
\end{aligned}$$
Moreover, thanks to~\eqref{4.25}--\eqref{4.26}, $V(0)$ satisfies
$$\int_{V(0)}^{K_1}f_1(s)\mathrm{d}s=\frac{d_1\sigma^2}{d_2}\int_{V(0)}^{K_2}f_2(s)\mathrm{d}s.$$
Hence, the quantity $\xi=V(0)$ satisfies~\eqref{4.23}--\eqref{4.24}.
	
\vskip 2mm
\noindent{\it Step 2: sufficient condition for the existence of $V$}. Assume that there is $\xi>0$ satisfying~\eqref{4.23}--\eqref{4.24}. If $K_1=K_2$, then $\xi=K_1=K_2$ and the function $V\equiv K_1=K_2$ obviously satisfies~\eqref{1.1} with $V(-\infty)=K_1=K_2=V(+\infty)$. One can then assume in the sequel that $K_1\neq K_2$. Let us set $V(0)=\xi\in(\min(K_1,K_2),\max(K_1,K_2))$ and define
$$V'(0^-)=\text{sgn}(V(0)-K_1)\sqrt{\frac{2}{d_1}\int^{K_1}_{V(0)} f_1(s)\mathrm{d}s},$$
and 
$$V'(0^+)=\text{sgn}(V(0)-K_1)\sqrt{\frac{2}{d_2}\int_{V(0)}^{K_2}f_2(s)\mathrm{d}s}.$$
It is obvious to see that $V'(0^-)=\sigma V'(0^+)$, thanks to~\eqref{4.24}. Notice also that $V(0)=\xi\neq K_1$, here.
	
\vskip 2mm
\noindent{\it Step 2.1}. As for~\eqref{4.19}, the solution $V$ of the Cauchy problem
\begin{equation}
\label{4.27}
\begin{cases}
\ d_1 V''+f_1(V)=0,~~x\le0, \cr
\ V(0)=\xi>0,~\displaystyle V'(0^-)=\text{sgn}(V(0)-K_1)\sqrt{\frac{2}{d_1}\int^{K_1}_{V(0)} f_1(s)\mathrm{d}s}.
\end{cases}
\end{equation}
is defined in $(-\infty,0]$ and satisfies~\eqref{eitherU} with $V$ instead of $U$ and $\overline{x}=-\infty$, that is,
$$\begin{aligned}
\begin{cases}
\text{ either}&\!\!\!V>K_1\hbox{ in $(-\infty,0]$ and }V'>0~\text{in}~(-\infty,0^-],\cr
\text{ or}&\!\!\!V<K_1\hbox{ in $(-\infty,0]$ and }V'<0~\text{in}~(-\infty,0^-],\cr
\end{cases}
\end{aligned}$$
and $V(-\infty)=K_1$.
	
\vskip 2mm
\noindent{\it Step 2.2}. Let $V$ denote the solution of
\begin{equation}
\label{4.28}
\begin{cases}
d_2 V''+f_2(V)=0,~~x\ge0,\cr
V(0)=\xi>0,~\displaystyle V'(0^+)=\text{sgn}(V(0)-K_1)\sqrt{\frac{2}{d_2}\int_{V(0)}^{K_2} f_2(s)\mathrm{d}s}.
\end{cases}
\end{equation}
Notice that, here, $\min(K_1,K_2)<V(0)<\max(K_1,K_2)$, hence $V'(0^+)\neq0$ since $\int_\nu^{K_2}f_2(s)\mathrm{d}s>0$ for all $\nu\in\R\setminus\{0,K_2\}$. The Cauchy-Lipschitz theorem implies that there is a unique solution of~\eqref{4.28} defined in a maximal interval $[0,x^*)$ for some $x^*\in(0,+\infty]$.  Multiplying the equation in~\eqref{4.28} by $V'$ and then integrating over $[0,x]$ for any $x\in[0,x^*)$, and using the formula of $V'(0^+)$, yields
\begin{equation}
\label{4.29}
\frac{d_2}{2}(V'(x^+))^2=\int^{K_2}_{V(x)}f_2(s)\mathrm{d}s\ \hbox{ for all }x\in [0,x^*).
\end{equation}	
Moreover, we claim that $V'$ has a strict constant sign in $[0^+,x^*)$. Indeed, otherwise, there is $x_0\in[0^+,x^*)$ such that $V'(x_0)=0$, and~\eqref{4.29} implies that 
$$\begin{aligned}
\begin{cases}
\ V(x_0)=K_2~~&\displaystyle\text{if}~~\int_0^{K_2} f_2(s)\mathrm{d}s>0,\cr
\ V(x_0)=K_2~\text{or}~0~~&\displaystyle\text{if}~~\int_0^{K_2} f_2(s)\mathrm{d}s=0.
\end{cases}
\end{aligned}$$
Thus, one would derive $V\equiv K_2$ or $V\equiv 0$ in $[0,x^*)$ by the Cauchy-Lipschitz theorem, contradicting $\min(K_1,K_2)<V(0)=\xi<\max (K_1,K_2)$. Thus, $V'$ has a constant strict sign in $[0^+,x^*)$. Hence,~$V(x)\neq K_2$ for every $x\in[0,x^*)$, by~\eqref{4.29}. Therefore, we conclude that
$$\begin{aligned}
\begin{cases}
\text{ if}~K_1<V(0)<K_2,\hbox{ then }V'>0~\text{in}~[0^+,x^*)\hbox{ and }K_1<V<K_2~\text{in}~[0,x^*),\cr
\text{ if}~K_2<V(0)<K_1,\hbox{ then }V'<0~\text{in}~[0^+,x^*)\hbox{ and }K_2<V<K_1~\text{in}~[0,x^*).
\end{cases}
\end{aligned}$$
Both cases imply that $x^*=+\infty$. Defining $V(+\infty)=a$, one has $K_1\le a\le K_2$ and~\eqref{4.29} implies
$$\frac{d_2}{2}(V'(x))^2\to \int^{K_2}_{a}f_2(s)\mathrm{d}s~~~\text{as}~x\to+\infty,$$
hence $V'(+\infty)=0$ and $V(+\infty)=a=K_2$.

\vskip 2mm
\noindent{\it Step 2.3: conclusion}. By gluing the solutions of the above two Cauchy problems~\eqref{4.27} and~\eqref{4.28}, one obtains the existence of a monotone positive classical stationary solution $V$ of~\eqref{1.1} such that $V(-\infty)=K_1$ and $V(+\infty)=K_2$. Lastly, if $K_1\ge\theta$, then we have already seen that $\xi>0$ solving~\eqref{4.23}--\eqref{4.24} is unique, hence the above proof shows that $V(0)=\xi$ is unique and the positive classical stationary solution $V$ of~\eqref{1.1} such that $V(-\infty)=K_1$ and $V(+\infty)=K_2$ is itself unique. The proof of Proposition~\ref{prop 2.7} is thereby complete.
\end{proof}

%%%%%%%%%%%%%%%%%%%%%%%%%%%%%%%%%%%%%%%%%%%%%%%%%%%%

\subsection{Blocking in the bistable patch 2: proof of Theorem~\ref{thm2.8}}
\label{Sec 4.3}

In this section, we study the qualitative behavior of the solution $u$ to~\eqref{1.1} under the KPP-bistable assumptions~\eqref{f1-kpp}--\eqref{f2-bistable}. We carry out the proof of Theorem~\ref{thm2.8} on various sufficient conditions for blocking in the bistable patch~2. The proof is based, among other things, on a comparison with some barriers, such as a traveling front with negative or zero speed (up to some exponentially small terms, when $\int_0^{K_2}\!f_2(s)\mathrm{d}s\le0$), or a stationary solution connecting $K_1$ to $0$ (when $\|u_0\|_{L^1(\R)}$ is small enough).

\begin{proof}[Proof of Theorem~$\ref{thm2.8}$]
(i) We first assume that
$$\int_0^{K_2}f_2(s)\mathrm{d}s<0.$$
Let $u$ be the solution to the Cauchy problem~\eqref{1.1} with a nonnegative continuous and compactly supported initial datum $u_0\not\equiv 0$. The strategy of the proof consists in constructing a supersolution which blocks the solution $u(t,x)$ for all times $t\ge0$ as $x\to +\infty$. Set $M:=\max\big(K_1,K_2,\|u_0\|_{L^\infty(\R)}\big)+1$. Since the function~$f_2$ satisfies~\eqref{f2-bistable} with $\int_0^{K_2}f_2(s)\mathrm{d}s<0$, there is a $C^1(\R)$ function $\overline{f}_2$ such that $\overline{f}_2\ge f_2$ in~$\R$, $\overline{f}_2(0)=\overline{f}_2(\theta)=\overline{f}_2(M)=0$, $\overline{f}_2'(0)<0$, $\overline{f}_2'(M)<0$, $\overline{f}_2>0$ in~$(-\infty,0)\cup(\theta,M)$, $\overline{f}_2<0$ in~$(0,\theta)\cup(M,+\infty)$, and $\int_0^M\overline{f}_2(s)\mathrm{d}s<0$ (it is even possible to choose $\overline{f}_2$ so that $\overline{f}_2=f_2$ in $(-\infty,K_2-\delta]$ for some small $\delta>0$). There is then a decreasing front profile $\overline{\phi}$ solving~\eqref{2.5} with $\overline{f}_2$ and $M$ instead of $f_2$ and $K_2$, and with negative speed $\overline{c}_2$ instead of $c_2$. Since $\overline{\phi}(-\infty)=M>\max(K_1,\|u_0\|_{L^\infty(\R)})$ and $u_0$ is compactly supported, one can then choose $A>0$ large enough so that $u_0(x)\le\overline{\phi}(x-A)$ for all~$x\ge0$, $u_0(x)\le\overline{\phi}(-A)$ for all $x\le0$, and $K_1\le\overline{\phi}(-A)$. Set, for $(t,x)\in[0,+\infty)\times\R$,
$$\overline{u}(t,x)=\left\{\baa{ll} \overline{\phi}(x-A) & \hbox{if }x\ge0,\vspace{3pt}\\ \overline{\phi}(-A) & \hbox{if }x<0.\eaa\right.$$
Since $f_1(\overline{\phi}(-A))\le0$ by~\eqref{f1-kpp}, since $d_2\overline{\phi}''+f_2(\overline{\phi})\le d_2\overline{\phi}''+\overline{f}_2(\overline{\phi})=-\overline{c}_2\overline{\phi}'<0$ in $\R$ and since $\overline{\phi}'(-A)<0$, it follows that $\overline{u}$ is a supersolution of~\eqref{1.1} in the sense of Definition~\ref{def2}, while $u_0\le\overline{u}(0,\cdot)$ in $\R$. Therefore, Proposition~\ref{prop 1.3} implies that $u(t,x)\le\overline{u}(t,x)$ for all $(t,x)\in[0,+\infty)\times\R$. Since $u$ is nonnegative and $\overline{\phi}(+\infty)=0$, this immediately yields the blocking property~\eqref{blocking}. 
  
\vskip 0.2cm
\noindent{}(ii) We then assume that $\int_0^{K_2}f_2(s)\mathrm{d}s=0$ and $K_1<K_2$. First, it is convenient to  introduce some parameters.  Let $\varep>0$ be such that
\begin{equation}
\label{4.37}
0<\varep<\min\left(\frac{|f_2'(0)|}{2},\frac{ |f_2'(K_2)|}{2}\right),\ \ f_2'\le \frac{f_2'(0)}{2}~\text{in}~[0,2\varep],~~f_2'\le \frac{f_2'(K_2)}{2}~\text{in}~[K_2-\varep,K_2+\varep].
\end{equation}
Choose $C>0$  large enough such that 
\begin{equation}
\label{4.38}
\phi\ge K_2-\varep~\text{in}~(-\infty,-C]\ \hbox{ and }\ \phi\le \varep~\text{in}~[C,+\infty).
\end{equation}
As the front profile $\phi$ solving~\eqref{2.5} is such that $\phi'$ is negative and continuous, there is $\kappa>0$ such that 
\begin{equation}
\label{4.39}
-\phi'\ge \kappa>0~~\text{in}~[-C,C].
\end{equation}
Finally, pick $\rho>0$ be such that 
\begin{equation}
\label{4.40}
\kappa\rho \ge \varep+\max_{[0,K_2+\varep]}|f_2'|.
\end{equation}
 
Let $u$ be the solution to the Cauchy problem~\eqref{1.1} with a nonnegative continuous and compactly supported initial datum $u_0\not\equiv 0$ and let $V$ be a positive monotone classical stationary solution of~\eqref{1.1} such that $V(-\infty)=K_1$ and  $V(+\infty)=K_2$, given in Proposition~\ref{prop 2.7}. Denote $w$ the solution to~\eqref{1.1} with initial datum $w(0,\cdot)=M:=\max\big(K_2,\Vert u_0\Vert_{L^\infty(\mathbb{R})}\big)$. As in the proof of the first part of Theorem~\ref{thm2.3}, Proposition~\ref{prop 1.3} implies that $w$ is nonincreasing in time, and that $0<u(t,x)< w(t,x)\le M$ and  $V(x)<w(t,x)$ for all $t> 0$ and $x\in\mathbb{R}$. From the Schauder estimates of Proposition~\ref{pro1}, it follows that $w(t,x)$ converges as $t\to+\infty$, locally uniformly in $x\in\mathbb{R}$, to a stationary solution $q$ of~\eqref{1.1}, such that $M\ge q(x)\ge V(x)\ge K_1$ for all $x\in\R$ and
\begin{equation}
\label{4.41}
\limsup_{t\to+\infty} u(t,\cdot)\le q~~\text{locally uniformly in}~\mathbb{R}.
\end{equation}
As shown in Theorem~\ref{thm2.4}, one also has $q(-\infty)=K_1$. On the other hand, since $f_2<0$ in $(K_2,+\infty)$ and $q$ is bounded, one easily infers that $\limsup_{x\to+\infty}q(x)\le K_2$. Furthermore, as in the proof of Propositions~\ref{prop2.1} and~\ref{prop 2.5}, the function $q$ is monotone in $(-\infty,0]$, and $q'$ has a constant strict sign in~$(-\infty,0^-]$ unless $q\equiv K_1$ in $(-\infty,0]$. Thus, if one would assume that $\sup_\R q>K_2\,(>K_1)$, there would exist $x_0\in(0,+\infty)$ such that $q(x_0)=\sup_\R q>K_2$, which is impossible since $f_2<0$ in $(K_2,+\infty)$. Therefore, $q\le K_2$ in $\R$ and even $q<K_2$ in $\R$ since the constant $K_2$ is a supersolution of~\eqref{1.1} and the stationary solution $q$ can not be identically equal to~$K_2$.
 
Similarly, we claim that
\be\label{claimK2}
\limsup_{A\to+\infty}\Big(\sup_{t\ge A,\,x\ge A}u(t,x)\Big)\le K_2.
\ee
Indeed, otherwise, since $u$ is bounded in $[0,+\infty)\times\R$, there are $\overline{K}_2\in(K_2,+\infty)$ (with~$\overline{K}_2$ the above limsup) and two sequences $(t_n)_{n\in\N}$ and $(x_n)_{n\in\N}$ diverging to $+\infty$ such that $u(t_n,x_n)\to\overline{K}_2$ as $n\to+\infty$ and $\limsup_{n\to+\infty}u(t_n+t,x_n+x)\le\overline{K}_2$ for all $(t,x)\in\R\times\R$. From parabolic estimates, the functions $(t,x)\mapsto u(t_n+t,x_n+x)$ converge in $C^{1;2}_{t;x;loc}(\R\times\R)$, up to extraction of a subsequence, to a bounded classical solution $u_\infty$ of $(u_\infty)_t=d_2(u_\infty)_{xx}+f_2(u_\infty)$ in $\R\times\R$ with $u_\infty\le u_\infty(0,0)=\overline{K}_2$ in $\R\times\R$. The negativity of $f_2(\overline{K}_2)$ leads to a contradiction. Therefore,~\eqref{claimK2} holds.

Let then $X>0$ be large enough so that
\be\label{utxX}
u(t,x)\le K_2+\frac{\varep}{2}\ \hbox{ for all $t\ge X$ and $x\ge X$}.
\ee
Thanks to~\eqref{4.41} and $q(X)<K_2$, there is $T\ge X$ so large that
\begin{equation}
\label{4.43}
\sup_{t\ge T}\,u(t,X)<K_2.
\end{equation}
Remember that the front profile $\phi$ associated with the reaction $f_2$, given in~\eqref{2.5} with speed $c_2=0$ (since $\int_0^{K_2}f_2(s)\mathrm{d}s=0$), satisfies $\phi(-\infty)=K_2$. Due to~\eqref{utxX}--\eqref{4.43} and the Gaussian upper bound of $u(t,x)$ for $|x|$ large at each time $t>0$ derived in Lemma~\ref{lemma1.3}, together with the exponential lower bound of $\phi(s)$ as $s\to+\infty$ in~\eqref{2.6}, there exists $B>0$ large enough such that  
\begin{equation}
\label{4.44}
u(T,x)\le \phi(x-X-B-C)+\varep~\text{for all}~x\ge X,\ \hbox{ and }\ \sup_{t\ge T}\,u(t,X)\le\phi(-B-C).
\end{equation}

Define 
$$\overline u(t,x)=\phi(\zeta(t,x))+\varep e^{-\varep(t-T)}~~\text{for}~t\ge T~\text{and}~ x\ge X,$$
where $\zeta(t,x)=x-X+\rho e^{-\varep(t-T)}-\rho-B-C$. We wish to show that $\overline u(t,x)$ is a supersolution of the equation $u_t=d_2u_{xx}+f_2(u)$ for $t\ge T$ and $x\ge X$. First of all, at time $t=T$, one has
$$\overline u(T,x)=\phi(x-X-B-C)+\varep \ge u(T,x)$$
for $x\ge X$, thanks to~\eqref{4.44}. Furthermore, for $t\ge T$, $\overline u(t,X)=\phi(\rho e^{-\varep(t-T)}-\rho-B-C)+\varep e^{-\varep(t-T)}\ge\phi(-B-C)\ge u(t,X)$ by~\eqref{4.44} again. It then remains to check that
$$\mathcal{N}\overline u(t,x):=\overline u_t(t,x)-d_2 \overline u_{xx}(t,x)-f_2(\overline u(t,x))\ge 0$$
for $t\ge T$ and $x\ge X$. A direct computation leads to
\begin{align*}
\mathcal{N}\overline u(t,x)=f_2(\phi(\zeta(t,x)))-f_2(\overline u(t,x))-\phi'(\zeta(t,x))\rho\varep e^{-\varep(t-T)}-\varep^2 e^{-\varep(t-T)}.
\end{align*}
We divide the proof into three cases:
\begin{itemize}
\item if $\zeta(t,x)\le -C$, one has $K_2+\varep\ge\overline u(t,x)\ge \phi(\zeta(t,x))\ge K_2-\varep$ by~\eqref{4.38};  one then derives from~\eqref{4.37} and the negativity of $\phi'$ that
\begin{align*}
\mathcal{N}\overline u(t,x)\ge -\frac{f_2'(K_2)}{2}\varep e^{-\varep(t-T)}-\varep^2 e^{-\varep(t-T)}= \Big(-\frac{f_2'(K_2)}{2}-\varep\Big)\varep e^{-\varep(t-T)}\ge 0;
\end{align*}
\item if $\zeta(t,x)\ge C$, then $0<\phi(\zeta(t,x))\le \varep$ by~\eqref{4.38} and $0<\overline u(t,x)\le 2\varep$; it follows from~\eqref{4.37} and the negativity of $\phi'$ that
\begin{align*}
\mathcal{N}\overline u(t,x)\ge -\frac{f_2'(0)}{2}\varep e^{-\varep(t-T)}-\varep^2 e^{-\varep(t-T)}= \Big(-\frac{f_2'(0)}{2}-\varep\Big)\varep e^{-\varep(t-T)}\ge 0;
\end{align*}
\item eventually, if $-C\le \zeta(t,x)\le C$, then $-\phi'(\zeta(t,x))\ge \kappa>0$ by~\eqref{4.39}, and~\eqref{4.40} then yields
\end{itemize}
\vskip -4mm
\begin{align*}
\mathcal{N}\overline u(t,x)\ge -\!\max_{[0,K_2+\varep]}\!|f_2'|\,\varep e^{-\varep(t-T)}\!+\!\kappa\rho\varep e^{-\varep(t-T)}\!-\!\varep^2 e^{-\varep(t-T)}\ge \Big(\kappa\rho-\!\max_{[0,K_2+\varep]}\!|f_2'|\!-\!\varep\Big)\varep e^{-\varep(t-T)}\ge 0.
\end{align*}

In conclusion, the function $\overline u$ is a supersolution of $u_t=d_2 u_{xx}+f_2(u)$ for $t\ge T$ and $x\ge X$. The maximum principle implies that 
$$u(t,x)\le\overline{u}(t,x)=\phi(x-X+\rho e^{-\varep(t-T)}-\rho-B-C)+\varep e^{-\varep(t-T)}~~\text{for all}~t\ge T~\text{and}~x\ge X.$$
Consequently, $\limsup_{x\to+\infty}\!\!\big(\sup_{t\ge T}u(t,x)\big)\!\le\!\varep$. On the other hand,   Lemma~\ref{lemma1.3} implies that $u(t,x)\!\to\!0$ as $x\to +\infty$ locally uniformly in $t\ge 0$. Since $\varep>0$ can be chosen arbitrarily small, one gets that $u$ is blocked in patch 2 and satisfies~\eqref{blocking}. This completes the proof of part~(ii) of Theorem~\ref{thm2.8}.

\vskip 0.2cm
\noindent{}(iii) We here assume that $K_1<\theta$. Let $u$ be the solution to~\eqref{1.1} with a nonnegative continuous and compactly supported initial datum $u_0\not\equiv 0$ satisfying $u_0<\theta$ in $\R$. The constant function equal to~$M:=\max\big(K_1,\|u_0\|_{L^\infty(\R)}\big)$ is a supersolution of~\eqref{1.1} in the sense of Definition~\ref{def2} (since $f_1(M)\le0$ and $f_2(M)<0$), and Proposition~\ref{prop 1.3} then implies that
\be\label{uM}
0<u(t,x)<\max\big(K_1,\Vert u_0\Vert_{L^\infty(\mathbb{R})}\big)<\theta\ \hbox{ for all $t\ge 0$ and $x\in\mathbb{R}$}.
\ee
Choose $\varep\in(0,K_2-\theta)$ and let $g_2$ be a $C^1(\R)$ function such that $g_2=f_2$ in $(-\infty,\theta]$,  $g_2>0$ in $(\theta,\theta+\varep)$, $g_2(\theta+\varep)=0$, $g_2'(\theta+\varep)<0$, $g_2<0$ in $(\theta+\varep,+\infty)$, and $\int_0^{\theta+\varep}g_2(s)\mathrm{d}s<0$. Let $z$ be the solution to~\eqref{1.1} in which $f_2$ is replaced by $g_2$, starting from the initial datum $u_0$. By comparison, and using~\eqref{uM}, one has $u(t,x)=z(t,x)$ for all $t\ge 0$ and $x\in\mathbb{R}$. Thanks to part~(i) of Theorem~\ref{thm2.8} applied to the solution $z$ with the nonlinearities $f_1$ and $g_2$, it follows that $z$ is blocked in patch 2 and $z$ satisfies~\eqref{blocking}, which is then also true for $u$. The conclusion is therefore achieved.

\vskip 0.2cm
\noindent{}(iv) We finally assume that~\eqref{1.1}  admits a nonnegative classical  stationary solution $U$ such that $U(-\infty)=K_1$ and $U(+\infty)=0$ (actually, $U$ is then positive in $\R$ as a consequence of the Cauchy-Lipschitz theorem for instance, as in the second paragraph of the proof of Proposition~\ref{prop2.1}). Fix then any $L>0$. Let $u$ be the solution to the Cauchy problem~\eqref{1.1} with any nonnegative continuous and compactly supported initial datum $u_0$ such that $\text{spt}(u_0)\subset [-L,L]$. Notice that, if $u_0\le U$ in $\mathbb{R}$, the conclusion of part~(iv) of Theorem~\ref{thm2.8} immediately follows. Let us now discuss the general case. 
	
By a rescaling of space in patch 2, namely, by setting
\begin{align*}
\widetilde u(t,x)=\begin{cases}
u(t,x),~&\text{for}~t\ge 0\hbox{ and }x< 0,\\ 
u(t,\sqrt{d_2/d_1}x) &\text{for}~t\ge 0\hbox{ and }x\ge 0,
\end{cases}
\end{align*}
we see that the function $\widetilde u$ satisfies
$$\begin{aligned}
\begin{cases}
\widetilde u_t=d_1 \widetilde u_{xx}+f_1( \widetilde u),~~&t>0,\ x< 0,\\
\widetilde u_t=d_1  \widetilde u_{xx}+f_2( \widetilde u),~~&t>0,\ x> 0,\\
\widetilde u(t,0^-)= \widetilde u(t,0^+),~~&t>0,\\
\widetilde u_x(t,0^-)=\sigma\sqrt{d_1/d_2}\, \widetilde u_x(t,0^+),~~&t> 0,
\end{cases}
\end{aligned}$$
while the rescaled function $\widetilde{U}$, defined by $\widetilde{U}(x)=U(x)$ for $x<0$ and $\widetilde{U}(x)=U(\sqrt{d_2/d_1}x)$ for $x\ge0$, is a positive classical stationary solution of the above problem, satisfying $\widetilde{U}(-\infty)=K_1$ and $\widetilde{U}(+\infty)=0$. When $\Vert u_0\Vert_{L^1(\mathbb{R})}$ is small, it is seen that $\Vert \widetilde u(0,\cdot)\Vert_{L^1(\mathbb{R})}=\Vert u_0\Vert_{L^1(-\infty,0)}+\sqrt{d_1/d_2}\Vert u_0\Vert_{L^1(0,+\infty)}$ remains small, while spt($\widetilde{u}(0,\cdot))\subset[-L,\sqrt{d_1/d_2}L]$. Therefore, for the proof of part~(iv) of Theorem~\ref{thm2.8}, it is not restrictive to assume that $d_1=d_2=:d$ in~\eqref{1.1}, which we do in the sequel. 

By the assumptions~\eqref{f1-kpp}--\eqref{f2-bistable} on $f_1$ and $f_2$, and their $C^1$ smoothness, there is $K>0$ such that~$f_i(s)\le Ks$ for all $s\ge0$ and $i\in\{1,2\}$. Let $v$ be the solution of the initial value problem
\begin{align*}
\begin{cases}
v_t=d v_{xx}+Kv,~~&t>0,~x\in\mathbb{R},\\
v_0(x)=u_0(x)+u_0(-x), &x\in\mathbb{R}.
\end{cases}
\end{align*}
Since $u_0\ge 0$ satisfies $\text{spt}(u_0)\subset[-L,L]$, so does $v_0$. By uniqueness, we see that $v$ is even with respect to $x$ and smooth with respect to $x$ in $(0,+\infty)\times\R$, whence $v_x(t,0)=0$ for all $t>0$. Proposition~\ref{prop 1.3} implies that $u(t,x)\le v(t,x)$ for all $t\ge 0$ and $x\in\mathbb{R}$.
	
We now claim that $v(1,\cdot)\le U$ in $\mathbb{R}$ provided $\Vert u_0 \Vert_{L^1(\mathbb{R})}$ is small enough. Indeed, by choosing $\varep>0$ such that
$$0<\varep\le \frac{\sqrt{\pi d}}{e^K}\min_{(-\infty,2L]}U,$$
we get that
$$v(1,x)\le\frac{e^K}{\sqrt{4\pi d}}\int_{\mathbb{R}}e^{-\frac{|x-y|^2}{4d}}v_0(y) \mathrm{d} y\le \frac{e^K}{\sqrt{4\pi d}}\Vert v_0\Vert_{L^1(\mathbb{R})}= \frac{e^K}{\sqrt{\pi d}}\Vert u_0\Vert_{L^1(\mathbb{R})}\le \min_{(-\infty,2L]}U$$
for all $x\in\R$, provided $\Vert u_0 \Vert_{L^1(\mathbb{R})}\le \varep$. Furthermore, for all $x\ge 2L$, there holds
\begin{align*}
v(1,x)\le \frac{e^{K}}{\sqrt{4\pi d}}\int_{-L}^L e^{-\frac{|x-y|^2}{4d}}v_0(y) \mathrm{d}y \le \frac{e^{K-\frac{x^2}{16d}}}{\sqrt{4\pi d}} \int_{-L}^L v_0(y) \mathrm{d}y = \frac{e^K}{\sqrt{\pi d}}\Vert u_0\Vert_{L^1(\mathbb{R})} e^{-\frac{x^2}{16 d}},
\end{align*}
since $\text{spt}(v_0)\subset[-L,L]$ and since $x-y\ge x/2>0$ for all $x\ge2L$ and $-L\le y\le L$. Observe also that~$U$ is positive continuous in $\R$ and that $U(x)\sim Ae^{-\sqrt{-f_2'(0)/d_2}\,x}$ as $x\to +\infty$, for some $A>0$. Thus, 
$$v(1,x)\le \frac{e^K}{\sqrt{\pi d}}\Vert u_0\Vert_{L^1(\mathbb{R})} e^{-\frac{x^2}{16d}}\le U(x)~~\text{for all}~x\ge 2L,$$
provided $\Vert u_0 \Vert_{L^1(\mathbb{R})}\le \varep$, up to decreasing $\varep>0$ if needed. 
	
Consequently,  $v(1,\cdot)\le U$ in $\mathbb{R}$ provided $\Vert u_0 \Vert_{L^1(\mathbb{R})}$ is small enough, and then $u(1,\cdot)\le v(1,\cdot)\le U$ in $\mathbb{R}$ and $u(t,x)\le U(x)$ for all $t\ge 1$ and $x\in\mathbb{R}$ by Proposition~\ref{prop 1.3}. Hence, $u(t,x)\to 0$ as $x\to+\infty$ uniformly in $t\ge 1$. Together with Lemma~\ref{lemma1.3} stating that $u(t,x)\to 0$ as $x\to+\infty$ locally uniformly in $t\ge 0$, we conclude that $u$ is blocked in patch 2 and satisfies~\eqref{blocking}. The proof of Theorem~\ref{thm2.8} is therefore complete.
\end{proof}

%%%%%%%%%%%%%%%%%%%%%%%%%%%%%%%%%%%%%%%%%%%%%%%%%%%%

\subsection{Propagation in the bistable patch 2: proofs of Theorems~\ref{thm_propagation-1}--\ref{thm_propagation-2}}
\label{Sec 4.4}

This section is devoted to the proofs of Theorems~\ref{thm_propagation-1}--\ref{thm_propagation-2} on propagation phenomena with positive speed or speed zero in the bistable patch~$2$. The proof of the propagation with positive speed in Theorems~\ref{thm_propagation-1}--\ref{thm_propagation-2} uses some tools inspired by~\cite{FM1977} on solutions developing into two spreading fronts for the homogeneous equation~\eqref{1.2}. Here, for our patch problem~\eqref{1.1}, new difficulties arise due to the presence of the interface between the two different media, and we have to show further estimates on the local behavior of the solutions at large time.

We start with the following auxiliary lemma that gives the existence of solutions to elliptic equations in large intervals. The proof is based on variational methods, see for instance \cite[Theorem A]{BL1980} and \cite[Problem (2.25)]{GHS2020}. We omit it here.  

\begin{lemma}
\label{lemma4.1} Assume that~\eqref{f2-bistable} holds and  $\int_{0}^{K_2}f_2(s) \mathrm{d}s>0$. Then there exist $R>0$ and a function $\psi$ of class $ C^2([-R,R])$ such that
\begin{align}
\label{4.53}
\begin{cases}
\ d_2\psi''+f_2(\psi)=0~~&\text{in}~[-R,R],\cr
\ 0\le \psi<K_2~~&\text{in}~[-R,R],\cr
\ \psi(\pm R)=0,~~&\cr
\ \displaystyle\mathop{\max}_{[-R,R]}\psi=\psi(0)>\theta.
\end{cases}
\end{align}
\end{lemma}

To prove Theorem~\ref{thm_propagation-1}, we take a roundabout way to prove the following result as a first step.

\begin{theorem}
\label{thm4.2}
Assume that~\eqref{f1-kpp}--\eqref{f2-bistable} hold and that $\int_{0}^{K_2}f_2(s) \mathrm{d}s>0$. Let  $R>0$ and $\psi$ be as in Lemma~$\ref{lemma4.1}$. Let $u$ be the solution to~\eqref{1.1} with a nonnegative continuous and compactly supported initial datum $u_0\not\equiv 0$. If $u_0\ge \psi(\cdot-x_0)$ in $[x_0-R,x_0+R]$ for some $x_0\ge R$, then the conclusion~\eqref{convphi} of Theorem~$\ref{thm_propagation-1}$ holds true.
\end{theorem}	

\noindent\textit{Proof of Theorem~$\ref{thm4.2}$ $($beginning$)$.} Let  $R>0$, $\psi\in C^2([-R,R])$, $x_0\ge R$ and $u_0$ be as in the statement. Let $v$ and $w$ be, respectively, the solutions to~\eqref{1.1} with initial data $v_0=M:=\max\big(K_1,K_2,\Vert u_0\Vert_{L^\infty(\mathbb{R})}\big)$, and $w_0$ given by $w_0(x)=\psi(x-x_0)$ for $x\in[x_0-R,x_0+R]$ and $w_0(x)=0$ for $x\in\mathbb{R}\setminus[x_0-R,x_0+R]$. Then Proposition~\ref{prop 1.3} yields $0<w(t,x)\le u(t,x)\le v(t,x)\le M$ for all $t> 0$ and  $x\in\mathbb{R}$. Moreover, as in the proof of the first part of Theorem~\ref{thm2.3}, $w$ is increasing with respect to $t$ in $[0,+\infty)\times\R$, whereas $v$ is nonincreasing with respect to $t$ in $[0,+\infty)\times\R$. From the parabolic estimates of Proposition~\ref{pro1}, the functions $w(t,\cdot)$ and $v(t,\cdot)$ converge as $t\to+\infty$, locally uniformly in $\mathbb{R}$, to classical stationary solutions $p$ and $q$ of~\eqref{1.1}, respectively. Moreover, 
\begin{equation}
\label{4.54}
0\le w_0<p\le \liminf_{t\to+\infty} u(t,\cdot)\le \limsup_{t\to+\infty} u(t,\cdot)\le q\le M,~~\text{locally uniformly in}~\mathbb{R}.
\end{equation}
	
Let us now show that
\begin{equation}
\label{4.55}
p(x)\to K_2~~\text{as}~x\to+\infty.
\end{equation}
As a matter of fact, since $p>w_0$ in $\mathbb{R}$, by continuity there exists $\varrho_0>1$ such that $p>\psi(\cdot-\varrho x_0)$ in~$[\varrho x_0-R,\varrho x_0+R]$ for all $\varrho\in [1,\varrho_0]$. Define
\begin{align*}
\varrho^*=\sup\big\{\varrho>0: p>\psi(\cdot-\widetilde\varrho x_0)~\text{in}~[\widetilde\varrho x_0-R,\widetilde\varrho x_0+R]~\text{for all}~\widetilde \varrho\in [1,\varrho]\big\}\ \in[\varrho_0,+\infty].
\end{align*}
We claim that $\varrho^*=+\infty$. Indeed, otherwise, one would have $p\ge \psi(\cdot-\varrho^* x_0)$ in $[\varrho^* x_0-R,\varrho^* x_0+R]$ with equality somewhere in $(\varrho^* x_0-R,\varrho^* x_0+R)$, since $p>0$ in $\mathbb{R}$ and $\psi(\pm R)=0$. The elliptic strong maximum principle then implies that $p\equiv \psi(\cdot-\varrho^* x_0)$ in $(\varrho^* x_0-R,\varrho^* x_0+R)$ and then at $\varrho^* x_0\pm R$ by continuity, which is impossible. Thus, $\varrho^*=+\infty$ and $p>\psi(\cdot-\varrho x_0)$ in $[\varrho x_0-R,\varrho x_0+R]$ for all~$\varrho\ge 1$. In particular, this implies that
$$p(x)>\psi(0)>\theta~~\text{for all}~x\ge x_0.$$
Since $p$ is bounded and since $f_2>0$ in $(\theta,K_2)$ and $f_2<0$ in $(K_2,+\infty)$, it then follows as in the proof of the limit $V(-\infty)=K_1$ in Proposition~\ref{prop2.1}, that~\eqref{4.55} holds. Likewise,
\begin{equation}
\label{4.56}
q(x)\to K_2~~\text{as}~x\to+\infty.
\end{equation}

The rest of the proof of Theorem~\ref{thm4.2} relies on three preliminary lemmas.

\begin{lemma}
\label{lemma 3.1}
Under the assumptions of Theorem~$\ref{thm4.2}$, there exist $X_1>0$, $X_2>0$, $T_1>0$, $T_2>0$, $z_1\in\mathbb{R}$, $z_2\in\mathbb{R}$, $\mu>0$ and $\delta>0$ such that
\begin{equation}
\label{4.57}
u(t,x)\le\phi(x-c_2(t-T_1)+z_1)+\delta e^{-\delta(t-T_1)}+\delta e^{-\mu(x-X_1)}~~\text{for all}~ t\ge T_1~\text{and} ~x\ge X_1,
\end{equation}
and 
\begin{equation}
\label{4.58}
u(t,x)\ge\phi(x-c_2(t-T_2)+z_2)-\delta e^{-\delta(t-T_2)}-\delta e^{-\mu(x-X_2)}~~\text{for all}~ t\ge T_2 ~\text{and} ~x\ge X_2,
\end{equation}
where $\phi$ is the traveling front profile solving~\eqref{2.5}, with speed $c_2>0$.
\end{lemma}

\begin{proof}
We first introduce some parameters. Choose $\mu>0$ such that
\begin{equation}
\label{4.59}
0<\mu<\sqrt{\min\Big(\frac{|f_2'(0)|}{2d_2}, \frac{|f_2'(K_2)|}{2d_2}\Big)}.
\end{equation}
Then we take $\delta>0$ such that (we remember that $c_2>0$)
\begin{equation}	
\label{4.60}\left\{\baa{l}
\displaystyle 0<\delta< \min\Big(\mu c_2,\frac{|f_2'(0)|}{2}, \frac{|f_2'(K_2)|}{2}\Big),\vspace{3pt}\\
\displaystyle f_2'\le \frac{f_2'(0)}{2}~\text{in}~[-2\delta,3\delta],~~f_2'\le \frac{f_2'(K_2)}{2}~\text{in}~[K_2-3\delta,K_2+2\delta].\eaa\right.
\end{equation}
Let $C>0$ be such that
\begin{equation}
\label{phiC}
\phi\ge K_2-\frac{\delta}{2}~\text{in}~(-\infty,-C]\ \hbox{ and }\ \phi\le \delta~\text{in}~[C,+\infty).
\end{equation}
Since $\phi'$ is negative and continuous in $\mathbb{R}$, there is $\kappa>0$ such that 
\begin{equation}
\label{4.61}
\phi'\le-\kappa<0~\text{in}~[-C, C].
\end{equation}
Finally, pick $\omega>0$ so large that
\begin{equation}
\label{4.62}
\kappa\omega\ge 2\delta+\max_{[-2\delta,K_2+2\delta]}|f_2'|,
\end{equation} 
and $B>\omega$ such that
\begin{equation}
\label{4.63}
\Big(\max_{[-2\delta,K_2+2\delta]}|f_2'|+d_2\mu^2\Big)e^{-\mu B}<\Big(\max_{[-2\delta,K_2+2\delta]}|f_2'|+d_2\mu^2\Big)e^{-\mu(B-\omega)}\le \delta.
\end{equation} 

\vskip 2mm
\noindent\textit{Step 1: proof of~\eqref{4.57}}. First of all, property~\eqref{claimK2} still holds as in the proof of part~(ii) of Theorem~\ref{thm2.8}, and there is $X_1>0$ such that
\be\label{4.64}
u(t,x)\le K_2+\frac{\delta}{2}\ \hbox{ for all }t\ge X_1\hbox{ and }x\ge X_1.
\ee
Moreover, since $u(t,x)$ has a Gaussian upper bound  at each fixed $t>0$ for all $|x|$ large enough by Lemma~\ref{lemma1.3}, whereas $\phi(s)$ decays exponentially to $0$ as $s\to+\infty$ by~\eqref{2.6}, there is $A\ge B$ such that
\begin{equation} 
\label{4.65}
u(X_1,x)\le \phi (x-X_1-A-C)+\delta~~\text{for all}~x\ge X_1.
\end{equation}

For $t\ge X_1$ and $x\ge X_1$, let us define
$$\overline u(t,x)=\phi(\overline\xi(t,x))+\delta e^{-\delta(t-X_1)}+\delta e^{-\mu(x-X_1)},$$
where 
$$\overline\xi(t,x)=x-X_1-c_2(t-X_1)+\omega e^{-\delta(t-X_1)}-\omega-A-C.$$
Let us check that $\overline u(t,x)$ is a supersolution to $u_t=d_2 u_{xx}+f_2(u)$ for $t\ge X_1$ and $x\ge X_1$. At time $X_1$, one has $\overline u(X_1,x)\ge \phi(x-X_1-A-C)+\delta\ge u(X_1,x)$ for all $x\ge X_1$, by~\eqref{4.65}. Moreover, for $t\ge X_1$, since~$\overline \xi(t,X_1)\le -A-C<-C$, one gets that $\overline u(t,X_1)\ge K_2-\delta/2+\delta e^{-\delta(t-X_1)}+\delta\ge K_2+\delta/2 \ge u(t,X_1)$ by~\eqref{phiC} and~\eqref{4.64}. Therefore, it remains to check that $\mathcal{N}\overline u(t,x)\!:=\!\overline u_t(t,x)\!-\!d_2\overline u_{xx}(t,x)\!-\!f_2(\overline u(t,x))\!\ge\!0$ for all $t\ge X_1$ and $x\ge X_1$. After a straightforward computation, one derives
\begin{align*}
\mathcal{N}\overline u(t,x)=f_2(\phi(\overline \xi(t,x)))-f_2(\overline u(t,x))-\phi'(\overline \xi(t,x)))\omega\delta e^{-\delta(t-X_1)}-\delta^2 e^{-\delta(t-X_1)}-d_2\mu^2 \delta e^{-\mu(x-X_1)}.
\end{align*}
We distinguish three cases: 
\begin{itemize}
\item if $\overline \xi(t,x)\le -C$, one has $K_2-\delta/2\le \phi(\overline\xi(t,x))<K_2$ by~\eqref{phiC}, hence $K_2+2\delta>\overline u(t,x)\ge K_2-\delta/2$; it follows from~\eqref{4.60} that $f_2(\phi(\overline \xi(t,x)))-f_2(\overline u(t,x))\ge -(f_2'(K_2)/2)\big(\delta e^{-\delta(t-X_1)}+\delta e^{-\mu(x-X_1)}\big)$ and it then can be deduced from~\eqref{4.59}--\eqref{4.60} as well as the negativity of $\phi'$ and $f_2'(K_2)$ that
\begin{align*}
\mathcal{N} \overline u(t,x)&\ge -\frac{f_2'(K_2)}{2}\Big(\delta e^{-\delta(t-X_1)}+\delta e^{-\mu(x-X_1)}\Big)-\delta^2 e^{-\delta(t-X_1)}-d_2\mu^2 \delta e^{-\mu(x-X_1)}\cr
&= \Big(-\frac{f_2'(K_2)}{2}-\delta\Big)\delta e^{-\delta(t-X_1)}+\Big(-\frac{f_2'(K_2)}{2}-d_2\mu^2\Big)\delta  e^{-\mu(x-X_1)}>0;
\end{align*}
\item if $\overline\xi(t,x)\ge C$, one derives $0<\phi(\overline \xi(t,x))\le \delta$ by~\eqref{phiC}, and then $0<\overline u(t,x)\le 3\delta$; it follows from~\eqref{4.60} that $f_2(\phi(\overline \xi(t,x)))-f_2(\overline u(t,x))\ge -(f_2'(0)/2)\big(\delta e^{-\delta(t-X_1)}+\delta e^{-\mu(x-X_1)}\big)$; by virtue of~\eqref{4.59}--\eqref{4.60} and the negativity of $\phi'$ and $f_2'(0)$, there holds
\begin{align*}
\mathcal{N} \overline u(t,x)&\ge -\frac{f_2'(0)}{2}\Big(\delta e^{-\delta(t-X_1)}+\delta e^{-\mu(x-X_1)}\Big)-\delta^2 e^{-\delta(t-X_1)}-d_2\mu^2 \delta e^{-\mu(x-X_1)}\cr
&=  \Big(-\frac{f_2'(0)}{2}-\delta\Big)\delta e^{-\delta(t-X_1)}+\Big(-\frac{f_2'(0)}{2}-d_2\mu^2\Big)\delta e^{-\mu(x-X_1)}>0;
\end{align*}
\item if $-C\le\overline\xi(t,x)\le C$, it turns out that $x-X_1\ge c_2(t-X_1)-\omega e^{-\delta(t-X_1)}+\omega+A\ge c_2(t-X_1)+B$, whence $e^{-\mu(x-X_1)}\le e^{-\mu(c_2(t-X_1)+B)}$. By~\eqref{4.60} and~\eqref{4.61}--\eqref{4.63}, one infers that
\begin{align*}
\mathcal{N}\overline u(t,x)&\ge -\!\max_{[0,K_2+2\delta]}\!|f_2'|\Big(\delta e^{-\delta(t-X_1)}\!\!+\!\delta e^{-\mu(x-X_1)}\Big)\!+\!\kappa\omega\delta e^{-\delta(t-X_1)}\!\!-\!\delta^2 e^{-\delta(t-X_1)}\!\!-\!d_2\mu^2 \delta e^{-\mu(x-X_1)}\cr
&\ge \Big(\kappa\omega-\delta-\max_{[0,K_2+2\delta]}|f_2'|\Big)\delta e^{-\delta(t-X_1)}-\Big(\max_{[0,K_2+2\delta]}|f_2'|+d_2\mu^2 \Big)\delta e^{-\mu(c_2(t-X_1)+B)}\cr
&\ge \Big(\kappa\omega-2\delta-\max_{[0,K_2+2\delta]}|f_2'|\Big)\delta e^{-\delta(t-X_1)}\ge 0.
\end{align*} 
\end{itemize}

As a consequence, we have proved that $\mathcal{N}\overline u(t,x):=\overline u_t(t,x)-d_2\overline u_{xx}(t,x)-f_2(\overline u(t,x))\ge 0$ for all $t\ge X_1$ and $x\ge X_1$. The maximum principle implies that 
\begin{align*}
u(t,x)\le \overline u(t,x)=\phi\big(x-X_1-c_2(t-X_1)+\omega e^{-\delta(t-X_1)}-\omega-A-C\big)+\delta e^{-\delta(t-X_1)}+\delta e^{-\mu(x-X_1)}
\end{align*}
for all $t\ge X_1$ and $x\ge X_1$, whence~\eqref{4.57} is achieved by taking $T_1=X_1$ and $z_1=-X_1-\omega-A-C$, since $\phi$ is decreasing.

\vskip 2mm
\noindent\textit{Step 2: proof of~\eqref{4.58}}. Since $p(x)\to K_2$ as $x\to+\infty$ by~\eqref{4.55}, there is $X_2>0$ such that $|p(x)-K_2|\le\delta/2$ for all $x\ge X_2$. Moreover, since $\liminf_{t\to+\infty}u(t,\cdot)\ge p$ locally uniformly in $x\in\mathbb{R}$ by~\eqref{4.54}, one can choose $T_2>0$ so large that 
\begin{equation}\label{4.66}
u(t,x)\ge p(x)-\frac{\delta}{2}\ge K_2-\delta~~\text{for all}~t\ge T_2\text{ and for all}~ x\in [X_2, X_2+ B+ 2C].
\end{equation}
	
For $t\ge T_2$ and $x\ge X_2$, we set
$$\underline u(t,x)=\phi(\underline\xi(t,x))-\delta e^{-\delta(t-T_2)}-\delta e^{-\mu(x-X_2)},$$
in which
$$\underline\xi(t,x)=x-X_2-c_2(t-T_2)-\omega e^{-\delta(t-T_2)}+\omega-B-C.$$
We shall check that $\underline u(t,x)$ is a subsolution to $u_t=d_2 u_{xx}+f_2(u)$ for all  $t\ge T_2$ and $x\ge X_2$. At time $t=T_2$, one has $ \underline u(T_2,x)\le K_2-\delta-\delta e^{-\mu(x-X_2)}\le K_2-\delta\le u(T_2,x)$ for $X_2\le x \le X_2+B+2C$ due to~\eqref{4.66}. For $x\ge X_2+B+2C$, since $\underline\xi(T_2,x)\ge X_2+B+2C-X_2-B-C=C$, one has $\phi(\underline\xi(T_2,x))\le \delta$ by~\eqref{phiC}, hence $\underline u(T_2,x)\le\delta-\delta-\delta e^{-\mu(x-X_2)}<0<u(T_2,x)$. In conclusion, $\underline u(T_2,x)\le u(T_2,x)$ for all $x\ge X_2$. At $x=X_2$, one sees that $\underline u(t,X_2)\le K_2-\delta e^{-\delta(t-T_2)}-\delta< u(t,X_2)$ for all $t\ge T_2$, owing to~\eqref{4.66}. It thus suffices to check that $\mathcal{N}\underline u(t,x):=\underline u_t(t,x)-d_2 \underline u_{xx}(t,x)-f_2(\underline u(t,x))\le 0$ for all  $t\ge T_2$ and $x\ge X_2$. By a straightforward computation, one has
$$\mathcal{N}\underline u(t,x)=f_2(\phi(\underline \xi(t,x)))-f_2(\underline u(t,x))+\phi'(\underline\xi(t,x))\omega\delta e^{-\delta(t-T_2)}+\delta^2 e^{-\delta(t-T_2)}+ d_2\mu^2 \delta e^{-\mu(x-X_2)}$$
By analogy to Step 1, we consider three cases:
\begin{itemize}
\item	if $\underline \xi(t,x)\le -C$, then $K_2-\delta/2\le \phi(\underline \xi(t,x))<K_2$ by~\eqref{phiC} and thus $K_2>\underline u(t,x)\ge K_2-3\delta$; thanks to~\eqref{4.60}, one has $f_2(\phi(\underline \xi(t,x)))-f_2(\underline u(t,x))\le (f_2'(K_2)/2)(\delta e^{-\delta(t-T_2)}+\delta e^{-\mu(x-X_2)})$; therefore, by using~\eqref{4.59}--\eqref{4.60} as well as the negativity of $\phi'$ and $f_2'(K_2)$, it comes that
\begin{align*}
\mathcal{N}\underline u(t,x)&< \frac{f_2'(K_2)}{2}\Big(\delta e^{-\delta(t-T_2)}+\delta e^{-\mu(x-X_2)}\Big)+\delta^2 e^{-\delta(t-T_2)}+ d_2\mu^2 \delta e^{-\mu(x-X_2)}\cr
&=\Big(\frac{f_2'(K_2)}{2}+\delta\Big)\delta e^{-\delta(t-T_2)}+\Big(\frac{f_2'(K_2)}{2}+d_2\mu^2\Big)\delta e^{-\mu(x-X_2)}<0;
\end{align*}
\item if $\underline \xi(t,x)\ge C$, then $0<\phi(\underline\xi(t,x))\le \delta$ by~\eqref{phiC} and thus $-2\delta<\underline u(t,x)\le \delta$; it follows from~\eqref{4.60} that $f_2(\phi(\underline \xi(t,x)))-f_2(\underline u(t,x))\le (f_2'(0)/2)(\delta e^{-\delta(t-T_2)}+\delta e^{-\mu(x-X_2)})$; therefore, owing to~\eqref{4.59}--\eqref{4.60} as well as the negativity of $\phi'$  and $f_2'(0)$, one infers that
\begin{align*}
\mathcal{N}\underline u(t,x)&< \frac{f_2'(0)}{2}\Big(\delta e^{-\delta(t-T_2)}+\delta e^{-\mu(x-X_2)}\Big)+\delta^2 e^{-\delta(t-T_2)}+ d_2\mu^2 \delta e^{-\mu(x-X_2)}\cr
&=\Big(\frac{f_2'(0)}{2}+\delta\Big)\delta e^{-\delta(t-T_2)}+\Big(\frac{f_2'(0)}{2}+d_2\mu^2\Big)\delta e^{-\mu(x-X_2)}<0;
\end{align*}
\item if $-C\le\underline\xi(t,x)\le C$, one has  $x-X_2\ge c_2(t-T_2)+\omega e^{-\delta(t-T_2)}-\omega+B\ge c_2(t-T_2)-\omega+B$, whence $e^{-\mu(x-X_2)}\le e^{-\mu(c_2(t-T_2)+B-\omega)}$; by~\eqref{4.60} and~\eqref{4.61}--\eqref{4.63}, one deduces that
\begin{align*}
\mathcal{N}\underline u(t,x)&\le \max_{[-2\delta,K_2]}|f_2'|\Big(\delta e^{-\delta(t-T_2)}\!+\!\delta e^{-\mu(x-X_2)}\Big)\!-\!\kappa\omega\delta e^{-\delta(t-T_2)}\!+\!\delta^2 e^{-\delta(t-T_2)}\!+\!d_2\mu^2 \delta e^{-\mu(x-X_2)}\cr
&\le \Big(\max_{[-2\delta,K_2]}|f_2'|-\kappa\omega+\delta\Big)\delta e^{-\delta(t-T_2)}+\Big(\max_{[-2\delta,K_2]}|f_2'|+d_2\mu^2\Big)\delta e^{-\mu(c_2(t-T_2)+B-\omega)}\cr
&\le \Big(\max_{[-2\delta,K_2]}|f_2'|-\kappa\omega+2\delta\Big)\delta e^{-\delta(t-T_2)}\le0.
\end{align*}
\end{itemize}
	
Consequently, one has $\mathcal{N}\underline u(t,x):=\underline u_t(t,x)-d_2\underline u_{xx}(t,x)-f_2(\underline u(t,x))\le 0$ for all $t\ge T_2$ and $x\ge X_2$. The maximum principle implies that 
\begin{align*}
u(t,x)\ge \underline u(t,x)=\phi \big(x-X_2-c_2(t-T_2)-\omega e^{-\delta(t-T_2)}+\omega-B-C\big)-\delta e^{-\delta(t-T_2)}-\delta e^{-\mu(x-X_2)}
\end{align*}
for all $t\ge T_2$ and $x\ge X_2$. Therefore,~\eqref{4.58} is proved by taking $z_2=-X_2+\omega-B-C$, since $\phi$ is decreasing. The proof of Lemma~\ref{lemma 3.1} is thereby complete.
\end{proof}

More generally, we have:

\begin{lemma}
\label{lemma 3.2}
Under the assumptions of Theorem~$\ref{thm4.2}$, for any $\varep>0$, there exist $X_{1,\varep}>0$, $X_{2,\varep}>0$, $T_{1,\varep}>0$, $T_{2,\varep}>0$, $z_{1,\varep}\in\mathbb{R}$ and $z_{2,\varep}\in\mathbb{R}$ such that
\begin{equation}
\label{4.67}
u(t,x) \le \phi(x-c_2(t-T_{1,\varep})+ z_{1,\varep})+\varep e^{-\delta(t-T_{1,\varep})}+\varep e^{-\mu(x-X_{1,\varep})}~~\text{for all}~t\ge T_{1,\varep}~\text{and}~x\ge X_{1,\varep},
\end{equation}
 and 
\begin{equation}
\label{4.68}
u(t,x)\ge\phi(x-c_2(t-T_{2,\varep})+z_{2,\varep})-\varep e^{-\delta(t-T_{2,\varep})}-\varep e^{-\mu(x-X_{2,\varep})}~~\text{for all}~t\ge T_{2,\varep}~\text{and}~x\ge X_{2,\varep},
\end{equation}
with the same parameters $\delta>0$ and $\mu>0$ as in Lemma~$\ref{lemma 3.1}$.
\end{lemma}

\begin{proof}
Let $\mu>0$, $\delta>0$, $C>0$, $\kappa>0$ and $\omega>0$ be defined as in~\eqref{4.59}--\eqref{4.62} (notice that these parameters are independent of $\varep$). It is immediate to see from Lemma~\ref{lemma 3.1} that, when $\varep\ge\delta$, the conclusion of Lemma~\ref{lemma 3.2} holds true with $X_{i,\varep}=X_i$, $T_{i,\varep}=T_i$ and $z_{i,\varep}=z_i$, for $i=1,2$. It remains to discuss the case
$$0<\varep<\delta.$$
For convenience, let us introduce some further parameters. Pick $C_\varep>0$ such that
$$\phi\ge K_2-\frac{\varep}{2}~\text{in}~(-\infty,-C_\varep]\ \hbox{ and }\ \phi\le \varep~\text{in}~[C_\varep,+\infty).$$
Define
\be\label{omegaeps}
\omega_\varep:=\frac{\varep\omega}{\delta}>0.
\ee
Finally, let $B_\varep>\omega_\varep$ be large enough such that 
$$\Big(\max_{[-2\delta,K_2+2\delta]}|f_2'|+d_2\mu^2\Big)e^{-\mu B_\varep}<\Big(\max_{[-2\delta,K_2+2\delta]}|f_2'|+d_2\mu^2\Big)e^{-\mu(B_\varep-\omega_\varep)}\le \delta.$$

\vskip 2mm
\noindent\textit{Step 1: proof of~\eqref{4.67}}. By repeating the arguments used in the proof of~\eqref{4.64}--\eqref{4.65}  in  Step 1 of Lemma~\ref{lemma 3.1} and by replacing $\delta$ by $\varep$, there is $X_{1,\varep}>0$ such that $u(t,X_{1,\varep})\le K_2+\varep/2$ for all $t\ge X_{1,\varep}$ and $u(X_{1,\varep},x)\le \phi (x-X_{1,\varep}-A_\varep-C_\varep)+\varep$ for all $x\ge X_{1,\varep}$, for some $A_\varep\ge B_\varep$. Define
$$\overline u_\varep(t,x)=\phi(\overline\xi_\varep(t,x))+\varep e^{-\delta(t-X_{1,\varep})}+\varep e^{-\mu(x-X_{1,\varep})}~~\text{for}~t\ge X_{1,\varep}~\text{and} ~x\ge X_{1,\varep},$$
where 
$$\overline\xi_\varep(t,x)=x-X_{1,\varep}-c_2(t-X_{1,\varep})+\omega_\varep e^{-\delta(t-X_{1,\varep})}-\omega_\varep-A_\varep-C_\varep.$$
Following the same lines as in Step 1 of Lemma~\ref{lemma 3.1}, one has $\overline{u}_\varep(X_{1,\varep},x)\ge u(X_{1,\varep},x)$ for all $x\ge X_{1,\varep}$, $\overline{u}_\varep(t,X_{1,\varep})\ge u(t,X_{1,\varep})$ for all $t\ge X_{1,\varep}$, and it can be deduced that $\overline u_\varep(t,x)$ is a supersolution to $u_t=d_2 u_{xx}+f_2(u)$ for all $t\ge X_{1,\varep}$ and $x\ge X_{1,\varep}$, by dividing the calculations into three cases: $\overline{\xi}_\varep(t,x)\le-C$, $\overline{\xi}_\varep(t,x)\ge C$ and $\overline{\xi}_\varep(t,x)\in[-C,C]$. Therefore, the maximum principle implies that
$$u(t,x)\le \phi\big(x-X_{1,\varep}-c_2(t-X_{1,\varep})+\omega_\varep e^{-\delta(t-X_{1,\varep})}-\omega_\varep-A_\varep-C_\varep\big)+\varep e^{-\delta(t-X_{1,\varep})}+\varep e^{-\mu(x-X_{1,\varep})}$$
for all $t\ge X_{1,\varep}$ and $x\ge X_{1,\varep}$. Consequently,~\eqref{4.67} follows by choosing $z_{1,\varep}=-X_{1,\varep}-\omega_\varep-A_\varep-C_\varep$.
	
\vskip 2mm	
\noindent\textit{Step 2: proof of~\eqref{4.68}}. Using the same argument as for the proof of~\eqref{4.66} with $\delta$ replaced by $\varep$, one infers that there exist $X_{2,\varep}>0$ and $T_{2,\varep}>0$ such that
$$u(t,x)\ge  K_2-\varep~~\text{for all}~t\ge T_{2,\varep}~\text{and}~ x\in[X_{2,\varep},X_{2,\varep}+ B_\varep+ 2C_\varep].$$
Then we set
$$\underline u_\varep(t,x)=\phi(\underline\xi_\varep(t,x))-\varep e^{-\delta(t-T_{2,\varep})}-\varep e^{-\mu(x-X_{2,\varep})}~~~\text{for}~t\ge T_{2,\varep}~\text{and}~x\ge X_{2,\varep},$$
in which
$$\underline\xi_\varep(t,x)=x-X_{2,\varep}-c_2(t-T_{2,\varep})-\omega_\varep e^{-\delta(t-T_{2,\varep})}+\omega_\varep-B_\varep-C_\varep.$$
As in the proof of~\eqref{4.58}, one can show that $\underline{u}_\varep(T_{2,\varep},x)\le u(T_{2,\varep},x)$ for all $x\ge X_{2,\varep}$, that $\underline{u}_\varep(t,X_{2,\varep})\le u(t,X_{2,\varep})$ for all $t\ge T_{2,\varep}$, and that $\underline u_\varep(t,x)$ is a subsolution of $u_t=d_2 u_{xx}+f_2(u)$ for all $t\ge T_{2,\varep}$ and $x\ge X_{2,\varep}$. By the maximum principle, one derives that
$$u(t,x)\ge \phi\big( x-X_{2,\varep}-c_2(t-T_{2,\varep})-\omega_\varep e^{-\delta(t-T_{2,\varep})}+\omega_\varep-B_\varep-C_\varep\big)-\varep e^{-\delta(t-T_{2,\varep})}-\varep e^{-\mu(x-X_{2,\varep})}$$
for all $t\ge T_{2,\varep}$ and $x\ge X_{2,\varep}$. Then~\eqref{4.68} follows by taking $z_{2,\varep}=-X_{2,\varep}+\omega_\varep-B_\varep-C_\varep$, since $\phi'<0$. The proof of Lemma~\ref{lemma 3.2} is thereby complete.
\end{proof}

Based on Lemmas~\ref{lemma 3.1} and~\ref{lemma 3.2}, we now provide the stability result of the bistable traveling front in patch~2.

\begin{lemma}
\label{lemma 3.4}
Assume that~\eqref{f1-kpp}--\eqref{f2-bistable} hold and that $\int_{0}^{K_2}f_2(s)\mathrm{d}s>0$. Let $\mu>0$, $\delta>0$, $C>0$, $\kappa>0$ and $\omega>0$ be as in~\eqref{4.59}--\eqref{4.62} in the proof of Lemma~$\ref{lemma 3.1}$. Then there exists $\widetilde M>0$ such that the following holds. If there are $\varep\in(0,\delta]$, $t_0>0$, $x_0>0$ and $\xi\in\mathbb{R}$ such that
\begin{equation}
\label{4.70}
\sup_{x\ge x_0}\big|u(t_0,x)-\phi(x-c_2t_0+\xi)\big|\le \varep,
\end{equation}
\begin{equation}\label{4.70'}
K_2-\varep\le u(t,x_0)\le K_2+\frac{\varep}{2}\hbox{ for all $t\ge t_0$},
\ee
$$\phi(x_0-c_2t_0+\xi)\ge K_2-\frac{\varep}{2},$$ 
and 
\begin{equation}
\label{4.71}
\Big(\max_{[-2\delta,K_2+2\delta]}|f_2'|+d_2\mu^2\Big)e^{-\mu(c_2t_0-x_0-\omega_\varep-\xi-C)}\le\delta
\end{equation}
with $\omega_\varep=\varep\omega/\delta$, then
$$\sup_{x\ge x_0}\big|u(t,x)-\phi(x-c_2t+\xi)\big|\le \widetilde M\varep~~\text{for all}~t\ge t_0.$$
\end{lemma}

\begin{proof}
Let $\mu>0$, $\delta>0$, $C>0$, $\kappa>0$ and $\omega>0$ be as in~\eqref{4.59}--\eqref{4.62}, and let $\varep\in(0,\delta]$, $t_0>0$, $x_0>0$ and $\xi\in\R$ be as in the statement, with $\omega_\varep=\varep\omega/\delta$, as in~\eqref{omegaeps}. We claim that 
$$\overline u(t,x)=\phi(x-c_2t+\omega_\varep e^{-\delta(t-t_0)}-\omega_\varep+\xi)+\varep e^{-\delta(t-t_0)}+\varep e^{-\mu(x-x_0)}$$
and 
$$\underline u(t,x)=\phi(x-c_2t-\omega_\varep e^{-\delta(t-t_0)}+\omega_\varep+\xi)-\varep e^{-\delta(t-t_0)}-\varep e^{-\mu(x-x_0)}$$
are, respectively, a super- and a subsolution of $u_t=d_2 u_{xx}+f_2(u)$ for $t\ge t_0$ and $x\ge x_0$. We just check that $\underline u(t,x)$ is a subsolution in detail (the supersolution can be handled in a similar way). 

At time $t=t_0$, one has $\underline u(t_0,x)=\phi(x-c_2t_0+\xi)-\varep-\varep e^{-\mu(x-x_0)}\le u(t_0,x)$ for all $x\ge x_0$ thanks to~\eqref{4.70}. Moreover, $\underline u(t,x_0)=\phi(x_0-c_2 t-\omega_\varep e^{-\delta(t-t_0)}+\omega_\varep+\xi)-\varep e^{-\delta(t-t_0)}-\varep\le K_2-\varep\le u(t,x_0)$ for all $t\ge t_0$, owing to~\eqref{4.70'}. It then remains to show that $\mathcal{N}\underline u(t,x):=\underline u_t(t,x)-d_2\underline u_{xx}(t,x)-f_2(\underline u(t,x))\le 0$ for all $t\ge t_0$ and $x\ge x_0$. For convenience, we set
$$\underline \xi(t,x):=x-c_2t-\omega_\varep e^{-\delta(t-t_0)}+\omega_\varep+\xi.$$
By a straightforward computation, one has
\begin{equation*}
\mathcal{N}\underline u(t,x)=f_2(\phi(\underline \xi(t,x)))- f_2(\underline u(t,x))+\phi'(\underline\xi(t,x))\omega_\varep\delta e^{-\delta(t-t_0)}+\varep\delta e^{-\delta(t-t_0)}+ d_2\mu^2 \varep e^{-\mu(x-x_0)}.
\end{equation*}
There are three cases:
\begin{itemize}
\item if $\underline \xi(t,x)\le -C$, then $K_2-\delta/2\le \phi(\underline \xi(t,x))<K_2$ by~\eqref{phiC}, hence $K_2>\underline u(t,x)\ge K_2-\delta/2-2\varep\ge K_2-3\delta$; therefore, by using~\eqref{4.59}--\eqref{4.60} and the negativity of $\phi'$ and $f_2'(K_2)$, it follows that
\begin{align*}
\mathcal{N}\underline u(t,x)&\le  \frac{f_2'(K_2)}{2}\Big(\varep e^{-\delta(t-t_0)}+\varep e^{-\mu(x-x_0)}\Big)+\varep\delta e^{-\delta(t-t_0)}+ d_2\mu^2 \varep e^{-\mu(x-x_0)}\cr
&=\Big(\frac{f_2'(K_2)}{2}+\delta\Big)\varep e^{-\delta(t-t_0)}+\Big(\frac{f_2'(K_2)}{2}+d_2\mu^2\Big)\varep e^{-\mu(x-x_0)}\le 0;
\end{align*}
\item if $\underline \xi(t,x)\ge C$, then $0<\phi(\underline\xi(t,x))\le \delta$ by~\eqref{phiC} and thus $-2\delta\le-2\varep\le\underline u(t,x)\le \delta$; therefore, owing to~\eqref{4.59}--\eqref{4.60} as well as the negativity of $\phi'$  and $f_2'(0)$, it follows that
\begin{align*}
\mathcal{N}\underline u(t,x)&\le  \frac{f_2'(0)}{2}\Big(\varep e^{-\delta(t-t_0)}+\varep e^{-\mu(x-x_0)}\Big)+\varep\delta e^{-\delta(t-t_0)}+ d_2\mu^2 \varep e^{-\mu(x-x_0)}\cr
&=\Big(\frac{f_2'(0)}{2}+\delta\Big)\varep e^{-\delta(t-t_0)}+\Big(\frac{f_2'(0)}{2}+d_2\mu^2\Big)\varep e^{-\mu(x-x_0)}\le 0;
\end{align*}
\item if $-C\le\underline\xi(t,x)\le C$, one has $x-x_0\ge c_2(t-t_0)+c_2t_0-x_0+\omega_\varep e^{-\delta(t-t_0)}-\omega_\varep-\xi-C\ge c_2(t-t_0)+c_2t_0-x_0-\omega_\varep-\xi-C$, hence $e^{-\mu(x-x_0)}\le e^{-\mu(c_2(t-t_0)+c_2t_0-x_0-\omega_\varep-\xi-C)}$; since $\omega_\varep=\varep\omega/\delta$, one infers from~\eqref{4.60},~\eqref{4.61}--\eqref{4.62}, and~\eqref{4.71}, that
\begin{align*}
\mathcal{N}\underline u(t,x)\!&\le \max_{[-2\delta,K_2+2\delta]}|f_2'|\Big(\varep e^{-\delta(t-t_0)}\!+\!\varep e^{-\mu(x-x_0)}\Big)\!-\!\kappa\omega_\varep\delta e^{-\delta(t-t_0)}\!+\!\varep\delta e^{-\delta(t-t_0)}\!+\!d_2\mu^2 \varep e^{-\mu(x-x_0)}\cr
&\le\!\Big(\!\max_{[-2\delta,K_2\!+\!2\delta]}\!|f_2'|\!-\!\kappa\omega\!+\!\delta\!\Big)\varep e^{-\delta(t-t_0)}\!\!+\!\!\Big(\!\max_{[-2\delta,K_2\!+\!2\delta]}\!|f_2'|\!+\!d_2\mu^2\!\Big)\varep e^{-\mu(c_2(t-t_0)+c_2t_0-x_0-\omega_\varep-\xi-C)}\cr
&\le \Big(\max_{[-2\delta,K_2+2\delta]}|f_2'|-\kappa\omega+2\delta\Big)\varep e^{-\delta(t-t_0)}\le  0.
\end{align*}
\end{itemize}

Eventually, one concludes that $\mathcal{N}\underline u(t,x):=\underline u_t(t,x)-d_2\underline u_{xx}(t,x)-f_2(\underline u(t,x))\le 0$ for all $t\ge t_0$ and~$x\ge x_0$. The maximum principle implies that 
\begin{align*}
u(t,x)\ge \phi\big(x-c_2t-\omega_\varep e^{-\delta(t-t_0)}+\omega_\varep+\xi\big)-\varep e^{-\delta(t-t_0)}-\varep e^{-\mu(x-x_0)}
\end{align*}
for all $t\ge t_0$ and $x\ge x_0$. For these $t$ and $x$, since $\phi'<0$, one derives that
\begin{align*}
u(t,x)\ge \phi(x-c_2t+\omega_\varep+\xi)-2\varep\ge \phi(x-c_2t+\xi)-\omega_\varep\Vert\phi'\Vert_{L^\infty(\mathbb{R})}-2\varep.
\end{align*}
Similarly, using especially that
$$\Big(\max_{[-2\delta,K_2+2\delta]}|f_2'|+d_2\mu^2\Big)e^{-\mu(c_2t_0-x_0-\xi-C)}\le\Big(\max_{[-2\delta,K_2+2\delta]}|f_2'|+d_2\mu^2\Big)e^{-\mu(c_2t_0-x_0-\omega_\varep-\xi-C)}\le\delta$$
by~\eqref{4.71}, one can also derive that $u(t,x)\le\overline{u}(t,x)=\phi\big(x-c_2t+\omega_\varep e^{-\delta(t-t_0)}-\omega_\varep+\xi\big)+\varep e^{-\delta(t-t_0)}+\varep e^{-\mu(x-x_0)}$ for all $t\ge t_0$ and $x\ge x_0$, hence
\begin{align*}
u(t,x)\le\phi(x-c_2t-\omega_\varep+\xi)+2\varep\le \phi(x-c_2t+\xi)+\omega_\varep\Vert\phi'\Vert_{L^\infty(\mathbb{R})}+2\varep.
\end{align*}
In conclusion, one has 
$$\sup_{x\ge x_0}\big|u(t,x)-\phi(x-ct+\xi)\big|\le\omega_\varep\Vert\phi' \Vert_{L^\infty(\mathbb{R})}+2\varep=\widetilde{M}\varep~\text{for all}~t\ge t_0,$$
where $\widetilde{M}:=\omega_\varep\Vert\phi' \Vert_{L^\infty(\mathbb{R})}/\varep+2=\omega\Vert\phi' \Vert_{L^\infty(\mathbb{R})}/\delta+2$ is independent of $\varep$, $t_0$, $x_0$ and $\xi$.	The proof of Lemma~\ref{lemma 3.4} is thereby complete.
\end{proof}

Now we are in a position to complete the proof of Theorem~\ref{thm4.2}.

\begin{proof}[Proof of Theorem~$\ref{thm4.2}$ $($continued$)$]
Let $X_1>0$, $X_2>0$, $T_1>0$, $T_2>0$, $z_1\in\mathbb{R}$, $z_2\in\mathbb{R}$, $\mu>0$ and $\delta>0$ be as in Lemma~\ref{lemma 3.1}, and let also $C>0$ be as in~\eqref{phiC} in the proof of Lemma~\ref{lemma 3.1}. For $t\ge\max(T_1,T_2)$ and $x\ge\max(X_1,X_2)$, there holds
\be\label{4.72}\baa{l}
\phi(x-c_2(t-T_2)+z_2)-\delta e^{-\delta(t-T_2)}-\delta e^{-\mu(x-X_2)}\vspace{3pt}\\
\qquad\qquad\qquad\qquad\qquad\qquad\le u(t,x)\le \phi(x-c_2(t-T_1)+z_1)+\delta e^{-\delta(t-T_1)}+\delta e^{-\mu(x-X_1)}.\eaa
\ee
Consider any given sequence $(t_n)_{n\in\mathbb{N}}$ such that $t_n\to+\infty$ as $n\to +\infty$. By standard parabolic estimates, the functions
$$(t,y)\mapsto u_n(t,y):=u(t+t_n,y+c_2t_n)$$
converge as $n\to+\infty$ up to extraction of a subsequence, locally uniformly in $(t,y)\in\mathbb{R}\times\mathbb{R}$, to a classical solution $u_\infty$ of $(u_\infty)_t=d_2(u_\infty)_{yy}+f_2(u_\infty)$ in $\mathbb{R}\times\mathbb{R}$. From~\eqref{4.72} applied at $(t+t_n,y+c_2t_n)$, the passage to the limit as $n\to+\infty$ gives
\begin{align*}
\phi(y-c_2(t-T_2)+z_2)\le u_\infty(t,y)\le\phi(y-c_2(t-T_1)+z_1)~~\text{for all}~(t,y)\in\mathbb{R}\times\mathbb{R}.
\end{align*}
Then, \cite[Theorem 3.1]{BH2007} implies that there exists $\xi\in\mathbb{R}$ such that $u_\infty(t,y)=\phi(y-c_2t+\xi)$ for all $(t,y)\in\mathbb{R}\times\mathbb{R}$, whence 
\begin{equation}
\label{4.73}
u_n(t,y)\to \phi(y-c_2 t+\xi)~\text{as}~n\to+\infty,~~\text{locally uniformly in}~(t,y)\in\mathbb{R}\times\mathbb{R}.
\end{equation}

Consider now any $\varep\in(0,\delta/3]$. Let $A_\varep>0$ be such that
\be\label{phieps2}
\phi\ge K_2-\frac{\varep}{2}\hbox{ in $(-\infty,- A_\varep]\ $ and }\ \phi\le\frac\varep2\hbox{ in $[A_\varep,+\infty)$}.
\ee
Set $E_1:=\max\big(A_\varep-c_2T_1-z_1, A_\varep-\xi\big)$ and $E_2:=\min(-A_\varep-c_2T_2-z_2,-A_\varep-\xi\big)<E_1$. Then, it can be deduced from~\eqref{4.73} that 
\begin{equation}
\label{4.74}
\sup_{E_2\le y \le E_1}\big|u_n(0,y)-\phi(y+\xi)\big|\le \varep~~\text{for all}~n~\text{large enough}.
\end{equation}
Since $t_n\to+\infty$ as $n\to+\infty$,~\eqref{4.72} and~\eqref{phieps2} imply that, for all $n$ large enough,
\begin{equation}
\label{4.75}
\begin{aligned}
\begin{cases}
0<u_n(0,y)\le\varep~~~&\text{for all}~y\ge E_1,\cr
K_2-\varep\le u_n(0,y)\le K_2+\varep&\displaystyle\text{for all}~E_2-\frac{c_2}{2}t_n \le y\le E_2.
\end{cases}
\end{aligned}
\end{equation}
Furthermore, since $E_1\ge A_\varep-\xi$ and $E_2\le -A_\varep-\xi$, one has
\begin{equation}
\label{4.76}
\begin{aligned}
\begin{cases}
\displaystyle 0<\phi(y+\xi)\le\frac{\varep}{2}<\varep&\text{for all}~y\ge E_1,\vspace{3pt}\cr
\displaystyle K_2-\varep<K_2-\frac{\varep}{2}\le \phi(y+\xi)<K_2&\text{for all}~y\le E_2.
\end{cases}
\end{aligned}
\end{equation}
Then~\eqref{4.75}--\eqref{4.76} imply that, for all $n$ large enough,
$$\big|u_n(0,y)-\phi(y+\xi)\big|\le 2\varep~~\text{for all}~y\in\Big[E_2-\frac{c_2}{2}t_n,E_2\Big]\cup[E_1,+\infty).$$
Together with~\eqref{4.74} and the definition of $u_n(t,y)$, one has, for all $n$ large enough,
\begin{equation}
\label{4.77}
\big|u(t_n,x)-\phi(x-c_2t_n+\xi)\big|\le 2\varep~~\text{for all}~x\ge E_2+\frac{c_2}{2}t_n.
\end{equation}

On the other hand, one infers from Lemma~\ref{lemma 3.2} that, for all $n$ large enough, 
\begin{align}
\label{cdn1}
K_2-3\varep\le  \phi(x-c_2(t_n-T_{2,\varep})+z_{2,\varep})-\varep e^{-\delta(t_n-T_{2,\varep})}-\varep e^{-\mu(x-X_{2,\varep})}\le u(t_n,x)~~~~~~~~~~~~\cr
\le\phi(x-c_2(t_n-T_{1,\varep})+ z_{1,\varep})+\varep e^{-\delta(t_n-T_{1,\varep})}+\varep e^{-\mu(x-X_{1,\varep})}\le K_2+2\varep,
\end{align}
for all $\max(X_1,X_2,X_{1,\varep},X_{2,\varep})\le x\le E_2+c_2t_n/2$, where $X_{1,\varep}>0$, $X_{2,\varep}>0$, $T_{1,\varep}>0$, $T_{2,\varep}>0$, $z_{1,\varep}\in\mathbb{R}$ and $z_{2,\varep}\in\mathbb{R}$ were given in Lemma~\ref{lemma 3.2}. Notice also that, for all $n$ large enough,
\begin{equation}
\label{cdn2}
K_2-\varep\le\phi(x-c_2t_n+\xi)<K_2~~\text{for all}~\max(X_1,X_2,X_{1,\varep},X_{2,\varep})\le x \le E_2+\frac{c_2}{2}t_n.
\end{equation} 
From~\eqref{cdn1}--\eqref{cdn2} one deduces that, for all $n$ large enough, 
$$\big|u(t_n,x)-\phi(x-c_2t_n+\xi)\big|\le 3\varep~~\text{for all}~\max(X_1,X_2,X_{1,\varep},X_{2,\varep})\le x\le E_2+\frac{c_2}{2}t_n.$$
Together with~\eqref{4.77}, one derives that, for all $n$ large enough,
$$\big|u(t_n,x)-\phi(x-c_2t_n+\xi)\big|\le 3\varep~~\text{for all}~x\ge\max(X_1,X_2,X_{1,\varep},X_{2,\varep}).$$
Furthermore, due to~\eqref{4.54}--\eqref{4.56}, there is $x_\varep\ge\max(X_1,X_2,X_{1,\varep},X_{2,\varep})$ such that, for all $n$ large enough,
$$K_2-3\varep\le u(t,x_\varep)\le K_2+\frac{3\varep}{2}~~\text{for all}~t\ge t_n,$$
and
$$\phi(x_\varep-c_2 t_n+\xi)\ge K_2-\frac{3\varep}{2},~~\Big(\max_{[-2\delta,K_2+2\delta]}|f_2'|+d_2\mu^2\Big)e^{-\mu(c_2t_n-x_\varep-3\varep\omega/\delta-\xi-C)}\le\delta.$$
It then follows from Lemma~\ref{lemma 3.4} (applied with $t_0=t_n$, $x_0=x_\varep$ and $3\varep$ instead of $\varep$) that, for all $n$ large enough,
$$\big|u(t,x)-\phi(x-c_2t+\xi)\big|\le 3\widetilde{M}\varep~~\text{for all}~t\ge t_n~\text{and}~ x\ge x_\varep,$$
with $\widetilde{M}$ given in Lemma~\ref{lemma 3.4}. Since $\varep\in(0,\delta/3]$ was arbitrary, one finally infers that 
$$\sup_{t\ge A,\,x\ge A}|u(t,x)-\phi(x-c_2t+\xi)|\to0\ \hbox{ as }A\to+\infty.$$
This completes the proof of Theorem~\ref{thm4.2}.
\end{proof}

Finally, we are in a position to prove Theorem~\ref{thm_propagation-1}.

\begin{proof}[Proof of Theorem~$\ref{thm_propagation-1}$]
Fix any $\eta>0$ throughout the proof. For some $L\ge2$ (which will be fixed later),  let $x_L\ge L/2>0$ and denote by $u_L$ the solution of the Cauchy problem~\eqref{1.1} with initial datum 
\begin{align*}
u_L(0,\cdot)=\begin{cases}
\ \theta+\eta~&\text{in}~[x_L-L/2+1,x_L+L/2-1],\cr
\ 0 &\text{in }\R\setminus(x_L-L/2,x_L+L/2),
\end{cases}
\end{align*}	
and $u_L(0,\cdot)$ is affine in $[x_L-L/2,x_L-L/2+1]$ and in $[x_L+L/2-1,x_L+L/2]$. It follows from local parabolic estimates that, for any $A>0$,
\begin{equation}
\label{local convergence}
u_L(t,x)\to\zeta(t)~\text{as}~L\to+\infty~~\text{locally in}~t\ge 0,~\text{uniformly in}~x\in[x_L-A,x_L+A],
\end{equation}
where $\zeta$ is the solution of the ODE $\zeta'(t)=f_2(\zeta(t))$ for $t\ge0$ with initial datum $\zeta(0)=\theta+\eta$. Let~$R>0$ and $\psi\in C^2([-R,R])$ be as in Lemma~\ref{lemma4.1}, and pick $\varep\in(0,K_2-\psi(0))$. Since $\zeta(t)\to K_2$ as~$t\to+\infty$ by~\eqref{f2-bistable}, it follows that there is $T>0$ such that $\zeta(T)\ge \psi(0)+\varep$. By~\eqref{4.53} and~\eqref{local convergence}, one can then choose $L\in(\max(2R,2),+\infty)$ sufficiently large such that, for every $x_L\ge L/2$,
$$u_L(T,\cdot)>\zeta(T)-\varep\ge \psi(0)\ge\psi(\cdot-x_L)~~\text{in}~[x_L-R,x_L+R].$$

Let now $u$ be the solution to~\eqref{1.1} with a nonnegative continuous and compactly supported initial datum $u_0\not\equiv 0$ satisfying $u_0\ge \theta+\eta$ in an interval of size $L$ included in patch $2$, say $(x_L-L/2,x_L+L/2)$ for some $x_L\ge L/2$ (thus, $x_L\ge R$). The comparison principle then gives that 
$$u(T,\cdot)\ge u_{L}(T,\cdot)>\psi(0)\ge\psi(\cdot-x_L)~~\text{in}~[x_L-R,x_L+R].$$
The conclusion of Theorem~\ref{thm_propagation-1} then follows from Proposition~\ref{prop 1.3} and from Theorem~\ref{thm4.2} applied with initial datum $\psi(\cdot-x_L)$ (extended by $0$ outside $[x_L-R,x_L+R]$).
\end{proof}

We finally turn to the proof of Theorem~\ref{thm_propagation-2}. For the proof of the propagation with speed zero when $f_2$ has zero mass over $[0,K_2]$, in order to get the property~\eqref{2.7}, we especially show and use the stability of the large-time limit of solutions of some auxiliary problems.

\begin{proof}[Proof of Theorem~$\ref{thm_propagation-2}$]
Assume that~\eqref{f1-kpp}--\eqref{f2-bistable} hold with $\int_0^{K_2}f_2(s)\mathrm{d}s\ge0$, and that there is no nonnegative classical stationary solution $U$ of~\eqref{1.1} such that $U(-\infty)=K_1$ and $U(+\infty)=0$. Proposition~\ref{prop 2.6} implies in particular that $K_1>\theta$, and Proposition~\ref{prop 2.7} yields the existence and the uniqueness of a positive classical stationary solution $V$ of~\eqref{1.1} such that $V(-\infty)=K_1$ and $V(+\infty)=K_2$. Furthermore, $V$ is monotone (and even strictly monotone if $K_1\neq K_2$, from the proof of Proposition~\ref{prop 2.7}). 

Let $u$ be the solution to the Cauchy problem~\eqref{1.1} with a nonnegative continuous and compactly supported initial datum $u_0\not\equiv 0$. Proposition~\ref{prop 1.3} implies that $0<u(t,x)<M:=\max(K_1,K_2,\Vert u_0 \Vert_{L^\infty(\mathbb{R})})$ for all $t>0$ and $x\in\mathbb{R}$.

Let $v$ and $w$ be as in the beginning of the proof of Theorem~\ref{thm2.4}, namely:~1)~$v$~is the solution to the Cauchy problem~\eqref{1.1} with initial datum $v(0,\cdot)=\eta\Psi(\cdot-x_0)<u(1,\cdot)$ in $\mathbb{R}$ for $\eta>0$ small enough and for any arbitrary $x_0\le -R$, where $R>0$ and $\Psi$ are given as in~\eqref{4.1}--\eqref{4.2}; and~2)~$w$~denotes the solution to~\eqref{1.1} with initial condition $w(0,\cdot)=M$ in $\mathbb{R}$. Proposition~\ref{prop 1.3} implies that $0<v(t,x)<u(t+1,x)<w(t+1,x)\le M$ for all $t>0$ and $x\in\mathbb{R}$. Moreover, as in the proof of the first part of Theorem~\ref{thm2.3},~$v$~is increasing with respect to $t$ and $w$ is nonincreasing with respect to $t$ in $[0,+\infty)\times\R$. From the parabolic estimates of Proposition~\ref{pro1}, $v(t,\cdot)$ and $w(t,\cdot)$ converge as $t\to+\infty$, locally uniformly in $\mathbb{R}$, to classical stationary solutions $p$ and $q$  of~\eqref{1.1}, respectively. Moreover, there holds
\begin{equation}
\label{4.30}
0<p\le \liminf_{t\to+\infty} u(t,\cdot)\le \limsup_{t\to+\infty} u(t,\cdot)\le q\le M,
\end{equation} 
locally uniformly in $\mathbb{R}$. From the proofs of Proposition~\ref{prop2.1} and Theorem~\ref{thm2.4}, it is seen that
\be\label{pq-infty}
p(-\infty)=q(-\infty)=K_1.
\ee
	
In the following, we wish to show that $p\equiv q\equiv V$ in $\mathbb{R}$ and $p(+\infty)=q(+\infty)=K_2$. First of all, since $p$ and $q$ are bounded and $f_2$ satisfies~\eqref{f2-bistable}, one infers that
\be\label{limsuppq}
\limsup_{x\to+\infty}\,p(x)\le K_2\ \hbox{ and }\ \limsup_{x\to+\infty}\,q(x)\le K_2.
\ee
Let us now prove that $p$ is stable in $(0,+\infty)$ in the sense that 
\begin{equation}
\label{4.30'}
\int_0^{+\infty} d_2|\varphi'|^2-f_2'(p)\varphi^2 \ge 0,
\end{equation}
for every $\varphi\in C^1((0,+\infty))$ with compact support included in $(0,+\infty)$. In fact, we first notice that the function $v$ satisfies
$$0\le v_t=d_2(v-p)_{xx}+f_2(v)-f_2(p)~~\text{for all}~t>0~\text{and}~x>0.$$
For any given $\varphi\in C^1((0,+\infty))$ with compact support included in $(0,+\infty)$, multiplying the above equation by the nonnegative function $\varphi^2/(p-v(t,\cdot))$ at a fixed time $t>0$ and integrating over~$(0,+\infty)$ yields
\begin{align*}
0\le& \int_{0}^{+\infty}d_2(p-v(t,\cdot))_x\left(\frac{\varphi^2}{(p-v(t,\cdot))}\right)_x-\frac{f_2(v(t,\cdot))-f_2(p)}{v(t,\cdot)-p}\varphi^2\cr
=&\int_{0}^{+\infty}d_2\left(2\frac{(p-v(t,\cdot))_x\varphi\varphi' }{p-v(t,\cdot)}-\frac{|(p-v(t,\cdot))_x|^2\varphi^2}{(p-v(t,\cdot))^2}\right)-\frac{f_2(v(t,\cdot))-f_2(p)}{v(t,\cdot)-p}\varphi^2\cr
\le & \int_{0}^{+\infty}d_2|\varphi'|^2 -\frac{f_2(v(t,\cdot))-f_2(p)}{v(t,\cdot)-p}\varphi^2.
\end{align*}
Since $v(t,\cdot)\to p$ as $t\to+\infty$ locally uniformly in $\mathbb{R}$, passing to the limit $t\to+\infty$ yields~\eqref{4.30'}.

Next, we show that $p(+\infty)=K_2$. Assume first that $p$ has two critical points $a<b\in[0,+\infty)$, that is, $p'(a)=p'(b)=0$. By reflection, the function $z_1:=p(2b-\cdot)$ satisfies $d_2 z''_1+f_2(z_1)=0$ in~$[b,2b-a]$, with $z_1(b)=p(b)$ and $z_1'(b)=p'(b)=0$. The Cauchy-Lipschitz theorem implies that~$z_1=p$ in~$[b,2b-a]$. Thus, $p(2b-a)=p(a)$ and $p'(2b-a)=0$. With an immediate induction, one infers that $p$ is periodic in $[a,+\infty)$. Together with~\eqref{4.30} and~\eqref{limsuppq}, one gets $0<p\le K_2$ in $[a,+\infty)$. But the nonconstant stationary periodic solutions of~\eqref{1.1} in $(0,+\infty)$ are known to be unstable. Hence, $p$ is constant in~$[a,+\infty)$. However, since $f'_2(\theta)>0$, the constant solution $\theta$ is unstable as well. Finally, $p\equiv K_2$ in~$[a,+\infty)$ and then in $[0,+\infty)$ by the Cauchy-Lipschitz theorem, and thus $K_1=K_2$ and $p\equiv K_1=K_2$ in $\R$ (indeed, as in the proof of Proposition~\ref{prop2.1}, $p'$ is either of a constant strict sign in $(-\infty,0^-]$, or identically equal to $0$ in $(-\infty,0^-]$). Therefore, either $p$ is constant (and $p\equiv K_1=K_2$ in $\R$), or $p$ has at most one critical point in $[0,+\infty)$. The later case implies that $p$ is strictly monotone in, say, $[B,+\infty)$ for some $B>0$ large. Hence, $p(+\infty)$ exists, with $p(+\infty)\in\{0,\theta,K_2\}$. Since $p(+\infty)\neq\theta$ (because $p$ is stable) and since there is no stationary solution $U$ of~\eqref{1.1} connecting $K_1$ and $0$, it follows that
$$p(+\infty)=K_2.$$
Together with~\eqref{pq-infty}, one concludes that $p\equiv V$ in $\R$ in all cases. As a consequence,~\eqref{limsuppq} and the inequality~$p\le q$ given by~\eqref{4.30} imply that $q(+\infty)=K_2$ and then
$$q\equiv V\equiv p\ \hbox{ in $\R$}.$$
The desired conclusion~\eqref{2.7} is therefore achieved, due to~\eqref{4.30}.
	
By using~\eqref{2.7} and the fact that $V(+\infty)=K_2>\theta$, the property~(i) of Theorem~\ref{thm_propagation-2} (in the case $\int_0^{K_2}f_2(s)\mathrm{d}s>0$) can be derived from Theorem~\ref{thm_propagation-1} and a comparison argument.
	
It now remains to prove property (ii), that is, we assume now that $\int_0^{K_2}f_2(s)\mathrm{d}s=0$. Our goal is to show that $\sup_{x\ge ct}u(t,x)\to0$ as $t\to+\infty$ for every $c>0$. So let us fix $c>0$ in the sequel. For $\varep\in(0,(K_2-\theta)/2)$, let $f_{2,\varep}$ be a $C^1(\mathbb{R})$ function such that 
$$\left\{\baa{l}
f_{2,\varep}(0)=f_{2,\varep}(\theta)=f_{2,\varep}(K_2+\varep)=0,~~f_{2,\varep}'(0)<0,~~f_{2,\varep}'(K_2+\varep)<0,\vspace{3pt}\\
f_{2,\varep}=f_2~\text{in}~(-\infty,K_2-\varep),~~f_{2,\varep}>0~\text{in}~(\theta,K_2+\varep),~~f_{2,\varep}<0~\text{in}~(K_2+\varep,+\infty).\eaa\right.$$
We can also choose $f_{2,\varep}$ so that $f_{2,\varep}\ge f_2$ in $\R$, so that $f_{2,\varep}$ is decreasing in $[K_2-\varep,K_2+\varep]$, and so that the family $(\|f_{2,\varep}\|_{C^1([0,K_2+\varep])})_{0<\varep<(K_2-\theta)/2}$ is bounded. Notice that, necessarily, $\int_0^{K_2+\varep}f_2(s)\mathrm{d}s>0$. For each $\varep\in(0,(K_2-\theta)/2))$, let $\phi_\varep$ be the unique traveling front profile of $u_t=d_2 u_{xx}+f_{2,\varep}(u)$ such that 
$$d_2\phi_\varep''+c_{2,\varep}\phi_\varep'+f_{2,\varep}(\phi_\varep)=0\hbox{ in }\R,~~\phi_\varep'<0~\text{in}~\mathbb{R},~~\phi_\varep(0)=\theta,~~\phi_\varep(-\infty)=K_2+\varep,~~\phi_\varep(+\infty)=0,$$
with speed $c_{2,\varep}>0$. It is standard to see that $\phi_\varep\to\phi$ in $C^2_{loc}(\R)$ and $c_{2,\varep}\to 0$ as $\varep\to0$. We can then fix $\varep\in(0,(K_2-\theta)/2)$ small enough such that $0<c_{2,\varep}<c$. As in the proof of~\eqref{claimK2}--\eqref{utxX} in Theorem~\ref{thm2.8}, there is then $X>0$ such that $u(t,x)\le K_2+\varep/2$ for all $t\ge X$ and $x\ge X$. Since~$u(t,x)$ has a Gaussian upper bound as $x\to+\infty$ at each fixed $t>0$ by Lemma~\ref{lemma1.3}, whereas~$\phi_\varep(s)$ has an exponential decay (similar to~\eqref{2.6}) as $s\to+\infty$, it follows that there is $A>0$ such that $u(X,x)\le \phi_\varep (x-c_{2,\varep}X-A)$ for all $x\ge X$, and $u(t,X)\le\phi_\varep (X-c_{2,\varep}t-A)$ for all $t\ge X$ (we also here use the fact $c_{2,\varep}>0$ and $\phi_\varep(-\infty)=K_2+\varep$). Since $f_{2,\varep}\ge f_2$ in $\R$, the maximum principle implies that~$0<u(t,x)\le\phi_\varep (x-c_{2,\varep}t-A)$ for all $t\ge X$ and $x\ge X$, hence $\sup_{x\ge ct}u(t,x)\to0$ as $t\to+\infty$, since $c_{2,\varep}<c$ and $\phi_\varep(+\infty)=0$. This completes the proof of Theorem~\ref{thm_propagation-2}.
\end{proof}
	
%%%%%%%%%%%%%%%%%%%%%%%%%%%%%%%%%%%%%%%%%%%%%%%%%%%%
%%%%%%%%%%%%%%%%%%%%%%%%%%%%%%%%%%%%%%%%%%%%%%%%%%%%

\section{The bistable-bistable case}
\label{Sec 5}

In this section, we only outline the proofs in the bistable-bistable case~\eqref{fi-bistable}, since most of the arguments are similar to those of the preceding section. However, the main novelty is the extinction result in the case of reaction terms $f_i$ having negative masses over $[0,K_i]$. We start with this case.

%%%%%%%%%%%%%%%%%%%%%%%%%%%%%%%%%%%%%%%%%%%%%%%%%%%%

\subsubsection*{Extinction in the case of reactions with negative masses}

\begin{proof}[Proof of Theorem~$\ref{thextinction}$]
We here assume that $\int_0^{K_i}f_i(s)\mathrm{d}s<0$ for $i=1,2$. Let $u$ be the solution to the Cauchy problem~\eqref{1.1} with a nonnegative continuous and compactly supported initial datum~$u_0\not\equiv 0$. Set $M:=\max\big(K_1,K_2,\|u_0\|_{L^\infty(\R)}\big)+1$. As in the proof of part~(i) of Theorem~\ref{thm2.8}, for each~$i\in\{1,2\}$, since $f_i$ satisfies~\eqref{f2-bistable} with $\int_0^{K_i}f_i(s)\mathrm{d}s<0$, there is a $C^1(\R)$ function $\overline{f}_i$ such that $\overline{f}_i\ge f_i$ in $\R$, $\overline{f}_i(0)=\overline{f}_i(\theta_i)=\overline{f}_i(M)=0$, $\overline{f}_i'(0)<0$, $\overline{f}_i'(M)<0$, $\overline{f}_i>0$ in $(-\infty,0)\cup(\theta_i,M)$, $\overline{f}_i<0$ in~$(0,\theta_i)\cup(M,+\infty)$, and~$\int_0^M\overline{f}_i(s)\mathrm{d}s<0$ (it is even possible to choose $\overline{f}_i$ so that $\overline{f}_i=f_i$ in $(-\infty,K_i-\delta]$ for some small~$\delta>0$). There is then a decreasing front profile $\overline{\phi}_i$ solving~\eqref{2.5} with $\overline{f}_i$ and $M$ instead of~$f_2$ and~$K_2$, and with negative speed $\overline{c}_i$ instead of $c_2$. Since $\overline{\phi}_i(-\infty)=M>\max(K_1,K_2,\|u_0\|_{L^\infty(\R)})$ and~$u_0$ is compactly supported, one can then choose two positive real numbers $A_1$ and $A_2$ so large that
$$u_0(x)\le\overline{\phi}_1(-x-A_1)\hbox{ for all $x\le0$},\ \ u_0(x)\le\overline{\phi}_2(x-A_2)\hbox{ for all $x\ge0$},\ \hbox{ and }\ \overline{\phi}_1(-A_1)=\overline{\phi}_2(-A_2).$$
Let $\overline{u}$ be the solution to~\eqref{1.1} with reactions $\overline{f}_i$ instead of $f_i$ and with initial datum $\overline{u}_0$ given by
$$\overline{u}_0(x):=\left\{\baa{ll}\overline{\phi}_1(-x-A_1) & \hbox{if }x\le0,\vspace{3pt}\\
\overline{\phi}_2(x-A_2) & \hbox{if }x>0.\eaa\right.$$
The comparison principle of Proposition~\ref{prop 1.3} implies that
\be\label{inequoveru}
0<u(t,x)\le\overline{u}(t,x)\ \hbox{ for all $t>0$ and $x\in\R$}.
\ee
Furthermore, since $\overline{c}_i<0$ and $\overline{\phi}_i'<0$ in $\R$ for each $i=1,2$, it follows that the time-independent function $v$ equal to $v(t,x):=\overline{u}_0(x)$ in $[0,+\infty)\times\R$ is a supersolution of~\eqref{1.1} (with reactions $\overline{f}_i$ instead of $f_i$) in the sense of Definition~\ref{def2}. Then, as in the proof of the first part of Theorem~\ref{thm2.3}, one has
\be\label{inequ0}
\overline{u}(t,x)\le\overline{u}_0(x)\hbox{ for all $(t,x)\in[0,+\infty)\times\R$}
\ee
and $\overline{u}$ is nonincreasing with respect to $t$ in $[0,+\infty)\times\R$. Together with the parabolic estimates of Proposition~\ref{pro1}, there is then a nonnegative classical bounded stationary solution $p$ of~\eqref{1.1} (with reactions $\overline{f}_i$ instead of $f_i$) such that $\overline{u}(t,x)\to p(x)$ as $t\to+\infty$ locally uniformly in $x\in\R$. The inequalities~\eqref{inequoveru}--\eqref{inequ0} also imply that $p(\pm\infty)=0$ and that $\overline{u}(t,\cdot)\to p$ as $t\to+\infty$ uniformly in $\R$. 

Let us finally show that $p\equiv 0$ in $\R$, which will lead to the desired extinction result. Assume by contradiction that $p\not\equiv0$. Since $p$ is nonnegative continuous and converges to $0$ at $\pm\infty$, there is then $x_0\in\R$ such that $p(x_0)=\max_\R p>0$. If $x_0>0$, then the integration of the equation $d_2p''+\overline{f}_2(p)=0$ against $p'$ over the interval $[x_0,+\infty)$ yields $\int_0^{p(x_0)}\overline{f}_2(s)\mathrm{d}s=0$, which is impossible from the choice of~$\overline{f}_2$. The case $x_0<0$ is similarly ruled out. Therefore, $x_0=0$ and the interface conditions at $0$ then imply that $p'(0^\pm)=0$ and the integration of the equation $d_2p''+\overline{f}_2(p)=0$ against $p'$ over the interval~$[0,+\infty)$ leads to the same impossibility. As a conclusion $p\equiv 0$ in $\R$ and the inequalities~\eqref{inequoveru} and the uniform convergence of $\overline{u}(t,\cdot)$ to $p\equiv0$ as $t\to+\infty$ imply that $\|u(t,\cdot)\|_{L^\infty(\R)}\to0$ as $t\to+\infty$. The proof of Theorem~\ref{thextinction} is thereby complete.
\end{proof}

%%%%%%%%%%%%%%%%%%%%%%%%%%%%%%%%%%%%%%%%%%%%%%%%%%%%

\subsubsection*{Stationary solutions connecting $K_1$ to $0$, and $K_1$ to $K_2$}

\begin{proof}[Proof of Proposition~$\ref{prop 2.5'}$]
(i) Suppose that $U$ is a positive classical stationary solution  of~\eqref{1.1} such that $U(-\infty)=K_1$ and $U(+\infty)=0$. From the strong maximum principle and the Hopf lemma (or the Cauchy-Lipschitz theorem), it follows that $U>0$ in $\mathbb{R}$. Multiplying $d_1U''+f_1(U)=0$ by $U'$ and integrating by parts over $(-\infty,x]$ for any $x\le0$ yields
\begin{equation}
\label{*}
\frac{d_1}{2}(U'(x^-))^2=\int_{U(x)}^{K_1} f_1(s)\mathrm{d}s\ge 0.
\end{equation}
Then, we claim that 
\begin{equation}
\label{claim}
\begin{aligned}
\begin{cases}
\text{ either}&\!\!U>K_1\hbox{ in $(-\infty,0]$ and }U'>0~\text{in}~(-\infty,0^-],\cr
\text{ or}&\!\!U<K_1\hbox{ in $(-\infty,0]$ and }U'<0~\text{in}~(-\infty,0^-],\cr
\text{ or}&\!\!U\equiv K_1~\text{in}~(-\infty,0].\cr
\end{cases}
\end{aligned}
\end{equation}
To prove~\eqref{claim}, we first show that either $U-K_1$ has a strict constant sign in $(-\infty,0]$ or $U\equiv K_1$ in $(-\infty,0]$. Indeed, if there is $x_0\le0$ such that $U(x_0)=K_1$, then~\eqref{*} implies $U'(x_0^-)=0$ and the Cauchy-Lipschitz theorem then yields $U\equiv K_1$ in $(-\infty,0]$. Assume now  that $U-K_1$ has a strict constant sign in $(-\infty,0]$. Then in~\eqref{*} the integral is positive from the assumption on $f_1$, hence $U'$ has a strict constant sign in $(-\infty,0^-]$. Our claim~\eqref{claim} follows, since $U(-\infty)=K_1$. The argument in patch 2 is exactly the same as the one in the proof of Proposition~\ref{prop 2.5}, thus completing the proof of part~(i) of Proposition~\ref{prop 2.5'}.

(ii) The proof of (ii) is an adaptation of the proof of Proposition~\ref{prop 2.6}, with the fact that the function $\nu\mapsto\int_\nu^{K_1}f_1(s)\mathrm{d}s$ is continuous in $[0,K_1]$, vanishes at $K_1$, is positive in $[0,K_1)$, due to the positivity of $\int_0^{K_1}f_1(s)\mathrm{d}s$, here. The rest of the proof is identical to that of Proposition~\ref{prop 2.6}.

(iii) The proof of (iii) follows the same lines as the proof of Proposition~\ref{prop 2.7}.
\end{proof}

%%%%%%%%%%%%%%%%%%%%%%%%%%%%%%%%%%%%%%%%%%%%%%%%%%%%

\subsubsection*{Blocking phenomena}

\begin{proof}[Proof of Theorem~$\ref{thm2.8'}$] 
It is exactly as that of part~(i) of Theorem~\ref{thm2.8} if $\int_0^{K_2}f_2(s)\mathrm{d}s<0$. In the case $\int_0^{K_2}f_2(s)\mathrm{d}s=0$ and $K_1<K_2$, let $w$ be the solution of~\eqref{1.1} with initial datum $w_0=M:=\max(K_2,\|u_0\|_{L^\infty(\R)})$. As in the proof of the first part of Theorem~\ref{thm2.3}, the function $w$ is nonincreasing with respect to $t$ in $[0,+\infty)\times\R$ and there is a nonnegative classical stationary solution $q$ of~\eqref{1.1} such that $w(t,\cdot)\to q$ as $t\to+\infty$ locally uniformly in $\R$, with $0\le q\le M$ in $\R$. Since $f_{1,2}<0$ in $(K_{1,2},+\infty)$, one gets that $\limsup_{x\to-\infty}q(x)\le K_1$ and $\limsup_{x\to+\infty}q(x)\le K_2$ and, since $f_1<0$ in $(K_1,+\infty)\supset(K_2,+\infty)$, one easily infers that $\sup_\R q\le K_2$ and even $q<K_2$ in $\R$. Next, properties~\eqref{4.41}--\eqref{claimK2} and~\eqref{4.43} hold and the rest of the proof is identical to that of part~(ii) of Theorem~\ref{thm2.8}. The other cases can be handled as in the proofs of parts~(iii) and~(iv) of Theorem~\ref{thm2.8} (since those proofs did not use the specific KPP assumption in patch~$1$).
\end{proof}

%%%%%%%%%%%%%%%%%%%%%%%%%%%%%%%%%%%%%%%%%%%%%%%%%%%%

\subsubsection*{Propagation with positive or zero speed}

Parallel to Lemma~\ref{lemma4.1} and Theorem~\ref{thm4.2}, which lead to the proof of Theorem~\ref{thm_propagation-1}, we have the following results.

\begin{lemma}
\label{lemma4.1'}
Assume that~\eqref{fi-bistable} holds and there is $i\in\{1,2\}$ such that  $\int_{0}^{K_i}f_i(s) \mathrm{d}s>0$. Then there exist $R_i>0$ and a function $\psi_i$ of class $ C^2([-R_i,R_i])$ such that
\begin{align*}
\begin{cases}
\ d_i\psi_i''+f_2(\psi_i)=0~~&\text{in}~[-R_i,R_i],\cr
\ 0\le \psi_i<K_i~~&\text{in}~[-R_i,R_i],\cr
\ \psi_i(\pm R_i)=0,~~&\cr
\ \displaystyle\max_{[-R_i,R_i]}\psi_i=\psi_i(0)>\theta_i.
\end{cases}
\end{align*}
\end{lemma}

\begin{theorem}
\label{thm4.2'}
Assume that~\eqref{fi-bistable} holds and there is $i\in\{1,2\}$ such that  $\int_{0}^{K_i}f_i(s) \mathrm{d}s>0$.  Let  $R_i>0$ and $\psi_i\in C^2([-R_i,R_i])$ be as in Lemma~$\ref{lemma4.1'}$. Let $u$ be the solution to~\eqref{1.1} with a nonnegative continuous and compactly supported initial datum $u_0\not\equiv 0$. If $u_0\ge \psi_i(\cdot-x_i)$ in $[x_i-R_i,x_i+R_i]$ for some $|x_i|\ge R_i$ with the interval $(x_i-R_i,x_i+R_i)$ included in patch~$i$, then the conclusion of part~(i) of Theorem~$\ref{thm_propagation-1'}$ holds true if $i=2$ and property~\eqref{convphi1} of part~(ii) of Theorem~$\ref{thm_propagation-1'}$ holds true if $i=1$.
\end{theorem}	

\begin{proof}[Proof of Theorem~$\ref{thm_propagation-1'}$]
Since part~(i) of Theorem~\ref{thm_propagation-1'} and property~\eqref{convphi1} of part~(ii) follow from Theorem~\ref{thm4.2'}, while the last two statements of part~(ii) of Theorem~\ref{thm_propagation-1'} (about the propagation in patch~$2$) follow exactly as in the proofs of Theorems~\ref{thm_propagation-1}--\ref{thm_propagation-2} once~\eqref{2.7'} is known, it remains to show the large-time behavior~\eqref{2.7'} of $u$ in part~(ii).

So, let us assume that $\int_0^{K_1}f_1(s)\mathrm{d}s>0$ and $\int_0^{K_2}f_2(s)\mathrm{d}s\ge0$, and let $u$ be the solution of~\eqref{1.1} with a nonnegative continuous and compactly supported initial datum $u_0\not\equiv 0$. By Proposition~\ref{prop 1.3}, one has $0<u(t,x)<M:=(K_1,K_2,\Vert u_0\Vert_{L^\infty(\mathbb{R})})$ for all $t>0$ and $x\in\mathbb{R}$. The arguments of the proof of Theorem~\ref{thm_propagation-1} also imply that, if $\eta>0$ is fixed and if $u_0\ge\theta_1+\eta$ in a large enough interval in patch~$1$, then $u(T,\cdot)>\psi_1(0)\ge\psi_1(\cdot-x_1)$ in $[x_1-R_1,x_1+R_1]$ for some $T>0$ and  $x_1\le -R_1$, where~$R_1>0$ and~$\psi_1\in C^2([-R_1,R_1])$ are given as in Lemma~\ref{lemma4.1'}. Let now $v$ and $w$ be, respectively, the solutions of~\eqref{1.1} with initial data $v(0,\cdot)=\psi_1(\cdot-x_1)$ (extended by $0$ in $\R\setminus[x_1-R_1,x_1+R_1]$) and $w(0,\cdot)=M$ in $\mathbb{R}$. Proposition~\ref{prop 1.3} implies that
$$0<v(t,x)<u(t+T,x)<w(t+T,x)\le M~~\text{for all}~ t>0~\text{and}~x\in\mathbb{R}.$$
Moreover, as in the proof of the first part of Theorem~\ref{thm2.3}, $v$ is increasing with respect to $t$ and $w$ is nonincreasing with respect to $t$, in $[0,+\infty)\times\R$. By the Schauder estimates of Proposition~\ref{pro1}, $v(t,\cdot)$ and $w(t,\cdot)$ converge as $t\to+\infty$, locally uniformly in $\mathbb{R}$, to classical stationary solutions~$p$ and~$q$ of~\eqref{1.1}, respectively. Therefore, 
\begin{equation}
\label{5.4}
0<p\le \liminf_{t\to+\infty} u(t,\cdot)\le \limsup_{t\to+\infty}u(t,\cdot)\le q\le M~~\text{locally uniformly in}~\mathbb{R}.
\end{equation}
Theorem~\ref{thm4.2'} then implies that $v$ propagates in patch~$1$ with speed $c_1$ and~\eqref{convphi1} holds with $v$ for some $\xi'_1\in\R$ instead of $\xi_1$. Hence, $p(-\infty)=K_1$. Therefore, $\liminf_{x\to-\infty}q(x)\ge K_1$ and, since $f_1<0$ in $(K_1,+\infty)$, one infers as before that $\limsup_{x\to-\infty}q(x)\le K_1$, hence $q(-\infty)=K_1$. On the other hand, one can show as in the proof of Theorem~\ref{thm_propagation-2} that $p$ is stable in $(0,+\infty)$, whence $p(+\infty)=K_2$ from the bistable profile of~$f_2$ and the nonexistence of a stationary solution $U$ of~\eqref{1.1} such that $U(-\infty)=K_1$ and $U(+\infty)=0$. As a consequence,~\eqref{2.7'} follows from~\eqref{5.4}, with $p$ being a positive classical stationary solution of~\eqref{1.1} such that $p(-\infty)=K_1$ and $p(+\infty)=K_2$. Moreover, $\liminf_{x\to+\infty}q(x)\ge K_2$ and, as before, $q(+\infty)=K_2$. Finally, if $K_2\ge K_1\ge\theta_2$ or $K_1\ge K_2\ge\theta_1$, then $p\equiv q\equiv V$, where $V$ is the unique positive classical stationary solution of~\eqref{1.1} satisfying $V(-\infty)=K_1$ and $V(+\infty)=K_2$, given by part~(iii) of Proposition~\ref{prop 2.5'}. The proof of Theorem~\ref{thm_propagation-1'} is thereby complete.
\end{proof}

%%%%%%%%%%%%%%%%%%%%%%%%%%%%%%%%%%%%%%%%%%%%%%%%%%%%
%%%%%%%%%%%%%%%%%%%%%%%%%%%%%%%%%%%%%%%%%%%%%%%%%%%%

\appendix

\section{Appendix}

In this appendix, we show Gaussian upper bounds for the solutions to the Cauchy problem~\eqref{1.1} with nonnegative continuous compactly supported initial data. We recall that the $C^1(\R)$ functions $f_i$ satisfy~\eqref{hypf}, and we call $K$ any nonnegative real number such that
\be\label{hypK}
f_1(s)\le Ks\hbox{ and }f_2(s)\le Ks\hbox{ for all }s\ge0.
\ee

\begin{lemma}
\label{lemma1.3}
Let $L_1>0$, $L_2>0$, and let $u$ be the solution to the Cauchy problem~\eqref{1.1} with a nonnegative continuous and compactly supported initial datum $u_0$ satisfying {\rm{spt}}$(u_0)\subset[-L_1,L_2]$. Then, with $M:=\max(K_1,K_2,\|u_0\|_{L^\infty(\R)})$ and $K\ge0$ as in~\eqref{hypK}, there holds, for all $t>0$,
$$u(t,x)\le M e^{K t}  e^{-\frac{(x+L_1)^2}{4d_1t}}\hbox{ for all }x\le-L_1,\ \hbox{ and }\ u(t,x)\le M e^{K t} e^{-\frac{(x-L_2)^2}{4d_2t}}\hbox{ for all }x\ge L_2.$$
\end{lemma}

\begin{proof}
It is based on the comparison between $u$ and the solution of certain initial--boundary value problem defined in a half-line. We only do the proof of the first inequality, as the second one can be handled analogously. By Proposition~\ref{prop 1.3}, one has $0<u(t,x)<M$ for all $t>0$ and $x\in\mathbb{R}$. Let $v$ be the solution of the following initial-boundary value problem
\begin{align}
\label{1.4}
\begin{cases}
\ v_t=d_1 v_{xx},& t>0,\ x\le0,\\
\ v(0,x)=\chi_{[-L_1, 0]}(x),&x\le0,\\
\ v(t,0)=1,&t> 0,
\end{cases}
\end{align}
where $\chi$ denotes the indicator function. With~\eqref{hypK}, the maximum principle yields
$$u(t,x)\le Me^{Kt}v(t,x)\ \hbox{ for all $t\ge 0$ and $x\le 0$}.$$
To solve~\eqref{1.4}, we define $w(t,x):=v(t,x)-1$ for $t\ge 0$ and $x\le 0$. Then $w$ satisfies
$$\begin{aligned}
\begin{cases}
\ w_t=d_1 w_{xx},&t>0,\ x\le0,\\
\ w(0,x)=-\chi_{(-\infty,-L_1)}(x),&x\le0,\\
\ w(t,0)=0,&t>0.	 
\end{cases}
\end{aligned}$$
For each $t\ge0$, $w(t,\cdot)$ is then the restriction to $(-\infty,0]$ of the function $W(t,\cdot)$, where $W$ solves the heat equation $W_t=d_1W_{xx}$ in $(0,+\infty)\times\R$ with initial condition $W(0,\cdot)$ given as the odd extension of~$w(0,\cdot)$, that is $W(0,x)=w(0,x)=-\chi_{(-\infty,-L_1)}(x)$ if $x\le0$ and $W(0,x)=-w(0,-x)=\chi_{(L_1,+\infty)}(x)$ if $x>0$. Denote by $S_1$ the standard heat kernel, namely $S_1(t,x)=(4\pi d_1t)^{-1/2}e^{-x^2/(4d_1t)}$ for $t>0$ and $x\in\R$. Then, for every $t>0$ and $x\in\mathbb{R}$,
$$W(t,x)=\int_{-\infty}^{+\infty}S_1(t,x-y)W(0,y)\mathrm{d}y=\int_{-\infty}^0( S_1(t,x-y)- S_1(t,x+y))w(0,y)\mathrm{d}y.$$
It follows that, for every $t>0$ and $x\le0$,
$$w(t,x)=W(t,x)=-\frac{1}{\sqrt{4\pi d_1t}}\int_{-\infty}^{-L_1}\left(e^{-\frac{(x-y)^2}{4d_1t}}-e^{-\frac{(x+y)^2}{4d_1t}}\right)\mathrm{d}y,$$
hence
\begin{align*}
v(t,x)=1+w(t,x)&=1-\frac{1}{\sqrt{4\pi d_1t}}\int_{-\infty}^{-L_1}\left(e^{-\frac{(x-y)^2}{4d_1t}}-e^{-\frac{(x+y)^2}{4d_1t}}\right)\mathrm{d}y\cr
&=1-\frac{1}{\sqrt{\pi}}\int_{-\infty}^{\frac{-x-L_1}{\sqrt{4d_1t}}}e^{-z^2}\mathrm{d}z+\frac{1}{\sqrt{\pi}}\int_{-\infty}^{\frac{x-L_1}{\sqrt{4d_1t}}}e^{-z^2}\mathrm{d}z\le\frac{2}{\sqrt{\pi}}\int_{\frac{-x-L_1}{\sqrt{4d_1t}}}^{+\infty}e^{-z^2}\mathrm{d}z.
\end{align*}
Finally, for every $t>0$ and $x\le-L_1$, there holds
$$u(t,x)\le M e^{K t}v(t,x)\le\frac{2Me^{Kt}}{\sqrt{\pi}}\int_{\frac{-x-L_1}{\sqrt{4d_1t}}}^{+\infty}e^{-z^2}\mathrm{d}z\le Me^{K t-\frac{(x+L_1)^2}{4d_1t}},$$
since $(2/\sqrt{\pi})\int_A^{+\infty}e^{-z^2}\mathrm{d}z\le e^{-A^2}$ for all $A\ge0$. This completes the proof.
\end{proof}

%%%%%%%%%%%%%%%%%%%%%%%%%%%%%%%%%%%%%%%%%%%%%%%%%%%%
%%%%%%%%%%%%%%%%%%%%%%%%%%%%%%%%%%%%%%%%%%%%%%%%%%%%

\end{document}